\documentclass[12pt]{amsart}
\usepackage[all,cmtip]{xy}
\usepackage{graphicx, overpic}
\usepackage[normalem]{ulem}
\usepackage{amssymb}
\usepackage{setspace}
\usepackage{color}
\usepackage{hyperref}
\usepackage{fancyhdr}

\definecolor{darkblue}{cmyk}{1,0, 0,.7}
\definecolor{orange}{cmyk}{0, 0.61, 0.87, 0}

\newcommand\note[1]{\mbox{}\marginpar{\scriptsize\raggedright\hspace{0pt}\color{darkblue}#1}}
\renewcommand\note[1]{}

\newcommand{\Label}[1]{\label{#1}\note{#1}}
\renewcommand{\Label}[1]{\label{#1}}

\let\oldmarginpar\marginpar
\renewcommand\marginpar[1]{\-\oldmarginpar[\raggedleft\footnotesize #1]%
{\raggedright\footnotesize #1}}

\newtheorem{lemma}{Lemma}[section]
\newtheorem{proposition}[lemma]{Proposition}
\newtheorem{theorem}[lemma]{Theorem}
\newtheorem{claim}[lemma]{Claim}

\newtheorem{question}[lemma]{Question}
\newtheorem{remark}[lemma]{Remark}

\newtheorem{corollary}[lemma]{Corollary}
\newtheorem{definition}[lemma]{Definition}

\newcommand{\Proof}{{\noindent \it Proof. }}
\newcommand{\proofof}[1]{{\noindent \it Proof of #1. }}

\newcommand{\cm}[1]{{\small{\sf #1}}}

\newcommand{\Qed}[1]{\nopagebreak[4]{\tiny \hfill
\fbox{\ref{#1}} \linebreak }\pagebreak[2]}

\renewcommand{\Im}{\operatorname{Im}}

\newcommand{\length}{\operatorname{length}}

\newcommand{\dist}{\operatorname{dist}}
\newcommand{\Conv}{\operatorname{Conv}}

\newcommand{\Core}{\operatorname{Core}}
\newcommand{\Area}{\operatorname{Area}}
\newcommand{\Gr}{\operatorname{Gr}}

\newcommand{\td}{\tilde}
\newcommand{\til}{\tilde}

\newcommand{\on}{\operatorname}

\newcommand{\pt}{\partial}
\newcommand{\bd}{\partial}

\newcommand{\cc}{\circ}
\newcommand{\iv}{^{-1}}

\newcommand{\h}{\mathbb{H}}
\newcommand{\s}{\mathbb{S}}
\newcommand{\C}{\mathbb{C}}
\newcommand{\R}{\mathbb{R}}
\newcommand{\n}{\mathbb{N}}
\newcommand{\N}{\mathbb{N}}
\newcommand{\Z}{\mathbb{Z}}

\newcommand{\sm}{\setminus}
\newcommand{\minus}{\setminus}
\newcommand{\st}{\subset}

\newcommand{\sub}{\subset}
\newcommand{\dt}{\ldots}
\newcommand{\cn}{\colon}
\newcommand{\col}{\colon}

\newcommand{\ify}{\infty}
\newcommand{\If}{\infty}
\newcommand{\In}{\infty}

\newcommand{\po}{\pi_1}
\newcommand{\PSL}{\operatorname{PSL(2, \mathbb{C})}}
\newcommand{\psl}{\operatorname{PSL(2, \mathbb{C})}}
\newcommand{\pslr}{\operatorname{PSL(2, \mathbb{R})}}
\newcommand{\Hol}{\operatorname{Hol}}

\newcommand{\rs}{{\hat{\mathbb{C}}}}

\renewcommand{\AA}{\mathcal{A}}

\newcommand{\RR}{\mathcal{R}}
\newcommand{\LL}{\mathcal{L}}
\newcommand{\BB}{\mathcal{B}}
\newcommand{\FF}{\mathcal{F}}
\newcommand{\MM}{\mathcal{M}}

\newcommand{\mA}{\mathcal{A}}

\newcommand{\mP}{\mathcal{P}}

\newcommand{\mT}{\mathcal{T}}
\newcommand{\mF}{\mathcal{F}}

\newcommand{\mL}{\mathcal{L}}

\newcommand{\rt}{\mathtt{r}}

\newcommand{\dl}{\delta}
\newcommand{\del}{\delta}

\newcommand{\Del}{\Delta}
\newcommand{\kp}{\kappa}
\newcommand{\kap}{\kappa}
\newcommand{\Gm}{\Gamma}
\newcommand{\gm}{\gamma}
\newcommand{\gam}{\gamma}
\newcommand{\Ld}{\Lambda}
\newcommand{\ld}{\lambda}
\newcommand{\lam}{\lambda}

\newcommand{\ap}{\alpha}

\newcommand{\ep}{\epsilon}

\newcommand{\Ct}{\tilde{C}}

\newcommand{\St}{\tilde{S}}

\newcommand{\Lt}{\tilde{L}}
\newcommand{\lt}{\tilde{l}}

\newcommand{\Pt}{\tilde{P}}
\newcommand{\Mt}{\tilde{M}}

\newcommand{\Tt}{\tilde{T}}
\newcommand{\ldt}{\tilde{\lambda}}
\newcommand{\phit}{\tilde{\phi}}

\newcommand{\nut}{\tilde{\nu}}

\newcommand{\mut}{\tilde{\mu}}
\newcommand{\kpt}{\tilde{\kappa}}

\newcommand{\nin}{\noindent}

\newcommand{\ul}{\underline}

\date{\today}

\usepackage{mathrsfs}

\usepackage[mathscr]{euscript}

\newcommand{\PP}{\mathscr{P}}
\newcommand{\PML}{\mathscr{PML}}
\newcommand{\ML}{\mathscr{ML}}
\newcommand{\gl}{\mathscr{GL}}
\newcommand{\GL}{\mathscr{GL}}

\newcommand{\TT}{\mathscr{T}}

\newcommand{\infi}{\infty}
\newcommand{\bdr}{\partial}

\begin{document}

 \title[\today]{2$\pi$-grafting and complex projective structures, I}

\author{Shinpei Baba} 

\address{MSRI \& Universit\"at Heidelberg 
}

\email{shinpei@mathi.uni-heidelberg.de}
\date{\today}

\maketitle


\begin{abstract}
Let $S$ be a closed oriented surface of genus at least two.
Gallo, Kapovich, and Marden  \cite{Gallo-Kapovich-Marden}  asked if   $2\pi$-grafting produces all projective structures on $S$ with arbitrarily fixed holonomy (Grafting Conjecture). 
In this paper,  we show that the conjecture  holds true  ``locally'' in the space $\GL$ of geodesic laminations on $S$ via  a natural projection of projective structures on $S$ into $\GL$   in  Thurston coordinates. 
In the sequel paper (\cite{Baba_10-2}),  using this local solution, we prove the conjecture for generic holonomy. 
\end{abstract}

\setcounter{tocdepth}{1}  
\tableofcontents

\section{Introduction}\Label{S:intro}
Let $F$ be a connected and oriented surface. 
A {\it (complex) projective} {\it structure} on $F$ is a $(\rs, \psl)$-structure, where $\rs$ is the Riemann sphere.
Equivalently, a projective structure on $F$ is a pair $(f, \rho)$ of 
\begin{itemize}
\item  an immersion $f\cn \td{F} \to \rs$ ({\it developing map)}, where $\td{F}$ is the universal cover of $F$, and 
\item a homomorphism $\rho\cn \po(F) \to \psl$ ({\it holonomy representation}) 
\end{itemize}
such that $f$ is $\rho$-equivariant, i.e. $f \cdot \gm = \rho(\gm) \cdot f$ for all $\gm \in \po(F)$; see for example \cite[\S 3.4]{Thurston-97}.  
The pair $(f, \rho)$ is defined up to an element $\ap$ of $\PSL$, i.e. $(f, \rho) \sim (\ap f,\, \ap\,\rho\,\ap^{-1})$.
(For general background about projective structures,  see \cite{Dumas-08, Kapovich-01}.)
Throughout this paper let $S$ be a closed oriented surface of genus $g > 1$. 

We aim to characterize  the set $\PP_\rho$ of all projective structures with fixed holonomy $\rho:\pi_1(S) \to \psl$. 
This basic question is  discussed in \cite[p 274]{Hubbard-81} \cite[\S 7.1, Problem 2]{Kapovich-95} \cite[Problem 12.1.1.]{ Gallo-Kapovich-Marden}  \cite[\S 1]{Dumas-08}; see also  \cite[\S 1.10]{Goldman-thesis}.
This aims for understanding of the geometry behind general representations $\pi_1(S) \to \PSL$, which are not necessarily discrete.

 Let $\PP$ be  the space of all (marked) projective structures on $S$, and let $\chi$ be the $\PSL$-character variety of $S$, i.e. the space of homomorphisms $\rho\cn \pi_1(S) \to \psl$, roughly, up to conjugation by an element of $\psl$\,; see \cite{Kapovich-01}. 
Then there is an obvious forgetful map $\Hol\cn \PP \to \chi$, called the {\it holonomy map}.
Clearly $\PP_\rho$ is a fiber of $\Hol$.
In addition  $\PP$ is diffeomorphic  to $\R^{2(6g-6)}$ and moreover it enjoys a natural complex structure (see \cite{Dumas-08}).
Then $\Hol$ is a local biholomorphism  (\cite{Hejhal-75, Hubbard-81, Earle-81}), and thus $\PP_\rho$ is a discrete subset of $\PP$.

There is a surgery operation of a projective structure, called  {\it (2$\pi$-)grafting},  that produces a different projective structure, preserving its holonomy representation (\S \ref{grafting}):
It inserts a cylinder along an appropriate essential loop ({\it admissible loop}) on a projective surface. 
Given $n \in \Z_{> 0}$, we can graft a projective surface $n$ times along the same admissible loop; we denote it by assigning  weight $2\pi n$ to the loop.

If there are disjoint admissible loops on a projective surface, we can simultaneously graft along all loops and obtain a new projective structure with the same holonomy. 
Similarly we use a multiloop with $2\pi$-multiple weights ({\it weighted multiloop}) to specify a general grafting along a multiloop.

For some special discrete representations $\pi_1(S) \to \PSL$,  graftings are known to produce all projective structures  in $\PP_\rho$ \,(\cite{Goldman-87, Ito-08, Baba12}; see also \S \ref{qf}).
Then, more generally, Gallo, Kapovich and Marden \cite[Problem 12.1.2]{Gallo-Kapovich-Marden} asked the following question.
 \begin{question}[Grafting Conjecture]\label{graftingconj}
Given two projective structures sharing (arbitrary) holonomy $\rho\cn \po(S) \to \psl$, is there a sequence of graftings and ungraftings that transforms one to the other?   
\end{question}

Holonomy representations are quite general. 
In fact, a homomorphism $\rho\cn \po(S) \to \psl$ is the holonomy representation of some projective structure (on $S$) if and only if $\rho$ satisfies:
\begin{itemize}
\item[(i)] $\Im (\rho)$ is  a nonelementary subgroup of $\psl$ and 
\item[(ii)] $\rho$ lifts to $\td{\rho}\cn \po(S) \to \on{SL}(2, \C)$
\end{itemize}
(\cite{Gallo-Kapovich-Marden}).
Recall that the character variety $\chi$ consists of two connected components (\cite{Goldman-88t}), one of which consists of the representations with the lifting property in (ii).
Thus  $\Hol$ is almost onto this component. 
In particular, holonomy representations are {\it not} necessarily discrete or faithful, and many holonomy representations have dense images in $\psl$ (c.f. \cite[Lemma 2.1]{Minsky13}).
Moreover if  $\rho\cn \po(S) \to \psl$ satisfies (i) and (ii) then $\PP_\rho$ contains infinitely many distinct projective structures, which can be constructed by grafting (implicitly in \cite{Gallo-Kapovich-Marden}; see also \cite{Baba-10}).

\subsection{Projective structures with fuchsian holonomy}\Label{qf}
We recall the characterization of $\PP_\rho$ when $\rho\cn \po(S) \to \psl$ is  a discrete and faithful representation into $\pslr$, called a {\it fuchsian (holonomy) representation}.
Then  $\Im (\rho) =: \Gm$ is called a {\it fuchsian group}, and its domain of discontinuity is a union of two disjoint round disks in $\rs$. 
Then, by quotienting out the domain by $\Gm$, we obtain two distinct projective structures with fuchsian holonomy $\rho$  ({\it uniformizable projective structures}), which  have different orientations. 
 Let  $C_0$ denote the one of our fixed orientation. 
 Then $C_0$ is isomorphic to the hyperbolic surface $\h^2/\Im(\rho)$ as projective surfaces and every essential loop on $C_0$ is admissible.

\begin{theorem}[Goldman \cite{Goldman-87}; also \cite{Kapovich-88}]\Label{10-30-1}
If $C \in \PP_\rho$, then  
$C$ is obtained by grafting $C_0$ along a weighted multiloop $M$,
$$C = \Gr_M(C_0).$$
\end{theorem}
In Theorem  \ref{10-30-1},  $M$ is unique up to an isotopy, and the same assertion holds  moreover  for {\it quasifuchsian representations} (although the proof is easily reduced to a fuchsian case  by a quasiconformal map).
Let $\ML$ be the space of measured laminations on $S$.
Then $\PP_\rho$ is naturally identified with the discrete subset $\ML_\N$ of  $\ML$ that consists of weighted multiloops.

Let $C$ and $C'$ be the projective structures sharing the fuchsian holonomy $\rho$.
Then $C = Gr_M(C_0)$ and $C' = \Gr_{M'}(C_0)$ for unique weighted multiloops $M$ and $M'$ on $C_0$ by Theorem \ref{10-30-1}.
Then it follows from Ito's work (\cite[Theorem 1.3]{Ito07}) that:
\begin{theorem}\Label{ito}
 $C$ and $C'$ can be transformed to a common projective structure in $\PP_\rho$ by grafting $C$ along $M'$ and $C'$ along $M$,
 $$\Gr_{M'}(C) = \Gr_{M}(C')$$
(see also \cite{Calsamiglia4DeroinFrancaviglia14}).
\end{theorem}

\subsection{Thurston coordinates}(More details in \S \ref{thurston}.)\Label{sThurston}
In a geometric manner, Thurston gave a natural homeomorphism 
$$\PP \to \TT \times \ML,$$
where $\TT$ is the space of marked hyperbolic structures $S$ (Teichm\"uller space). 
Thus, given  $C \in \PP$ , we denote {\it its Thurston coordinates} by $C \cong (\tau, L)$ with  $\tau \in \TT$ and $L \in \ML$. 

For example, suppose that a projective structure  $C  \cong (\tau, L)$ has fuchsian holonomy $\rho\cn \po(S) \to \pslr$.
Then $\tau$ is   $\h^2 /\Im(\rho)$, and $L$ is  the weighted multiloop $M$ given by Theorem \ref{10-30-1},  so that $C = \Gr_L(\tau)$; see \cite{Goldman-87}. 

Without restricting holonomy, let $C \cong (\tau, L) \in \PP$.  
Suppose that there is a weighted multiloop $M$ on $\tau$ such that each loop of  $M$ does not intersect $L$ transversally (for example, $M$ is supported on some closed leaves of $L$).
Note that, by the transversality, the addition $M + L$ is a well-defined measured lamination. 
Then there is a corresponding circular admissible (weighted) multiloop $\MM$ on $C$ that is equal to $M$ in $\ML$, such that $\Gr_\MM(C) \cong (\tau, L + M)$.
Then we may simply write $\Gr_M(C)$ to denote $\Gr_\MM(C)$, abusing notation.

Given a projective structure $C  \cong (\tau, L)$ in Thurston coordinates,  
there is a natural marking preserving map $\kap\col C \to \tau$, called the {collapsing map}, which is a diffeomorphism except on the inverseimage of the closed leaves of $L$. 
Nonetheless, there is a natural measured lamination $\LL$ on $\tau$ such that leaves of $\LL$ are circular and $\LL$ descends to $L$ by $\kap$. 
Thus $\LL$ is a natural representative of $L$ on $C$.  
(See \S \ref{thurston}.)

\subsection{Traintracks and measured laminations}
(See \S \ref{S:traintracks} for details.)
A {\it (fat) traintrack} $T$ on  a surface is a subsurface that is a union of rectangles $(\textit{branches})$ with disjoint interiors glued along vertical edges in a certain manner. 
We say that a traintrack {\it carries} a measured lamination if  $T$ contains a measured lamination in a natural manner without ``backtracks".
 Then  the transversal measure assigns each branch of $T$ a non-negative real number (\textit{weight}). 

When a single traintrack $T$ carries two measured laminations $L$ and $M$, the {\it sum} $L + M$  is defined to be a measured lamination carried by $T$ that is given by adding the weights of $L$ and $M$  branch-wise. 
Similarly, when appropriate, we obtain the {\it difference} $L - M$ that is a measured lamination carried by $T$ represented by the differences on the weights of $L$ and $M$.

\subsubsection{Existence of admissible traintracks.}

Given an admissible loop $\ell$ on a projective surface $C$,  an isotopy of $\ell$ on $C$ does not necessarily keep $\ell$ admissible. 
On the other hand, given a loop $\ell$ on $C$ whose holonomy is loxodromic, in general it is hard to tell if $\ell$ can be isotoped an admissible loop. 
Thus, we introduce {\it admissible traintracks} in order to specify admissible loops (Definition \ref{10-2-1}), still allowing a ``uniform amount'' of  isotopies.  \note{english still}
If a traintrack $\mathcal{T}$ on $C$ is admissible, then it is foliated by circular arcs parallel to vertical edges (\S \ref{10-5-1}). 
Indeed, if a loop $\ell$ is carried by $\mathcal{T}$ and it is transversal to this {\it circular foliation},  then $\ell$ is admissible (Lemma \ref{admissible}).
Note that we do {\it not} need to isotope $\ell$ to make it admissible, and in addition  $\ell$ stays admissible under an isotopy through such transversal loops carried by $\mathcal{T}$.
 Moreover, such an isotopy preserves the resulting projective structure $\Gr_\ell(C)$. 
In fact

\vspace{3mm}
\nin{\bf Corollary \ref{AdmissibleTraintrack}.}
Given $C \cong (\tau, L) \in \PP$ and a geodesic lamination $\nu$ on $\tau$ containing the underlying lamination $|L|$, there is an admissible traintrack on $C$ fully carrying $\nu$.
\vspace{3mm}

Suppose that there is an admissible traintrack $\mathcal{T}$ on $C \in \PP$, and let $\ell$ be a loop carried by $\mathcal{T}$ transversal to the circular foliation so that $\ell$ is admissible.  
Then the grafting of $C$ along $\ell$ restricts to the grafting of $\mathcal{T}$ along $\ell$. 
Then $\Gr_\ell(\mathcal{T})$ is naturally an admissible traintrack on $\Gr_\ell(C)$. 
In this paper, in oder to compare different projective structures sharing holonomy, we construct admissible traintracks on them that are related by grafting. 

In \cite{Baba-10}, given  arbitrary $C \cong (\tau, L) \in \PP$, the author constructed an admissible loop $\ell$ on $C$, so that $\ell$  is a good approximation of  a minimal sublamination of $L$ in the Chabauty topology on the space $\GL$ of geodesic laminations  (using the ``Closing Lemma" in \cite[I.4.2.15]{CanaryEpsteinGreen84}). 
Corollary \ref{AdmissibleTraintrack} a provides more general way of constructing admissible loops.
In particular, if a loop on $\tau$ is close to $L$ in  $\ML$ or if it intersects leaves of  $L$ only at uniformly small angles  (see Definition \ref{9-12-12no1}),  then there is an admissible loop on $C$ in the same isotopy class.

\subsection{Local characterization of $\PP_\rho$ in $\PML$}\label{ml}

Let $\PML$ be the space of projective measured laminations on $S$. 
Note that $\PML$ is homeomorphic to the sphere of dimension $6g-7$.
We show that, if two projective structures in $\PP_\rho$ are close in $\PML$ in Thurston coordinates, then they are related by a single grafting along a weighted multiloop:

\vspace{3mm}
\nin{\bf Theorem A.} (see Theorem \ref{ThmA}.)

{\it
Let $C \cong (\tau, L)$ be a projective structure on $S$ with (arbitrary) holonomy $\rho$. 
Then for every $\ep >0$, there is a neighborhood $U$ of the projective class $[L]$ in $\PML$ such that   
 if another projective structure $C' \cong (\tau', L')$  with holonomy $\rho$ satisfies $[L'] \in U$, then we have either
\begin{enumerate}
\item[(i)] $[L] = [L']$,  and $L - L'$ is a weighted multiloop $M'$ such that $C = Gr_{M'}(C'),$
 or 
\item[(ii)]  there are  an admissible traintrack $\TT$ on $C$  carrying both $L$ and $L'$ and a weighted multiloop $M$ carried by $\TT$, such that $M$ is $\ep$-close to $L' - L$  on $\TT$ (\S \ref{S:traintracks}) and  
 $$Gr_M(C) = C'.$$
\end{enumerate}
\vspace{3mm}
}

\begin{remark}

In {\rm (ii)},  by ``$\ep$-close'',  we mean, roughly, that $M$ is a good approximation of $L - L'$ for sufficiently small  $\ep > 0$ (see \S \ref{S:traintracks}). 

Case {\rm (i)} may happen only when $L$ and $L'$ are both multiloops.
Since generic measured laminations are not  multiloops,  for  generic $C \in \PP$ , only {\rm (ii)} occurs.

  In the case of {\rm (i)}, the weight of $L$ is larger than that of $L'$ on every branch; whereas, in the case of {\rm (ii)}, the weight is smaller on every branch.  \note{omission.}
This dichotomy is due to the discreteness of $\PP_\rho$ in $\PP$ and the smallness of $U$. 

\end{remark}

Let  $\rho\cn \po(S) \to \psl$ be a fuchsian representation.
Let $C$ be  $\h^2 / \Im \rho =: \tau$, the uniformizable structure  with holonomy $\rho$ as in \S \ref{qf}. 
Then $C \cong (\tau, \emptyset)$, where $\emptyset$ is the empty lamination.
Then, for every $C'  \cong (\tau, M)$ with  holonomy $\rho$, we have $C' = \Gr_M(C)$ by  Theorem \ref{10-30-1}.
Thus theorem A  (i) holds true with $U = \PML$ and $M = M - \emptyset$. 
Hence Theorem A generalizes Theorem \ref{10-30-1}.

\subsection{Local characterization of $\PP_\rho$ in $\GL$}\label{120612}

Let $\GL$ be the set of geodesic laminations on $S$. 
Naturally $\ML$ projects to $\GL$ by forgetting transversal measures. 
Theorem B  below gives a local characterization $\PP_\rho$ in  $\GL$ analogous to Theorem A.
Note that geodesic laminations are more essential to pleated surfaces than measured laminations are. 
Indeed, Theorem B is essentially stronger than Theorem A, and in particular it generalizes not only Theorem \ref{10-30-1} but also Theorem \ref{ito}.

 \begin{definition}\Label{9-12-12no1}
If $\ell$ and $\ell'$  are simple geodesics on a hyperbolic surface $\tau$ intersecting at a point $p$, we let $\angle_p (\ell, \ell')$ denote the angle, taking a value in  $[0, \pi/2]$,  between $\ell$ and $\ell'$ at $p$.
 Let $\ld$ and $\ld'$ be (possibly measured) geodesic laminations on $\tau$.
 Then the \underline{angle} between $\ld$ and $\ld'$ is 
 $$sup_p\, \angle_p(\ell_p, \ell_p')~ \in [0, \pi/2],$$ taken over all points $p$ in the intersection of  $\ld$ and $\ld'$, where $\ell_p$ and $\ell_p'$ are the leaves of $\ld$ and $\ld'$, respectively, intersecting at $p$.
We denote this angle by $\angle_\tau(\ld, \ld')$ or simply $\angle(\lam, \lam')$. 
\end{definition}

In Definition \ref{9-12-12no1}, if $\lam$ or $\lam'$ is a geodesic lamination on a different hyperbolic surface homeomorphic to $\tau$, then we always take its geodesic representative on $\tau$ in order to measure the angle $\angle_\tau(\ld, \ld')$.

Let $C \cong (\tau, L)$ be a projective structure on $S$ with holonomy \linebreak[4]$\rho\cn \po(S) \to \psl$.
Then  $L$ determines whether $C$ is obtained by grafting another projective structure in an obvious way as described in \S \ref{sThurston}.
If $L$ contains a closed leaf $\ell$ and its weight  $w(\ell)$ is equal to or more than $2\pi$, then change the weight of $\ell$ by subtracting a  $2\pi$-multiple so that $0 \leq w(\ell) < 2\pi$.
By applying this weight reduction to all closed leaves of $L$, we obtain a measured lamination $L_0$ such that every closed leaf of  $L_0$ has weight less than $2\pi$. 
Let $M= L - L_0$, so that  $M$ is a weighted multiloop. 

Throughout this paper, let $C_0$ denote the projective structure given by $(\tau, L_0)$ in Thurston coordinates, with $L_0$ as above.
Then $C = \Gr_M(C_0)$, and the holonomy of  $C_0$ is also $\rho$.
Note that, since generic $L \in \ML$ contains  no closed loops,  for generic $C \in \PP$ in Thurston coordinates, we have $C = C_0$ and $M = \emptyset$.

Then analogues of Theorem \ref{10-30-1} and Theorem \ref{ito} hold for projective structures in $\PP_\rho$ whose geodesic laminations are close to $L$ in terms of the angle defined above:

\vspace{3mm}
{\it
\nin{\bf Theorem B.} (See Theorem \ref{8-28-12}.)
For every $\ep > 0$ and every projective structure $C \cong (\tau, L)$ on $S$ with holonomy $\rho$, there exists $\dl > 0$, such that, 
if  another projective structure $C' \cong (\tau', L')$ with holonomy $\rho$ satisfies $\angle_\tau ( L, L' ) < \del$, then there are admissible traintracks on $\TT_0$, $\TT$, and $\TT'$ on $C_0$, $C$, $C'$, respectively,  that are isotopic on $S$ and carry both $L$ and $L'$ (thus also $L_0$), so that

\begin{itemize}
\item[(i)] $C'$ is obtained by grafting $C_0$ along a weighted multiloop $M'$ carried by $\TT_0$, such that  $M'$ is $\ep$-close to the measured lamination given by $L' - L_0$ on $\TT_0$, and
 \item[(ii)]  
if weighted multiloops $\hat{M}$ and $\hat{M}'$ are carried by  $\TT$ and $\TT'$, respectively, and $\hat{M} + M = \hat{M}' + M'$ on the traintracks, then we have
  $$\Gr_{\hat{M}} (C) = \Gr_{\hat{M}'}(C').$$
\end{itemize} 
}

(See Figure 1.)
In (ii), there are infinitely many choices for $\hat{M}$ and $\hat{M'}$ satisfying the equality; in particular   we can let $\hat{M} = M'$ and $\hat{M}' = M$.

\vspace{3mm}
\begin{figure}[h]
\centerline{
\xymatrix{ 
& \hat{C} &\\
C \ar[ru]^{\Gr_{\hat{M}}}& & \ar[lu]_{Gr_{\hat{M}'}} C'\\
&\ar[lu]^{\Gr_{M}} C_0 \ar[ru]^{\Gr_{M'}}&
}
}\caption{}
\end{figure}\Label{fGrafting}

Moreover, given a compact subset $K$ in the moduli space of (unmarked) hyperbolic structures on $S$,  
there is $\del > 0$ such that Theorem B holds for all projective structures $C \cong (\tau, L)$ on $S$ with its unmarked $\tau$ in $K$;  see Theorem \ref{062213}.
In addition Theorem B (i)  implies that, if $L$ contains no closed leaves of weight at least $2\pi$, then $C'$ is obtained by grafting $C$ along a multiloop. 

In the case that $\rho\cn \pi_1(S) \to \pslr$ is fuchsian, Theorem B (i) corresponds to Theorem  \ref{10-30-1}: 
Let  $C \cong (\tau, \emptyset)$ denote the unique uniformizable structure in $\PP_\rho$. 
Then indeed $\angle_\tau (\emptyset, L' ) = 0 $ for every $L' \in \ML$.
For generic $C'  \in \PP_\rho$, we have $C_0 = C$,  and $C'$ is obtained by grafting $C$ along the weighted multiloop $L' - \emptyset$. 
Moreover Theorem B (ii) implies Theorem \ref{ito} (see Theorem \ref{cor:ito}).

\subsection{Pleated surfaces}
Consider (abstract) pleated surfaces equivariant via a  fixed representation $\rho \col \pi_1(S) \to \PSL$.
We show some continuity of $\rho$-equivariant pleated surfaces in terms of their pleating laminations.
(In comparison, \cite{Bonahon98, KeenSeries95} yield continuity of the pleated surfaces bounding the convex cores of  hyperbolic three-manifolds, when associated discrete representations vary.) 

The correspondence between a projective structure $C = (f, \rho)$ and its Thurston coordinates $(\tau, L)$ is given via a pleated surface $\beta\cn \h^2 \to \h^3$ (\S \ref{6.18.13}, \S\ref{thurston}).
In particular $\beta$ is equivariant under the holonomy representation of $C$, and it ``realizes'' $(\tau, |L|)$, where $|L|$ is the underlying geodesic lamination of $L$.
A pair of  $\tau \in \TT$ and $\lam \in \GL$ is {\it realized} by a pleated surface $\beta\col \h^2 \to \h^3$, if $\beta$ bend $\h^2$ (in $\h^3$) exactly along the {\it total lift}  $\til{\lam}$ of $\lam$ to $\h^2$ (\S \ref{sConvention}) and it is totally geodesic elsewhere (this is slightly stronger than a usual notion of a realization); see  \S  \ref{sBending}.

Then,  in fact, the assumptions in Theorem A and B can be interpreted in term of pleated surfaces, by the following theorem. 

\vspace{3mm}
\nin{\bf Theorem C} (See Theorem \ref{12-19} for the precise statement.)
{\it Let $\rho\cn \po(S) \to \psl$ be a homomorphism.
Suppose that there is a $\rho$-equivariant pleated surface $\beta\cn \h^2 \to \h^3$ realizing $(\tau, \lambda) \in \TT \times \GL$.

For every $\ep > 0$, there exists $\dl > 0$ such that, if there is another $\rho$-equivariant pleated surface $\beta'\cn \h^2 \to \h^3$ realizing $(\sigma, \nu)  \in \TT \times \GL$ with $\angle_{\tau}(\lambda, \nu) < \dl$,
then $\sigma$ is $\ep$-close to $\tau$ in $\TT$ and $\beta'$ and $\beta$ are $\ep$-close.}
\vspace{3mm}

If we apply Theorem C to pleated surfaces associated with projective structures, Theorem A and B may seem natural.
For example, in Theorem B, if $\del > 0$ is sufficiently  small, then $\tau'$ must be very close to  $\tau$ in $\TT$ by Theorem C.
Thus the differences of the projective structures $C, C_0, C'$ are captured, mostly, by the differences of the measured laminations in Thurston coordinates.

\vspace{3mm} 
\nin{\it Acknowledgements.}
I am grateful to William Goldman,  Misha Kapovich, and Subhojoy Gupta.
I thank Ken Bromberg for telling me about Ito's paper. 
I also thank the referee of the paper. 

I acknowledge support from the GEAR Network (U.S. National Science Foundation grants DMS 1107452, 1107263, 1107367), and the European Research Council (ERC-Consolidator grant no. 614733).

\subsection{Outline of the proofs}
{\it Theorem C. } (\S \ref{stability}.)
The closeness of $\tau$ and $\tau'$ and of $\beta$ and $\beta'$ is given by constructing a marking-preserving homeomorphism $\phi\col \tau \to \sigma$ that is almost an isometry (more precisely, it is a rough isometry with small distortion). 
 We here outline the construction of $\phi$.   
First we define a homeomorphism $\psi$ from the geodesic representative  $\nu_\tau$ of $\nu$ on $\tau$  onto the geodesic lamination $\nu$ on $\sigma$ so that it induces bilipschitz maps of small distortion between corresponding leaves: 
If $\ell$ and $\ell'$ are corresponding leaves of the total lifts $\til{\nu}_\tau$ and $\til{\nu}$ to $\h^2$, then $\beta'|\ell'$ is  a geodesic in $\h^3$ and, 
since $\angle_\tau(\lambda, \nu)$ is sufficiently small, $\beta|\ell$ is a  bilipschitz embedding of small distortion (by Proposition \ref{BilipschitzGeodesics}). 
Then $\beta(\ell)$ and $\beta'(\ell')$ are Hausdorff-close and they share their endpoints on $\rs$.
The nearest point projection of  $\beta(\ell)$ onto the geodesic $\beta'(\ell')$ yields a desired bilipschitz map $\ell \to \ell'$.  
By applying this to all corresponding leaves of $\til{\nu}_\tau$ and $\til{\nu}$, we obtain $\psi\cn \nu_\tau \to \nu$.

Extend $\nu_\tau$ and $\nu$ to maximal laminations on $\tau$ and $\sigma$ that are isomorphic (as topological laminations), so that they divides $\tau$ and $\sigma$ into ideal triangles. 
Then we extend $\psi\cn \nu_\tau \to \nu$ to $\phi\col \tau \to \sigma$, so that  $\psi$ is  a quasiisometry of small distortion  between all corresponding complementary ideal triangles. 
It turns out that $\psi$ almost preserves  horocycle laminations of the triangulation, and therefore $\psi$ has almost no ``shearing'' between nearby ideal triangles. 
Thus $\psi\col \tau \to \sigma$ is  almost an isometry.

{\it Theorem A.}(The proof of Theorem B is similar)
If $\angle_\tau(L, L') > 0 $ is sufficiently small, then we can apply theorem C to the pleated surfaces corresponding to $C$ and $C'$; then we can naturally identify  $\tau$ and $\tau'$  by an almost isometric homeomorphism preserving the marking. 

There is a nearly-straight traintrack $T$  on $\tau$ carrying $L$ (Lemma \ref{8-2}).
By the almost isometry between $\tau$ and $\tau'$,  we can regard $T$ also as a nearly-straight traintrack on $\tau'$ carrying $L'$.
Then $T$ yields corresponding traintracks $\mathcal{T}$ on $C$ and $\mathcal{T}'$ on $C'$ that descend to $T$ via the  collapsing maps $C \to \tau$ and $C' \to \tau'$.
Moreover $\mathcal{T}$ and $\mathcal{T}'$ decompose $C$ and $C'$, isomorphic, into subsurfaces in a compatible manner  (\S \ref{sec:decomposition}). 
In particular, $C \minus  \mathcal{T}$ and $C' \minus \mathcal{T}'$ are isomorphic (as projective surfaces). 
In addition, if $\BB$ and $\BB'$ are corresponding branches of $\mathcal{T}$ and $\mathcal{T}'$, respectively, then they are related by a grafting along a multiarc.
Then the multiloop for each grafting in Theorem A  is obtained as the union of such multiarcs (see \S \ref{s:charactrization}, c.f. \cite{Baba12}). 

Moreover the number of the arc times $2\pi$ is approximately the difference of the weights of $L$ and $L'$ on the branch of $T$ corresponding to $\BB$ and $\BB'$.
Accordingly $M$ is approximately the difference of $L$ and $L'$  on the traintrack $T$. 

 \section{Conventions and notation}\Label{sConvention}
\note{order}
 \begin{itemize} 
 \item   Given a projective structure $C \cong (\tau, L)$ on $S$ in Thurston coordinates,  we let  $C_0$  denote the ``reduced'' structure $(\tau, L_0)$ constructed in \S \ref{120612} .
\item By a {\it component}, we mean a {\it connected} component.   
\item  For a geodesic  metric space $X$  and points $x, y \in X$,  we denote the geodesic segment connecting $x$ to $y$   by $[x, y]$. Then $\length[x, y]$ denotes the length of $[x, y]$.
\item Let $X$ be a manifold, and let $Y$ be a subset of $X$. 
Given a covering map $\phi\cn \td{X} \to X$,   the {\it total lift} of $Y$ is the inverse image $\phi\iv(Y)$. 
\item We say two submanifolds intersect (at a point) {\it $\ep$-nearly  orthogonally} for $\ep > 0$, if the intersection angle is $\ep$-close to $\pi/2$.
\item By a {\it loop}, we mean a simple closed curve.
\item By a {\it marking homeomorphism}, we mean a homeomorphism that represents a given marking on a geometric structure. 
 \end{itemize}


\section{Preliminaries}

\subsection{Geodesic laminations and pleated surfaces}(See \cite{CanaryEpsteinGreen84, Casson-Bleiler-88} for details)\Label{6.18.13} 
A {\it geodesic lamination} $\lam$ on a hyperbolic surface $\tau$ is a set of disjoint simple geodesics whose union is a closed subset of $\tau$. 
The simple geodesics of $\lam$ are called {\it leaves}. %
Let $|\lam|$ denote the closed subset. 
Occasionally, $\lam$ may refer to the closed subset $|\lam|$, when it is clear from the context.
 A geodesic lamination $\lam$ is {\it minimal} if there is no non-empty sublamination of $\lam$.

A {\it measured (geodesic) lamination} $L$ on $\tau$ is a pair $(\lam, \mu)$ of a geodesic lamination $\lam$ and a transversal measure $\mu$ of $\lam$. 
If $L$ is non-empty, by identifying $\mu$ with its scalar multiples by positive real numbers, we obtain a projective measured lamination $[L]$ of $L$.

Given a geodesic lamination $\lam$ on a hyperbolic surface $\tau$, a {\it stratum} is a leaf of $\lam$ or the closure of a (connected) component of $\tau \minus |\lam|$.
A {\it pleated surface} $\beta \cn \h^2 \to \h^3$ {\it realizing} a geodesic lamination $\lam$ on $\h^2$ is a continuous map such that  $\beta$ preserves the lengths of paths  on $\h^2$ and it isometrically embeds each stratum of $(\h^2, \lam)$ into a  (totally geodesic) hyperbolic plane in $\h^3$. 
More generally,  if $\lam$ is a geodesic lamination on a hyperbolic surface $\tau$, then a pleated surface $\h^2 \to \h^3$ {\it realizes} $(\tau, \lam)$ if it realizes the total lift of $\lam$ to $\h^2$.
If $\beta$ realizes $(\tau, \lam)$, then, unless otherwise stated, we in addition assume that there is no proper sublamination of $\lam$ that $\beta$ realizes, so that $\beta$ ``exactly'' realizes $(\tau, \lam)$.

\subsection{Grafting}\Label{grafting}(see \cite{Goldman-87, Kapovich-01}.)
Let $C = (f, \rho)$ be a projective structure on $S$. 
A loop $\ell$ on $C$ is called {\it admissible} if 
\begin{itemize}
\item[(i)] $\rho(\ell) \in \psl$ is loxodromic, and
\item[(ii)] $f$ embeds $\td{\ell}$ into $\rs$, where  $\td{\ell}$ is a lift of $\ell$ to the universal cover of $S$.
\end{itemize} 

If $\ell$ is admissible,  the loxodromic element $\rho(\ell)$ generates an infinite cyclic group $Z$ in $\psl$.
Then its limit set $\Ld(Z)$ is the union of  the attracting and repelling fixed points of $\rho(\ell)$ (on $\rs$), and  $Z$ acts on its complement $\rs \sm \Lambda(Z)$ freely and properly discontinuously.
Thus the quotient $(\rs \sm \Ld(Z)) / Z$ is a projective torus $T_\ell$ ({\it Hopf torus}). 
Then, by (ii), $\ell$ is isomorphically embedded in $T_\ell$. 
Since $\ell$ is also a loop on $C$, there is a canonical way to combine the projective surfaces $C$ and $T_\ell$ by cutting and pasting along $\ell$ as follows. 
We see that $T_\ell \sm \ell$ is a cylinder and $C \sm \ell$ is a surface with two boundary components.
Thus we obtain  a  new projective structure on $S$ by pairing up the boundary components of $T_\ell \sm \ell$ and $C \sm \ell$  in an alternating manner and isomorphically identifying them.
This surgery operation is called {\it $(2\pi$-)grafting} of $C$ along $\ell$, and we denote the new projective structure by $\operatorname{Gr}_\ell (C)$.
It turns out that $\rho$ is holonomy representation $\Gr_\ell(C)$.

\subsection{Thurston coordinates}\Label{thurston}
(see \cite{Kamishima-Tan-92, Kullkani-Pinkall-94} and also \cite{Dumas-08, Tanigawa-97, Baba-10}.)

We here explain more about the parametrization  
\begin{equation}
\PP \cong \TT \times \ML. \label{11-15-1}
\end{equation}
discussed in \S \ref{sThurston}. 

For example, suppose that a projective structure $C \in \PP$ is isomorphic, as a projective surface, to an ideal boundary component of  a hyperbolic three-manifold.
Then the Thurston coordinates of $C$ are the structure on the corresponding boundary component of the convex core of the three-manifold: a hyperbolic surface bent along a measured lamination.

\subsubsection{Bending maps}\Label{sBending}
Let $(\tau, L) \in \TT \times \ML$, and regard the measured lamination $L$ as a geodesic measured lamination on the hyperbolic surface $\tau$.
Set $L = (\ld,\mu)$, where $\ld \in \gl(S)$ and $\mu$ is a transversal measure supported on $\ld$. 
Let $\Lt = (\ldt, \mut) \in \ML(\h^2)$ be the total lift of $L$ to $\h^2$.
Then there is a corresponding pleated surface $\h^2 \to \h^3$ obtained by bending a hyperbolic plane inside $\h^3$ along $\ldt$ by the angles given by $\mut$.
This map  is called the {\it bending map} $\beta\cn \h^2 \to \h^3$ {\it induced by} $(\tau, L)$.
Then $\beta$ is unique up to a postcomposition with an element of $\psl$.
If $C = (f, \rho) \in \PP$ corresponds to $(\tau, L) \in \TT \times \ML$ by (\ref{11-15-1}), then we say that $\beta$ is the bending map {\it associated} with $C$.
Since the $\po(S)$-action on $\h^2$ preserves $\Lt$,  
there is a homomorphism $\po(S) \to \psl$ under which $\beta$ is equivariant.
Then this homomorphism is unique up to a conjugation by an element of $\psl$, and it indeed is the holonomy representation $\rho$ of  $C$.

On the other hand,  given a measured lamination $L$ on $\h^2$,   this pair $(\h^2, L)$ gives a projective structure on an open disk.

\subsubsection{Maximal balls and collapsing maps}\Label{11-18-4}
 (See \cite[\S4]{Kullkani-Pinkall-94},\cite[\S 1.1]{Kamishima-Tan-92}.)

Let $C \in \PP$.
Let $\Ct$ be the universal cover of $C = (f, \rho)$.
An open topological ball $B$ in $\Ct$ is called a {\it maximal ball} if  the developing  map $f\cn \Ct \to \rs$ embeds $B$ onto a round open ball in $\rs$ and there is no such a ball in $\Ct$ properly containing $B$. 
Let $B$ be a maximal ball, and let $H$ be the hyperplane in $\h^3$ bounded by the round circle $\pt f(B)$.
Then, recalling $\rs$ is naturally the ideal boundary of  $\h^3$,  let $\Phi\cn f(B) \to H$ be the canonical conformal map obtained from the nearest point projection onto $H$. 
Let  $\pt_\If B$ be $\bd f(B) \minus f(cl(B))$, where $``cl"$ denotes the closure on $\til{C}$.

Suppose that  $(\tau, L) \in \TT \times \ML$ corresponds to $C = (f, \rho)$ in (\ref{11-15-1}).
Let $\beta\cn \h^2 \to \h^3$ be the bending map induced by $(\tau, L)$.  
Then the maximal $B$  corresponds to a {\it stratum} $X$ of $\td{L}$, which is either a leaf of $\td{L}$ or the closure of a  component  of $\h^2 \sm |\td{L}|$.
Indeed $\beta$ isometrically embeds $X$ into the hyperbolic plane bounded by $\pt_\If f(B)$.

Clearly $\Phi \cc f$ embeds  $B$ onto $H \st \h^3$.
The {\it core} of the maximal ball $B$, denoted by  $\Core(B)$, is the convex hull of $\pt_\In B$ with $B$ conformally identified with $\h^2$.
Thus we have a unique embedding $\kp_B$ of $\Core(B)$ onto  $X \st \h^2$ so that $\beta \cc \kp_B =  \Phi \cc f$ on $\Core(B)$.
Therefore, for each $x \in \Core(B)$, the hyperplane  $H$ is called a {\it hyperbolic support plane} of $\beta$ at $x$.
It turns out that, for different maximal balls $B$ in $\td{C}$, their cores $\Core(B)$ are disjoint.
Moreover $\Ct$ decomposes into these cores.
Therefore we can define a continuous map $\kpt\cn \Ct \to \h^2$ by $\kpt = \kp    _B$ on $\Core(B)$ for all maximal balls $B$. 
Then $\kpt$ commutes with the action of $\po(S)$, and thus it descends to the {\it collapsing map} $\kp\cn C \to \tau$, which respects the markings by homeomorphisms from $S$. (See \cite[\S 8]{Kullkani-Pinkall-94}), \cite[\S 2.3]{Kamishima-Tan-92}

\subsubsection{Canonial lamination on projective surfaces} 

Each boundary component of $\Core(B)$ is a biinfinite line properly embedded in $\Ct$.
Then, by taking the union of  $\pt \Core(B)$ over all maximal balls $B$ in $\Ct$, we obtain a (topological) lamination $\nut$ on $\Ct$.
Then $\kpt$ embeds  each leaf of $\nut$ onto a leaf of $\td{\lam}$. 
Since $\pi_1(S)$ preserves the decomposition of $\til{C}$ into the cores, $\nut$ descends to a lamination $\nu$ on $C$, and $\kp$ embeds each leaf of $\mu$ onto a leaf of $\ld$. 

Moreover $\nu$ is equipped with a natural transversal measure $\omega$ so that $\LL := (\nu, \omega)$ descends to $L$ by $\kap$. 
Then $\LL$ is called the {\it canonical measured lamination} on $C$.
If $\ap$ is  a curve transversal to $\omega$ and it is of infinitesimal length, its transversal measure $\omega(\ap)$ is the angle between the hyperbolic support planes in $\h^3$ corresponding the leaves of $\nu$ containing the endpoints of $\ap$. 
The transversal measure $\omega$ is infinitesimally given by the angles between hyperplanes supporting of $\beta$.  

Let  $M$ be the union of the closed leaves $\ell_1, \ell_2, \dt, \ell_n$ of $L$. 
In particular $\kap$ is a $C^1$-diffeomorphism in the complement of $\kap^{-1}(M)$, and there $\omega$ is exactly the pullback of $\nu$ by $\kap$. 
We describe $\LL$ on  $\kap^{-1}(M)$ below. 

For each $i = 1,2, \dots, n$,  $\kp\iv(\ell_i)$ is a compact cylinder embedded in $C$.
This cylinder is foliated by closed leaves $m$ of $\nu$  that are diffeomorphic to  $\ell$ by $\kap$. 
The total transversal measure of  $\kp\iv(\ell_i)$ is the weight $w_\mu(\ell_i)$ of $\ell_i$ given by $\mu$. 

In addition,  for every $s \in \ell$, $\kp^{-1}(s)$ is a circular arc connecting the boundary circles of $\kp\iv(\ell_i)$ and it is orthogonal to each closed leaf $\nu$ in $\kp^{-1}(\ell_i)$. 

\subsubsection{Thurston metric on projective structures}\Label{11-12-3}
Every projective surface $C$ on $S$ has a natural Hyperbolic/Euclidean type metric associated with $\kap \col (C, \LL) \to (\tau, L)$. 

The cylinder $\kp\iv(\ell_i)$ has a natural Euclidean metric, and it is isometric to a product of a circle of length  $\length_\tau(\ell_i)$ and  the interval $[0, w_\mu(\ell_i)]$.
The Riemannian metric respects the conformal structure of $C$ on $\kap^{-1}(\ell_i)$.
For each closed leaf $\ell$ of $\nu$ in $\kap^{-1}(\ell_i)$, $\kap| \ell$ is an isometry onto $\ell$. 
For each $s \in \ell$, the metric on the circular arc $\kap^{-1}(s)$ is given by the transversal measure $\omega$.
This Euclidean metric is the restriction of the Thurston metric on $C$ to $\kappa^{-1}(\ell_i)$.

On the other hand, the restriction of $\kp\cn C \to \tau$ to $C \sm \kp\iv(M)$  is a $C^1$-diffeomorphism onto $\tau \sm M$. 
Thus $C \sm \kp\iv(M)$ has the hyperbolic metric obtained by pulling back the hyperbolic metric of $\tau$ via $\kp$.

On each stratum $R$ of $(C, \LL)$,  the Thurston metric is the restriction of the Euclidean or Hyperbolic metric defined above. 
In this paper, it suffices to use the Thurston metric on each stratum. 
(If $L$ is a union of disjoint weighted loops, the Thurston metric on $C \cong (\tau, L)$ is the piecewise  Euclidean/hyperbolic metric that is the sum of the Euclidean metric on the cylinders and the hyperbolic metric in the complement. 
For general $L$, we can take a sequence of weighted loops $\ell_i$ converging to $L$ as $i \to \infty$. 
Then the Thurston metric on $C \cong (\tau, L)$ is  the limit of the Thurston metrics on the projective surfaces given by $(\tau, \ell_i)$.)

\subsection{Equivariant homotopies}

\begin{lemma}\Label{9-11-2}
Let $\rho\cn \pi_1(S) \to \psl$ be a homomorphism. 
Suppose that there are two continuous maps  $\beta\cn \td{S} \to \h^3$ and  $\beta'\cn \td{S} \to \h^3$ that are $\rho$-equivariant. 
Then  $\beta$ and $\beta'$ are $\rho$-equivariantly homotopic, i.e. homotopic through $\rho$-equivariant maps $\td{S} \to \h^3$. 
\end{lemma}

 \proof 
We proceed in steps.
{\it Step 1.} We first construct an equivariant homotopy for each loop $l$ on $S$.
Let $\lt$ be a lift of $l$ to the universal cover $\St$ of $S$.
Then $\beta_1 | \lt$ and  $\beta_2 |\lt $ are equivariant under the restriction of $\rho$ to  $\langle l \rangle$, the infinite cyclic subgroup of $\pi_1(S)$ generated by $l \in \po(S)$. 
Note that $\rho(l)$ may be of any type of hyperbolic isometry, i.e. parabolic, elliptic, or loxodromic. 
Then, in each case,  we can easily construct a homotopy between $\beta_1|\lt$ and $\beta_2|\lt$  that is equivariant under  $\rho|\langle l \rangle$.

{\it Step 2.} Next let $P$ be a pair of pants embedded in $S$.
Let $l_1, l_2, l_3$ be the boundary loops of $P$.
Let $\Pt$ be a lift of $P$ to $\St$.
Then we show that there is a homotopy between  $\beta_1|\Pt$ and $\beta_2| \Pt$ equivariant under $\rho|\po(P)$.
For each $j =1,2,3$, pick a lift $\lt_j$ of $l_j$ to $\Pt$. 
Then, by Step 1, we have a homotopy connecting  $\beta_1|\lt_j$ and $\beta_2|\lt_j$ equivariant under $\rho|\po(l_j)$.
By equivariantly extending those homotopies, we have a homotopy $\Phi_{\pt \Pt}\cn \pt \Pt \times [0,1] \to \h^3$ between $\beta_1| \pt \Pt$ and $\beta_2 | \pt \Pt$  that is $\rho|\po(P)$-equivariant.
Pick  disjoint arcs $a_1, a_2, a_3$  properly embedded in $P$ that  decompose $P$ into two hexagons. 
Then we can easily extend the homotopy $\Phi_{\pt \Pt}$ to a homotopy between the lifts of arcs $a_i ~(\text{for } i  = 1,2,3)$ to $\Pt$ so that the extension is still equivariant under $\rho| \po(P)$.
Since $a_i$'s decompose $P$ into simply connected surfaces, we can further extend the homotopy to the $\rho| \pi_1(P)$-equivariant homotopy between $\beta_1 |\til{P}$ to $\beta_2 | \til{P}$.  

{\it Step 3.} Pick a {\it maximal} multiloop $M$ on $S$, which decomposes $S$ into pairs of pants $P_k ~ (k = 1, 2, \dt, 2(g-1))$.
Let $\Mt$ denote the total lift of $M$ to $\St$. 
Then we can obtain a $\rho$-equivariant homotopy $\Phi_{\Mt}$ between $\beta_1| \Mt$  and $\beta_2| \Mt$ similarly to the way we obtained the homotopy $\Phi_{\pt \Pt}$ in Step 2. 
For each  $k \in \{ 1,2, \dt, 2(g-1)\}$, let $\Pt_k$ be a lift  of $P_k$ to $\St$.
Then $\Phi_{\Mt}$ induces a homotopy  $\Phi_{\pt \Pt_k}$ between $\beta_1| \pt \Pt_k$ and $\beta_2| \pt \Pt_k$ that is equivariant under $\rho|\pi_1(P_k)$.
Similarly to Step 2, we can extend this induced homotopy to a  homotopy   $\Phi_{\Pt_k}$ between $\beta_1| \Pt_k$ and $\beta_2| \Pt_k$ that is equivariant under $\rho|\po(P_k)$.
By $\rho$-equivalently extending  the homotopies $\Phi_{\Pt_k} ~(k = 1,2, \ldots, 2(g-1))$, we obtain a $\rho$-equivariant homotopy between $\beta_1$ and $\beta_2$.
\endproof

\subsection{Isomorphisms of projective structures via developing maps} 

\begin{definition}\Label{isomorphic}
Let $F$ be a surface and $\rho\cn \pi_1(F) \to \psl$ be a homomorphism. 
Let $C_1 = (f_1, \rho)$ and $C_2 = (f_2, \rho)$ be  projective structures on $F$ sharing holonomy $\rho$, where  $\td{F}$ is the universal cover of $F$ and  $f_1, f_2\cn \td{F} \to \rs$ are their developing maps.
Then $C_1$ and $C_2$  are {\it isomorphic} (as projective structures) via $f_1$ and $f_2$, if there is a homeomorphism $\phi\cn F \to F$ homotopic to the identify map, such that, letting $\phit\cn \td{F} \to \td{F}$ be the lift of $\phi$, we have $f_1 = f_2 \cc \phit\cn \til{C}_1 \to \rs$.
We also say that the isomorphism $\phi$ is {\it compatible} with $f_1$ and $f_2$.
\end{definition}


\section{Bilipschitz curves  on pleated surfaces}\Label{BilipschitzGeodesics}
Let $L$ be a measured geodesic lamination on $\h^2$ with $\Area_{\h^2}(|L|) = 0$.
Let $\beta_L\cn \h^2 \to \h^3$ be the bending map induced by $L$. 
In this section, we prove
\begin{proposition}\Label{8-14-1}
For every $\epsilon > 0$, there is $\delta > 0$  such that, if $l$ is a geodesic on $\h^2$ with $\angle_{\h^2}(l, L) < \delta$, then,
\begin{itemize}
\item[(i)] $\beta_{L}$ is a $(1 + \epsilon)$-bilipschitz embedding $l  \to \h^3$,\end{itemize}
 and, letting $m$ be the geodesic  in $\h^3$ connecting the endpoints of the quasigeodesic $\beta_{L}\vert l$, 
\begin{itemize}
\item[(ii)] for each point $x \in l$, $\beta_L(x)$ is $\epsilon$-close to $m$, and, if $\beta_L$ is differentiable at $x$, then the tangent vector of $\beta_L|l$ at $x$ is \underline{$\epsilon$-parallel} to $m$,\end{itemize}
 that is,  
 the tangent vector of $\beta_L|l$ in $\h^3$ at $x$ is $\epsilon$-nearly orthogonal  to the(totally geodesic) hyperbolic plane  orthogonal to $m$ and containing $\beta(x)$.  
  \end{proposition}
\begin{remark}  
Similar  statements are in \cite{Canary-Epstein-Green-06, Epstein-Mardern-Markovic, Baba-10}. 
However the condition on $\angle_{\h^3}(l, L)$ is new.
\end{remark}
Let $\Phi_m\cn \h^3 \to m$ be the nearest point projection.
Then
\begin{corollary}\Label{9-29-1}
(iii) $\Phi_m \cc \beta_L | \,l$  is a $(1 + \epsilon)$-bilipschitz map $l \to m$. 
\end{corollary}

\begin{proof}[Proof of corollary]
For each point $y \in m$, $\Phi_m\iv(y)$ is the  hyperbolic lane in $\h^3$ orthogonal to $m$.
Then $\h^3$ is foliated by the hyperplanes. 
Since $\Area_{\h^2}(\ld) = 0$, $\beta|l$ is differentiable almost everywhere.  
By Proposition \ref{8-14-1} (ii),   the curve $\beta_L|l$ stays in a small neighborhood of $m$ and  $\epsilon$-orthogonally intersects the hyperplanes of $\h^3$ at almost every point of $l$.
If $\delta > 0$ is sufficiently small,  at almost every point on $l$,   the ratio of the lengths of the tangent vector along $\beta_L\vert l$ and of its $\Phi_m$-image is bounded by $ (1 + \epsilon)$.
\end{proof}

We first prove an analogue of Proposition \ref{8-14-1}  for geodesic segments of bounded lengths: 
\begin{proposition}\Label{8-28-1}
For every (large) $K > 0$ and  (small) $\epsilon > 0$,  there exists a $\delta > 0$ such that:
\begin{itemize}
\item[(i)] If $L$ is a measured geodesic lamination on $\h^2$, and $l\cn \R \to \h^2$ is a parametrized geodesic  at unit speed such that $\angle(l, L) < \delta$, then, if points $x, y$ on $l \,(\cong \R)$ satisfies  $0 <  y - x < K$, then we have $(1 - \epsilon) \cdot \dist_{\h^2}(x, y) <  \dist_{\h^3} (\beta_L(x), \beta_L(y))$.
\item[(ii)] If $\beta| l$ is differentiable at $x \in \R$,  for all $y \in  \ell$ with $0 <  y - x  < K$,  then $\theta_y(x)  < \epsilon$,
where $\theta_y(x) \in [0, \pi]$ is the angle between the  geodesic segment from $\beta_L(x)$ to $\beta_L(y)$ and the tangent vector of $\beta_L | l$ at $\beta(x)$; see Figure \ref{10-30-2'}.
\end{itemize}

\begin{figure}
\begin{overpic}[scale=.5
] {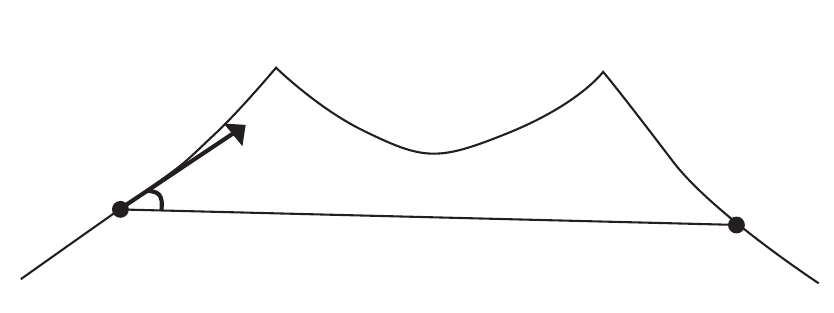} 
      \put(5,18 ){$\beta_L(x)$}  
 \put(90,15 ){$\beta_L(y)$}  
 \put(47,24 ){$\beta_L|l$}  
\put(24, 14.5){$\theta_y(x)$}
      \end{overpic}
\caption{}\label{10-30-2'}
\end{figure}

\end{proposition}

\Proof~
 First consider a right hyperbolic triangle $\triangle ABC$ in $ \h^2$ (with geodesic edges)  with $\angle C = \pi/2$, where $A, B, C$  are its vertices.
Then it is easy to prove
\begin{lemma}\Label{8-30-1}
For every $K > 0$ and $\epsilon' > 0$, there is $\delta > 0$ such that, if $\angle B < \delta$ and $\dist(A, B)   <  K$, then 
\begin{itemize}
\item[(i)] $(1 - \epsilon') \cdot \operatorname{dist}(A,B) <  \dist(B,C) - \dist(C,A)$, and
\item[(ii)]  $\angle  A'BC < \epsilon'$ for every $A' \in \h^2$ with $\dist(C, A') < \dist (C, A)$.
\end{itemize}
\end{lemma}

Let $K > 0$ and $\epsilon > 0$.
Let $\epsilon' = \epsilon/2$.
Then let $\delta > 0$ be the number obtained by applying  Lemma \ref{8-30-1} to $K$ and $\epsilon'$.
Then we can  in addition assume that $\delta < \epsilon/2$.

Let $x$ and $y$ be distinct points on $l$ with $0 <  y - x  < K$.
Let $I $ be the minimal sublamination of $L$ containing the leaves that intersect $[x,y]$.
We can assume that $[x,y]$ intersects at least one leaf of $L$ transversally, since otherwise Proposition \ref{8-28-1} clearly holds. 
Let $m$ denote the leaf of $I$ closest to $x$.
Then, there is a unique point $z \in \h^2$ such that  $\triangle x y z$ is a hyperbolic triangle with $\angle z = \pi/2$ and such that $[x,z] \st \eta(m)$, where $\eta\cn \h^2 \to \h^2$ is the translation along $l$ taking the point $l \cap m$ to $x$.  (See Figure 3.)
\begin{figure}
\begin{overpic}[scale=.4
] {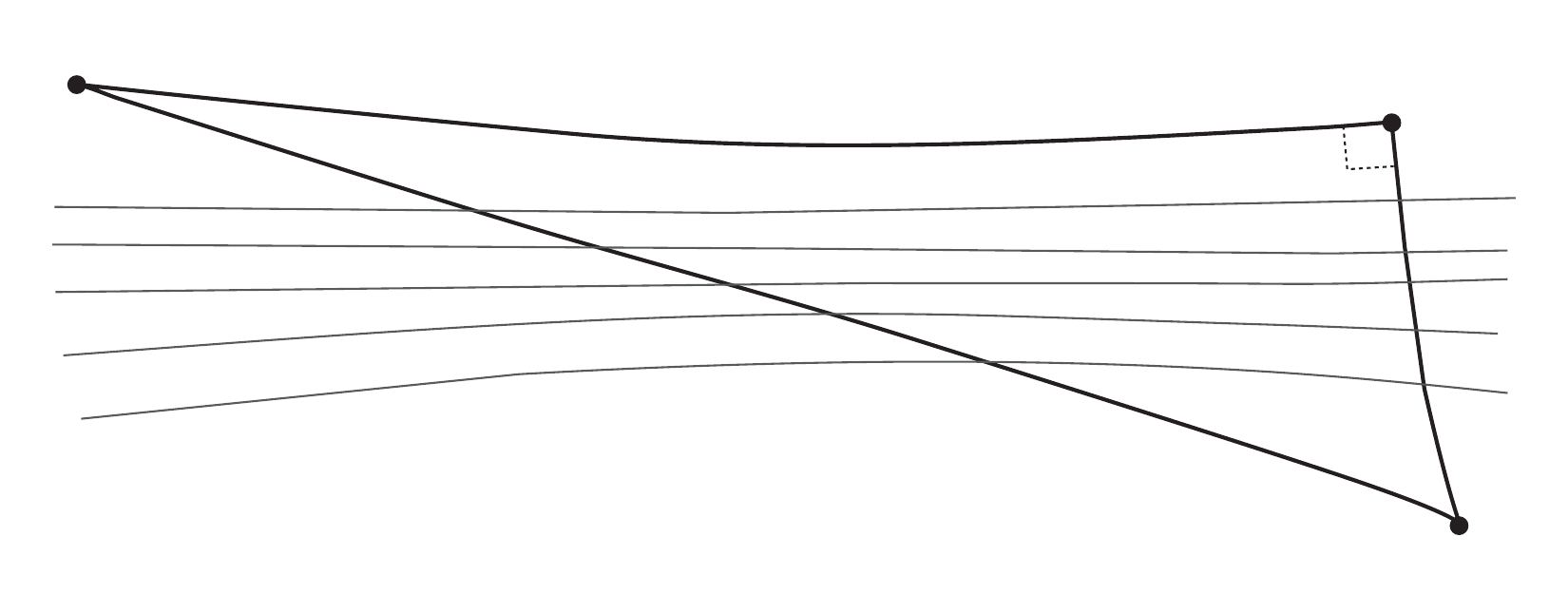} 
\put(2,33){$x$}
\put(89,31){$z$}
\put(93,1){$y$}
\put(10,25){$m$}
      \end{overpic}
\caption{}\label{10-30-2}
\end{figure}

Then $[x, z]$ is disjoint from $I$ if $x$ is in the complement of $I$.
Then, since $\angle (l, L) < \delta$, in particular  $\angle yxz < \delta$.
Let $\beta_I\cn \h^2 \to \h^3$ be the bending map induced by $I$.
Then $\beta_I$ isometrically embeds $[x, z]$ into $\h^3$.
Therefore $\dist_{\h^3}(\beta_I(x), \beta_I(z)) = \dist_{\h^2}(x, z)$. 
Since bending maps are $1$-lipschitz, $\dist(\beta_I(z), \beta_I(y)) \leq \dist(z, y)$.
By the triangle inequality, we have  
\begin{eqnarray*}
\dist(\beta_I(x), \beta_I(y)) &\geq& \dist(\beta_I(x), \beta_I(z)) - \dist(\beta_I(z), \beta_I(y))\\
 &\geq& \dist (x,z) - \dist(z,y).
\end{eqnarray*}

Then, by Lemma \ref{8-30-1} (i),  we have $$\dist(\beta_I(x), \beta_I(y)) > (1- \epsilon')\cdot \dist(x,y).$$
Since $\beta_I = \beta_L$ on $[x, y]$,  $\dist(\beta_L(x), \beta_L(y)) > (1 - \epsilon') \cdot \dist(x,y)$; thus we have shown (i).

 By Lemma \ref{8-30-1} (ii) applied to $\triangle A' B C =  \triangle \beta_I(y) \beta_I(x) \beta_I(z)$, we have \linebreak[4] $\angle \beta_I(y) \beta_I(x) \beta_I(z) < \epsilon'$.
 By the triangle inequality on the sphere in $\h^3$ centered at $\beta_I(x)$ of  infinitesimal radius,   $$\theta_y(x) \leq \angle y x z + \angle \beta_I(y) \beta_I(x) \beta_I(z)  < \delta + \epsilon'  < \epsilon.$$
 Thus we have proved (ii).
\Qed{8-28-1}

{\it Proof  of Proposition \ref{8-14-1}.}
(i)  
We first show for every $\epsilon' > 0$, there is $\delta > 0$ such that, if a geodesic lamination $L$ on $\h^2$ and a geodesic $l \cn \R \to \h^2$ satisfy $\angle_{\h^2}(l, L) < \delta$, then   $\theta_x(y) < \epsilon'$ for all  distinct points $x, y$   on $l$ with $x < y$ such that $\beta_L|l$ is differentiable at $y$.  
Pick   $K > 0$  and $\epsilon'' > 0$ with $\epsilon'' < \epsilon'/2$.
Then we can assume that $[x, y]$ is not contained in a leaf of $L$ and,  by Proposition \ref{8-28-1} (ii), that $\dist(x, y) > K$.
Let $\delta' = \delta'(K, \epsilon'') > 0$ be the number obtained by applying Lemma \ref{8-28-1}  to $K$ and $\epsilon''$.
Then divide the geodesic segment $[x,y]$ into subsegments $[p_0, p_1],  [p_1,p_2], \dt, [p_{n-1}, p_n]$, where $x = p_0 <  p_1 < \dt <  p_n = y$,   so that 
\begin{itemize}
\item $K/2 <  p_{i + 1} - p_i  < K$ for $i = 0,1,\dt, n-1$.
\item $p_1, p_2, \dots, p_{n-1}$ are in the complement of $|L|$ (since $\Area_{\h^2}(|L|) = 0$).

\end{itemize}
Let $\beta$ be  the bending map $\beta_L$.
Then the union of the geodesic segments $[\beta(p_i), \beta(p_{i + 1})]$ in $\h^3$ over $i = 0, \dots, n-1$ is a piecewise-geodesic curve in $\h^3$ connecting $\beta(x)$ to $\beta(y)$.
If $\angle(l, L) < \delta'$, then, by Lemma \ref{8-28-1} (ii),  we have $\theta_{p_{n-1}}(y) < \epsilon''$ and
 $\pi - \angle_{\h^3}(\beta(p_{i-1}), \beta(p_i), \beta(p_{i +1}) ) < 2\epsilon''$ for all $i = 1,2,\dt, n-1$.
 By Lemma \ref{8-28-1} (i), $\dist(\beta(p_{i}), \beta(p_{i+1})) > (1 - \epsilon'') \cdot (K/2)$ for $i = 0,1, \dt, n-1$.
Then,   if $\epsilon''>0$ is sufficiently small, since the exterior angles of the piecewise-geodesic curve are sufficiently small relative to the lengths of the segments, we have $\angle \beta(x) \beta(y) \beta(p_{n-1}) < \epsilon'/2$ (see \cite[\S I.4.2]{Canary-Epstein-Green-06}; also \cite{Epstein-Mardern-Markovic, Baba-10}). 
Then, by the triangle inequality, $$0 < \theta_x(y) \leq \angle    \beta(p_{n-1}) \beta(y) \beta(x)   +   \theta_{p_{n-1}}(y) < \epsilon'/2 + \epsilon''.$$
Hence $0 < \theta_x(y) < \epsilon'$.
We have
 $$ \frac{d ~\dist(\beta(x),\beta(y))}{dy} = \cos (\theta_x(y))$$
(see \cite[\S I.4.2]{Canary-Epstein-Green-06}; also \cite{Epstein-Mardern-Markovic, Baba-10}).
Then, for every $\epsilon > 0$, by taking a smaller $\epsilon' > 0$ if necessary,  we have $\frac{1}{1 + \epsilon} <  \cos(\theta_x(y)) \leq 1$ for all different $x, y$ on $l$ such that $\beta|l$ is differentiable at $y$.
Since $\beta|l$ is differentiable at almost all points of $l$,  $\beta|l$ is a $(1 + \epsilon)$-bilipschitz embedding. 

(ii) 
For every $\epsilon > 0$, pick  $\epsilon' > 0$ with $2\epsilon' < \epsilon$.
Then  we have shown, in proving (i), that there exists $\delta > 0$, such that, if  $\angle_{\h^2} (l , L )< \delta$, then
$\theta_x(y) < \epsilon'$ for all different   $x, y \in l$ such that  $\beta|l$ is differentiable at $y$.  
Since $\beta| l$ is bilipschitz, it takes the endpoints $\pm \If$ of the geodesic $l\cn \R \to \h^2$ to the distinct points $\beta(-\If), \beta(\If)$ of the ideal boundary of $\h^3$.  
Thus taking the limits as $x$ goes to the end points of $l$, we have $\theta_{- \If} (y), \theta_{\If} (y) \leq \epsilon'$.
Thus,  we have $\beta(-\If) \beta(y) \beta(\If) > \pi - 2 \epsilon'$.
Let $m$ be the geodesic in $\h^3$ connecting  $\beta(-\If), \beta(\If)$ so that $m$ is a bounded distance away from $\beta|l$.
It is well-known that the area of a triangle in $\h^2$ is equal to $\pi$ minus the sum of the angles of its vertices. 
Thus the area of the geodesic triangle $\triangle \beta(-\If) \beta(y) \beta(\If)$ is less than $2 \epsilon'$.
Thus,  if necessary by taking  smaller $\epsilon' > 0$, we can assume that  $\dist_{\h^3}(\beta(y), m) < \epsilon$. 
Recalling that $\Phi_m\cn \h^3 \to m$ is the nearest point projection, $\triangle \beta(-\If)\beta(y)(\Phi_m\cc\beta)(y)$ has area less than $\epsilon'$.
Applying the same formula to this ideal triangle, we have $\angle_{\h^3} \beta(- \If) \beta(y) (\Phi_m\cc\beta)(y) < \pi/2 -  \epsilon'$.
Since $\theta_{- \If} (y) < \epsilon'$ and  $2 \epsilon'  < \epsilon$, by the triangle inequality, we see that the tangent vector of  $\beta|l$ at $y$ is $\epsilon$-parallel to $m$.
\Qed{8-14-1}

\section{Local stability of bending maps in $\GL$}\Label{stability}

\subsection{Bending maps with a fixed bending lamination.}\Label{062113}

\begin{definition}
Let $X, Y$ be  metric space with distance functions $d_X, d_Y$. 
For every $\ep > 0$, a  map $\phi\cn X \to Y$ is an \underline{$\ep$-rough isometric embedding}, if 
$d_X( p, q) - \ep < d_Y(\phi(p), \phi(q)) < d_X(p, q) + \ep$ for all $p, q \in X$.
It is an \underline{$\ep$-rough isometry} if, in addition,  the $\ep$-neighborhood of $\Im(\phi)$ is $Y$.
\end{definition}

\begin{theorem}\Label{12-19}
Let $(\tau, \lambda) \in \TT \times \GL$ and $\rho\cn \po(S) \to \psl$ be a homomorphism.
Suppose that there is a $\rho$-equivariant pleated surface $\beta\cn \h^2 \to \h^3$ realizing $(\tau, \lambda)$. 
Then, for every $\ep > 0$, there is a $\dl > 0$, such that, if there is a pair $(\sigma, \nu) \in \TT \times \GL$ and a $\rho$-equivariant pleated surface $\beta'\cn \h^2 \to \h^3$ realizing $(\sigma, \nu)$ and $\angle_{\tau}(\lambda, \nu) < \dl$,
then $\beta'$ and $\beta$ are $\ep$-close in the following sense: 
There is a marking-preserving homeomorphism  $\psi\cn \tau \to \sigma$ such that $\psi$ is an $\ep$-rough isometry and, letting  $\td{\psi}\cn \h^2 \to \h^2$ be its lift, the maps $\beta$ and $\beta' \cc \td{\psi}$ are $\ep$-close in the $C^0$-topology and moreover
  in the $C^1$-topology in the complement of the $\ep$-neighborhood of $|\td{\lam}| \cup |\td{\nu}|$ in  the universal cover $\td{\tau} = \h^2$, where $\td{\lam}$ and $\td{\nu}$ are the total lifts of $\lam$ and (the geodesic representative of) $\nu$ on $\tau$ to $\til{\tau}$. 
\end{theorem}

The rest of \S \ref{062113} is the proof of Theorem \ref{12-19}.
Suppose that there is a sequence of $\rho$-equivariant pleated surfaces $\beta_i\cn \h^2 \to \h^3$ realizing some $(\sigma_i, \nu_i) \in \TT \times \GL$.
Let $\nu_{i, \tau}$ denote the geodesic representative of $\nu_i$ on $\tau$. 
Assuming  $\angle_\tau (\nu_{i, \tau} ,\ld) \to 0$ as $i \to \ify$, we will construct a  homeomorphism $\psi_i\cn \tau \to \sigma_i$ such that, for every $\ep > 0$, if $i$ is sufficiently large, then $\beta$ and $\beta_i \cc \til{\psi}_i\col \til{\tau} \to \h^3$ are $\ep$-close in the $C^0$-topology and in the $C^1$-topology in the complement of the $\ep$-neighborhood of $\nu_{i, \tau}$.

\vspace{3mm}
{\it Outline of the construction.}
We first construct $\psi_i \col \nu_{i, \tau} \to \nu_i$ such that $\psi_i$ is $(1 + \ep)$-bilipschitz on every leaf of $\nu_{i, \tau}$ for sufficiently large $i$  and, as desired,  $ \beta' \cc \til{\psi}_i \to \beta$ in the $C^0$-topology.  
This bilipschitz property is given by Proposition \ref{8-14-1}.
Then we continuously extend it to $\psi_i\col \tau \to  \sigma_i$ so that $\psi_i$ is $(1 + \ep)$-bilipschitz on each stratum of $(\tau, \nu_{i, \tau})$ and $ \beta' \cc \til{\psi}_i \to \beta$ as desired.  
In particular, given a compact subset $K$ of a stratum of $(\tau, \nu_{i, \tau})$,  $\psi_i$ is an $\ep$-rough isometry on $K$ for sufficiently large $i$ (Lemma \ref{9-14-12no1}).
We take a``sufficiently thick part" of the stratum to be the compact subset $K$.
Finally, in order to show that $\psi_i$ is an $(1 + \ep)$-rough isometry,  we show that $\psi_i$ is an $\ep$-rough isometry along arcs transversal to the lamination $\lam$ (Lemma \ref{061313n2}). 

\vspace{3mm}

Since $\GL$ is compact with the Chabauty topology, we can  assume that $\nu_{i, \tau}$ converges to some $\nu_\ify  \in \GL(\tau)$ as $i \to \In$. 
Then $\angle_\tau (\nu_\ify, \lambda) = 0$.
We moreover have
\begin{proposition}\Label{1-14no1}
$\lambda$ is a sublamination of $\nu_\ify$.
\end{proposition}

\proof 
Since $\angle_\tau(\nu_\ify,\lambda) = 0$, the union $\nu_\ify \cup \lambda$ is a geodesic lamination on $\tau$. 
Suppose that $\lambda$ is {\it not} a sublamination of $\nu_\ify$.
Then there is a  leaf of $\lambda$ {\it not} contained in $\nu_\ify$.
Below each tilde symbol $``\sim"$ denotes either the universal cover of a surface, e.g. $\td{\tau} \cong \h^2$, or the total lift of a geodesic lamination to the universal cover,  e.g. $\td{\lam}$.  \note{eng: articles}
Then there are  distinct components $R$ and $R'$ of $\td{\tau} \cong \h^2$ minus the total lift  of ${\nu_\ify \cup \lambda}$, such that 
\begin{itemize}
\item a leaf of $\td{\lam}$ separates $R$ and $R'$,
\item yet $R$ and $R'$ are contained in a single component $P$ of $\td{\tau} \minus \td{\nu}_\infi$, and
\item either
\begin{itemize} 
\item $R$ and $R'$ share a boundary geodesic and $\beta$ bends $\h^2$ along the geodesic by the angle $\pi$, or
\item $\beta(R)$ and $\beta(R')$ are contained in distinct copies of $\h^2$ in $\h^3$. 
\end{itemize}
\end{itemize} 

Since $\nu_{i, \tau} \to \nu_\If$,  for every $i \in \n$,   we can pick  a component  $P_i$ of $\td{\tau} \sm \tilde{\nu}_{i, \tau}$ such that $P_i$ converges to $P$ uniformly on compacts as $i \to \In$. 
Then $P_i \cap R \to R$ and $P_i \cap R' \to R'$ as $i \to \infi$. 
 Then, let $Q_i$ be the component of $\td{\sigma}_i \sm \td{\nu}_i$ that  corresponds to $P_i$ so that a marking-preserving homeomorphism $\sigma_i \to \tau$ induces a homeomorphism $Q_i \to P_i$.
Then $\beta_i\col \h^2 \to \h^3$ isometrically embeds $Q_i$ in a copy $H_i$ of $\h^2$.
 By Lemma \ref{9-11-2}, $\beta$ and $\beta_i$ are $\rho$-equivariantly homotopic, and  thus $\beta| P_i$ and $\beta_i| Q_i$ are a bounded distance apart pointwise via the homeomorphism  $Q_i \to P_i$.

\begin{claim}
For every $\ep > 0$, if $i$ is sufficiently large, then $\beta|P_i$ and $\beta_i|Q_i$ are $\ep$-close in the $C^0$-topology via $Q_i \to P_i$. 
\end{claim}
\begin{proof}
For every $\ep > 0$, if $i$ is sufficiently large, then $\angle_{\til{\tau}}(\ell, \lam) < \ep$ for each boundary geodesic $\ell$ of $P_i$. 
Therefore, by Proposition \ref{8-14-1}, if $i$ is sufficiently large,  $\beta_i | \pt Q_i$ is $\ep$-close to $\beta | \pt P_i$.
Thus we can in addition assume that  $\beta_i | \ell$  is an $(1 + \ep, \ep)$-quasiisometric embedding for  every geodesic or geodesic segment  $\ell$ in $P_i$ not transversal to $\lam$. 
This implies  the claim. 
\end{proof}

By this claim, $\beta| P_i$ become more and more totally geodesic as $i \to \infi$. 
Since $P_i \cap R  \to R$ and $P_i \cap R' \to R$, the hyperbolic plane containing $\beta(R)$ must coincide with the hyperbolic plane containing $\beta(R')$, and moreover $\beta(R)$ and $\beta(R')$ must be disjoint. 
This is a contradiction to the third hypothesis of $R$ and $R'$ above.
\Qed{1-14no1}

For each $i$, we enlarge the geodesic lamination $\nu_i$ to a {\it maximal} lamination, which decomposes $\sigma_i$ into ideal triangulations. 
By taking a subsequence if necessary, we can assume that the maximal lamination $\nu_i$ converges to a maximal lamination containing  $\nu$.
Thus accordingly we denote the limit by $\nu$.  
Similarly let $\nu_{i, \tau}$ be the geodesic lamination on $\tau$ representing $\nu_i$.
Then we still have $\angle_\tau (\nu_{i, \tau} ,\ld) \to 0$ as $i \to \ify$ (by Proposition \ref{1-14no1}).
Although $\nu_i$ is not a ``minimal'' lamination realizing $\beta_i$, it will not affect our arguments.

We construct a homeomorphism $\psi_i\cn \tau \to \sigma_i$ for all sufficiently large $i$.
First, since $\angle_\tau (\nu_i ,\ld) \to 0$ as $i \to \ify$,  for every $\ep > 0$, if $i$ is sufficiently large, by Lemma \ref{9-11-2},  and  Corollary \ref{9-29-1}, there is a bijection $\psi_i\cn \nu_{i, \tau} \to \nu_i$ that is a $(1 + \ep)$-bilipschitz map on each leaf of $\nu_{i, \tau}$ (to be precise  $\psi_i\cn|\nu_{i, \tau}| \to| \nu_i|$).
If there is a sequence of leaves $\ell_j$ of $\nu_{i, \tau}$ converging to a leaf $\ell_\In$ of $\nu_{i, \tau}$, then $\beta_i| \ell_j$ converges to $\beta_i | \ell_\In$  uniformly on compacts as $j \to \In$. 
Then, since if  $i$ is sufficiently large, $\beta_i|\ell_j$ are $(1 + \ep)$-bilipschitz for all $j$,   the endpoints of $\beta_i|\ell_j$ converge to the endpoints to $\beta_i|\ell_\In$ on $\rs$ as $j \to \In$. 
Since $\psi_i$ is obtained from  Corollary \ref{9-29-1},
we see that the entire map $\psi_i\cn \nu_{i, \tau} \to \nu_i$ is a homeomorphism with the topology induced from $\til{\tau}$ and $\til{\sigma}_i$.  

Given $\ep > 0$ and  a connected component $\Delta$ of $\tau \sm \nu_{i, \tau}$, let $\Delta_\ep$ be the \textit{$\ep$-thick part} of $\Delta$, that is, the union of  disks of radius $\ep$ embedded in $\Delta$.
 Then $p  \in \bd \Delta \cap \pt \Delta_\ep$ if and only if  $\Delta$ contains a disk of radius $\ep$ tangent to $\pt \Delta$ at $p$.
 Since $\Delta$ is an ideal triangle,  $\pt \Delta \cap \pt \Delta_\ep$ is compact,  and if $\ep > 0$ is sufficiently small, then it is  a union of three long (but finite) segments of the edges of $\Delta$.
 
For every $\zeta > 0$,  if $\ep > 0$ is sufficiently small, every $(1 + \ep)$-bilipschitz curve in $\h^3$, in particular $\beta_i|\ell_j$ above, is contained in the $\zeta$-neighborhood of the geodesic in $\h^3$ connecting its endpoints on $\rs$. 
Thus, since $\pt \Delta \cap \pt \Delta_\ep$ is bounded, we have
\begin{lemma}\Label{9-14-12no1}
 For every $\ep > 0$, if $i \in \N$ is sufficiently large, 
 then, for every component $\Delta$ of $\tau \sm \nu_{i, \tau}$,  the map $\psi_i$ restricts to an $\ep$-rough isometric embedding on $\bd \Delta \cap \pt \Delta_\ep$ into the corresponding component $\Delta'$ of $\sigma_i \minus \nu_i$ with respect to the path metrics on the ideal triangles $\Delta$ and $\Delta'$.
\end{lemma}

Next we extend $\psi_i\cn \nu_{i, \tau} \to \nu_i$ to $\tau \to \sigma_i$ by extending $\psi_i$ to the interior each component $\Delta$ of $\tau \sm \nu_{i, \tau}$ in a natural way. 
The ideal triangle  $\Delta$ contains a unique inscribed circle, which is tangent to each edge of $\Del$ at a single point.
Then, by connecting those tangency points, we obtain a hyperbolic triangle inscribed in $\Del$.
Each component of $\Del$ minus the inscribed triangle is a hyperbolic triangle $\check{\Delta}$ with a single ideal vertex $v$. 
Then $\check{\Delta}$ has two edges of infinite length sharing $v$, and a point of one edge corresponds to a point on the other edge so that a horocycle centered at $v$  passes through both points. 
By connecting all pairs of such corresponding points by geodesic segments, 
we obtain  a foliation of $\check{\Delta}$  by geodesic segments.
Then continuously extend $\psi_i$, which is so far defined on $\pt \Delta$, to $\check{\Delta}$ so that $\psi_i$ linearly takes each such geodesic segment connecting points on $\bdr \Delta$ to geodesic segments connecting $\psi_i$-images of the points on $\pt \Delta'$ (see Figure \ref{10-1-12}).

\begin{figure}[h]
\begin{overpic}[trim = 10mm 10mm 10mm 10mm, clip, width = 4in,
]{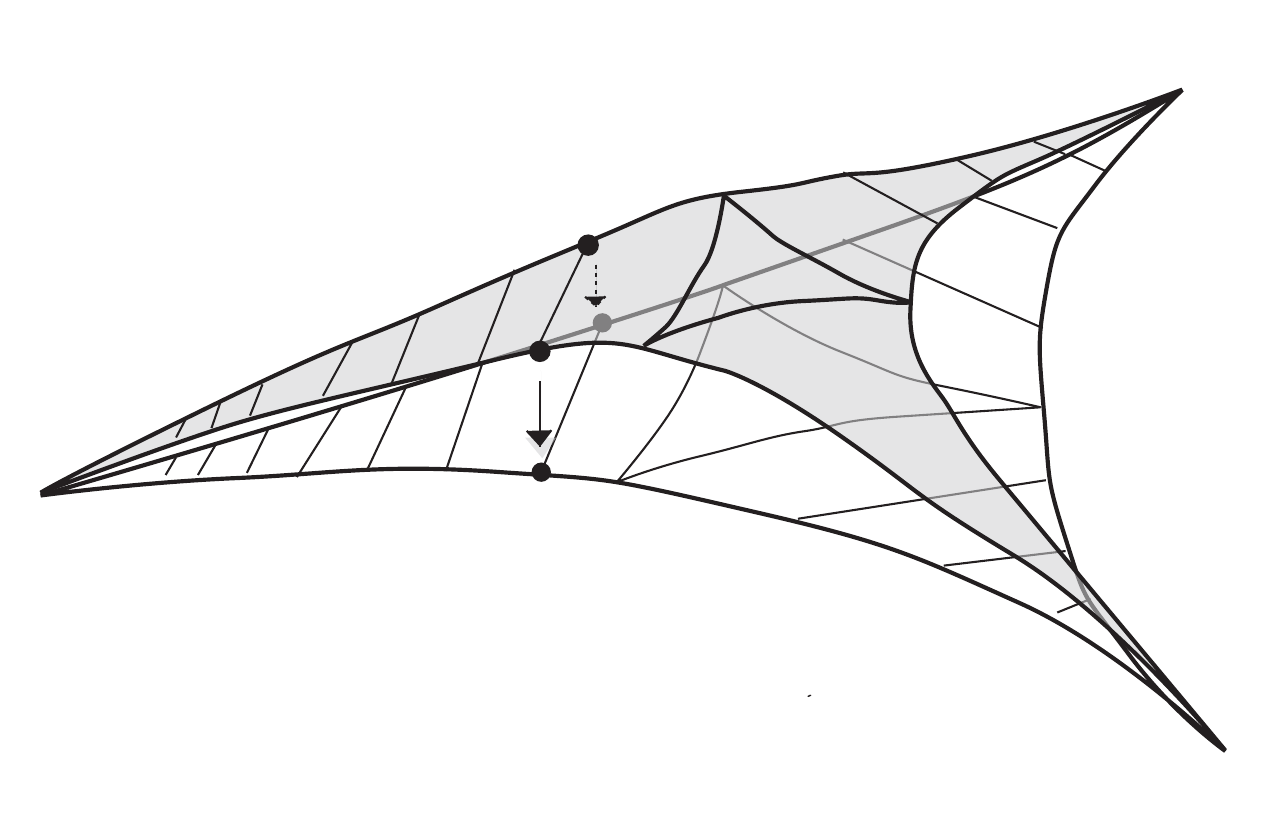}
      \end{overpic}
\caption{}\label{10-1-12}
\end{figure}

If $\ep > 0$ is sufficiently small, then the inscribed triangle in $\Del$ is contained in the $\ep$-thick part $\Del_\ep$ for all components $\Del$ of $\tau \minus \nu_{i, \tau}$.
Thus, by Lemma \ref{9-14-12no1}, we can further extend $\psi_i$ to  the inscribed triangle of $\Delta$ so that  $\psi_i$ restricts to an $\ep$-rough isometric embedding on $\Del_\ep$ into $\Del'$.

For every $\ep > 0$, if $i$ is sufficiently large, then
$\beta| \td{\nu}_{i, \tau}$ and $\beta_i| \td{\nu}_i$ are $\ep$-close pointwise via $\td{\psi}_i\cn \td{\tau} \to \td{\sigma}_i$ since $\psi_i| \nu_{i, \tau}$ is defined using  Corollary \ref{9-29-1}. 
In addition, for every $\zeta > 0$, if $\ep > 0$ is sufficiently small, then $\til{\tau}\, (\cong \h^2)$ is covered by the $\zeta$-neighborhoods of the $\ep$-thick parts of the ideal triangles of $\til{\tau} \minus \til{\nu}_{i, \tau}$.  
For every $\ep > 0$, if $i$ is sufficiently large, then $\beta$ and $\beta_i \cc \td{\psi}_i$ are also  $\ep$-close in the $C^0$-topology via $\td{\psi}$.

For very $\ep > 0$, if $i$ is sufficiently large, then the $\ep$-neighborhood of $\nu$ contains $\nu_{i, \tau}$ in $\tau$.
In the complement of the $\ep$-neighborhood of $\td{\nu}$, $\beta$ and $\beta_i \cc \td{\psi}_i$ are totally geodesic.
We can in addition assume that $\beta$ and $\beta_i \cc \td{\psi}_i$  are $\ep$-close, moreover, in the $C^1$-topology, for sufficiently large $i$.
 
Note that this $C^1$-convergence is weaker than that in Theorem \ref{12-19}, since we have enlarged each $\nu_i$ to a maximal lamination. 
The $\ep$-neighborhood of the extended lamination $\nu_{i, \tau}$ may be bigger than the $\ep$-neighborhood $N_{i, \ep}$ of the original lamination $\nu_{i, \tau}$. 
However, since $\beta$ and $\beta_i \cc \til{\psi}_i$ are totally geodesic in the complement of $N_{i, \ep}$, it is easy to make it  $C^1$-convergence  there for sufficiently large $i$ by a small perturbation.
 
 Thus, it only remains to show:   
\begin{proposition}\Label{061313n1}
For every $\ep > 0$, if $i \in \N$ is sufficiently large, then $\psi_i\cn \cm \tau \to \sigma_i$ is an $\ep$-rough isometry.
\end{proposition}

\proof
Let $x$ be a point of $|\td{\nu}|$, and let $\ell$ be the leaf of $\td{\nu}$ containing $x$. 
Consider a (totally geodesic) hyperbolic plane $H$ of $\h^3$ transversally intersecting the geodesic $\beta(\ell)$ at $\beta(x)$.  
Then, by the transversality, there is a neighborhood $a$ of $x$ in $\beta^{-1}(H)$ homeomorphic to an arc,  which we call an {\it orthogonal arc} through $x$

If a sequence of leaves of $\td{\nu}$ converges to $\ell$, then accordingly their $\beta$-images are geodesics in $\h^3$ converging to $\beta(\ell)$ uniformly on compacts. 
Thus, for every $\ep > 0$, if $a$ is sufficiently short then,   for  a stratum $R$ of $(\h^2, \td{\nu})$ which intersects $a$,  the angle $\angle_{\h^3}(\beta(R), H)$ is $\ep$-close to $\angle(H, \beta(\ell))$.

Since $\lam$ has measure zero in $\tau$ and the pleated surface $\beta\col \h^2 \to \h^3$ preserves the length of paths,  the length of $a$ is equal to the total length of the arcs $a \minus |\td{\nu}|$.

In particular, $a$ intersects $\ell$ only at $x$ and it is ``transversal'' in the sense that there is  a $\del > 0$ such that, if $s$ is a geodesic segment in $\h^2$ with its endpoints on different components of  $a \minus x$, then $s$ intersects $\ell$ transversally at an angle of more than $\del$.

We have shown that $\beta_i \circ \td{\psi}_i$ converges to $\beta$ as $i \to \In$ uniformly everywhere in the $C^0$-topology  and pointwise almost everywhere in the $C^1$-topology.
By this convergence, if $i$ is sufficiently large and $\td{\psi}_i(a)$ intersects a stratum $R_i$ of $(\h^2, \til{\nu}_i)$, then $\beta_i(R_i)$ is transversal to $H$ and $\angle(\beta_i(R_i), H) > \del$ for some fixed $\del > 0$.

Therefore, for sufficiently large $i \in \N$,  there is an arc $b_i$ embedded in $\beta_i^{-1}(H)  \st \h^2$ such that  $\beta_i|b_i$ converges to $\beta|a$ in the $C^0$-topology uniformly. 
Note that, since $\nu_i$ on $\sigma_i$ and $\lam$ on $\tau$ have measure zero, $\beta_i|b_i$ and $\beta|a$ are almost everywhere smooth. 
Thus the  uniform convergence $\beta_i | b_i \to \beta| a$ is moreover in the $C^1$-topology almost everywhere. 
Similarly  the length of $b_i$ is the sum of the lengths of the segments of $b_i \minus |\td{\nu}_i|$.
\begin{lemma}\Label{061313n2}
$\length(b_i)$ converges to $\length(a)$ as $i \to \In$.
\end{lemma}

\begin{proof}

Since $\beta_i|b_i$ converges to $\beta|a$ and pleated surfaces $\h^2 \to \h^3$ preserve length, for every $\ep > 0$, we have  $\length(a)  < \ep + \length(b_i)$ for sufficiently large $i$.
Thus it suffices to show the opposite $\length(b_i)  < \ep + \length (a)$ for sufficiently large $i$.

Recall that $\psi_i\cn \tau \to \sigma_i$ is a marking-preserving homeomorphism taking $\nu_{i, \tau} \to \nu_i$. 
Since the endpoints of $a$ are in the complement of $\nu$, for sufficiently large $i$, the components of $\h^2 \minus \til{\nu}_{i, \tau}$ intersecting $a$ bijectively correspond to the components  of $\h^2 \minus \til{\nu}_i$ intersecting $b_i$.
Therefore, there is a homeomorphism $\eta_i\cn a \to b_i$ such that, if $\Delta$ and $\Delta_i$ are corresponding complementary ideal triangles of  $\td{\nu}_{i, \tau}$ and $\td{\nu}_i$, respectively, then $\eta_i$ takes the arc $a \cap \Delta$ to  the arc $b_i \cap \Delta_i$ homeomorphically.  

Let $\hat{a}_i$ be the union of arcs of $a \minus |\td{\nu}_{i, \tau}|$ intersecting  the $\ep$-thick part of $\h^2 \minus \td{\nu}_{i, \tau}$. 
Then $\hat{a}_i$ is union of finitely many disjoint arcs. 
Let $\check{a}_i = a \minus \hat{a}_i$. 
For every $\ep > 0$, if $i$ is sufficiently large, then $\psi_i$ is an $\ep$-rough isometry in the $\ep$-thick part of $\tau \minus \nu_{i, \tau}$. 
Therefore, if $i$ is sufficiently large, $\eta_i$ changes the total length of $\hat{a}_i$ by at most $\ep$. 

For all $i$, we have  $\Area(\tau) = \Area(\sigma_i)$. 
For all $\ep > 0$ and $\del > 0$,  if $i$ is sufficiently large, $\psi_i$ changes the total area of the $\ep$-thick part $\tau \minus \nu_{i, \tau}$ by at most $\del$.

There is a $\theta > 0$ such that,  for sufficiently large $i$, 
\begin{itemize}
\item If $\Delta$ is  a component of  $\h^2 \minus \td{\nu}$ and $a$ intersects a boundary geodesic $\ell$ of $\Delta$,  then the angle between $a \cap \Delta$ and $\ell$  is at least $\theta$.
\item   if $\Delta_i$ is  a component of  $\h^2 \minus \td{\nu}_i$ and a boundary geodesic $\ell_i$  of $\Delta_i$ intersects $b_i$, then the angle between $b_i \cap \Delta_i$ and $\ell_i$ is  at least $\theta$.  
\item If $\Delta_{i, \tau}$ is  a component of  $\h^2 \minus \td{\nu}_{i, \tau}$ and $a$ intersects a boundary geodesic $\ell$ of $\Delta_{i, \tau}$,  then the angle between $a \cap \Delta_{i, \tau}$ and $\ell$  is at least $\theta$.
\end{itemize}
For every $\zeta > 0$, if $i$ is sufficiently large, then the $\psi_i$ takes the $\ep$-thick part of $\tau \minus \nu_{i, \tau}$  into the $(\ep - \zeta)$-thick part of $\sigma_i \minus \nu_i$ and the $\ep$-thin part of $\tau \minus \nu_{i, \tau}$ maps into $(\ep + \zeta)$-thick part of $\sigma \minus \nu_i$.
Therefore, for every $\del > 0$, if $\ep > 0$ is sufficiently small and  $i \in \Z_{>0}$ is sufficiently large, then the length of  $\check{a}_i$ is $\del$-close to the length of  $\eta_i(\check{a}_i)$, since otherwise $\psi_i$ must increase the total area of the $\ep$-thin part of $\tau \minus \nu$ some definite amount, such that $\Area(\sigma_i) > \Area(\tau)$; this is a contraction.
Therefore for every $\ep > 0$, if $i$ is large enough,  $\length(b_i) < \length(a) + \ep$.
\end{proof}

\begin{lemma}\Label{061413n2}
For  every $p \in \td{\tau}$ and  $\ep > 0$, there is a neighborhood $U$ of $p$ in $\td{\tau}$ such that, if $i \in \N$ is sufficiently large, then $\til{\psi}_i(U) \st \td{\sigma}_i$ has diameter less than $\ep$.
\end{lemma}

\begin{proof}
First suppose that $p$ is in  $\til{\tau} \minus \td{\nu}$. 
Let $\Delta$ is the component of $\til{\tau} \minus \til{\nu}$ containing $p$. 
Then take a sufficiently small closed ball centered at $p$ so that it is contained in $\Delta$. 
Let $U$ be the interior of the closed ball, which is an open ball centered at $p$.
Then for sufficiently large $i$, $U$ is  contained also in a component $\Delta_i$ of $\til{\sigma}_i \minus \til{\nu}_{i, \tau}$.
There is a $\del > 0$ such that $U$ is contained in the $\del$-thick part of $\Delta_i$. 
Thus, for every $\ep > 0$, if  $i$ large enough,  $\psi_i$ is an $\ep$-rough isometry near $p$. 
Thus if $U$ is a sufficiently small neighborhood of $p$, then $\td{\psi}_i(U)$  has diameter less than $\ep$.

Next suppose that $p$ is on a leaf $\ell$ of $\td{\nu}$.  
Then we construct a small ``rectangular'' neighborhood bounded by geodesic segments disjoint from $\til{\nu}$ and curves, as above, mapping into hyperbolic planes orthogonal to $\beta(\ell)$  by $\beta$. 
For $\del > 0$, let $x_1$ and $x_2$ be the points on $\ell$  that have distance $\del$ from $p$, so that $p$ bisects the geodesic segment $[x_1, x_2]$.
Let $H_1$ and $H_2$ be the hyperbolic planes in $\h^3$ that are orthogonal to $\beta(\ell)$ at $\beta(x_1)$ and $\beta(x_2)$, respectively. 
Given $\del > 0$, let $a_i$ be an orthogonal curve on $\til{\tau}$ passing through $x_i$ for each $i = 1,2$ such that
\begin{itemize}
\item $\length_{\til{\tau}}(a_i) < \del$,
\item  $\beta(a_i)$ is contained in $H_i$. 
\item the corresponding endpoints of $a_1$ an $a_2$ are in the same component of $\til{\tau} \minus \til{\nu}$.
\end{itemize}
By the third condition,  the corresponding endpoints of $a_1$ and $a_2$ are in the complements of $\til{\tau} \minus \til{\nu}$.
Thus let  $b_1$ and $b_2$ be geodesic segments in $\til{\tau} \minus \til{\nu}$ that connect the corresponding endpoints of $a_1$ and $a_2$.

Then   $\length(b_1) \to 0$ and $\length(b_2) \to 0$ as $\del \to 0$.
Then let $U_{\del}$ be the rectangular neighborhood of $p$ bounded by $a_1, a_2, b_1, b_2$, so that $p \in U_\del$.

We claim that, for every $\ep > 0$,  the diameter of  $\psi_i(U_{\del})$ is less than $\ep$, if $\del \to 0$ is sufficiently small and $i$ is sufficiently large. 
Since $\psi_i$ are homeomorphisms, it is suffices to show that the $\psi_i$-images of the edges $a_1, a_2, b_1, b_2$ have length less than $\ep$. 

Since $\length(a_1), \length(a_2) < \del$, by Lemma \ref{061313n2}, if $\del > 0$ is sufficiently small and $i$ is sufficiently large, then $\psi_i(a_1)$ and $\psi_i(a_2)$ have length less than $\ep$.

The geodesic segments $b_1$ and $b_2$ are disjoint from $\til{\nu}_{i, \tau}$ for sufficiently large $i$. 
Therefore, for every $\ep > 0$, the restrictions of $\psi_i$ to $b_1, b_2$ are $\ep$-rough isometric embeddings for sufficiently large $i$. 
Hence, if $\del > 0$ is sufficiently small and $i$ is sufficiently large, then $\psi_i(b_1)$ and $\psi_i(b_2)$ have length less than $\ep$.
\end{proof}

\begin{proposition}\Label{061413n1}
For every $p, q \in \td{\tau} (\cong \h^2)$ and $\ep > 0$, if $i \in \N$ is sufficiently large, then 
$$-\ep < \length_{\til{\tau}} [p, q] - \length_{\til{\sigma}} [\psi_i(p), \psi_i(q)] < \ep.$$
\end{proposition}
\begin{proof}

First Suppose that $p$ and $q$ are in the interior of a single stratum $\Delta$ of $(\til{\sigma}, \td{\nu})$. 
Then the assertion holds true since, given $\ep > 0$,  $\td{\psi}_i$ is a $(1 + \ep, \ep)$-quasiisometric embedding on $\Delta$ for sufficiently large $i$ .

Second suppose that $p$ and $q$ are contained in a single leaf  $\ell$ of $\til{\nu}$.
For every $\ep > 0$, if $i$ is sufficiently large, then,  since $\angle_\tau(\lam, \nu_{i, \tau}) \to 0$, we can
pick orthogonal arcs $a_i$ from $p$ and $a_i'$ from $q$  on $\til{\tau}$ such that $\length(a_i), \length(a_i') < \ep$ and, letting $r_i$ and $s_i$ be the other endpoints of $a_i$ and $a_i'$, such that $r_i$ and $s_i$ are in a single stratum of $(\til{\tau}_i,\til{\nu}_{i, \tau})$,   using Lemma \ref{061313n2}.
Let $b_i$ and $b_i'$ be the arcs on $\til{\sigma}_i$  that correspond to $a_i$ and $a_i'$, as discussed just before Lemma \ref{061313n2}, so that $\beta_i|b_i$ and $\beta_i| b_i'$ are $\ep$-close to $\beta|a_i$ and $\beta|a_i'$, respectively, in hyperbolic planes orthogonal to $\beta(\ell)$.
Then we can in addition assume that 
$\length(b_i), \length(b_i') < \ep$. 
Let $c_i = [\psi_i(r_i), \psi_i(s_i)]$, which is contained in a stratum of $(\til{\sigma}_i, \til{\nu}_i)$. 
Therefore, if $i$ is sufficiently large, then  $| \length (c_i)  - \length[r_i,s_i]  | < \ep$.
Hence, for every $\ep > 0$, since we can assume that the lengths of $a_i, a_i', b_i, b_i'$ are less than $\ep$ for sufficiently large $i$, we have $- \ep < \length [p, q] - \length [\psi_i(p), \psi_i(q)]  < \ep$.

Last suppose that $p, q$ are in different strata, so that $[p, q]$ transversally intersects $\td{\nu}$. 
Since $\beta_i \cc \td{\psi}_i$ converges to $\beta$ as $i \to \In$ and $\beta, \beta_i$ preserve length of curves, thus  $\length[p, q] < \ep +  \length[\psi_i(p), \psi_i(q)]$ for sufficiently large $i$.

For each $x \in [p,q] \cap \td{\nu}$,  let $\ell_x$ be the leaf of $\td{\nu}$ containing $x$.
Then, as discussed above, there is an orthogonal $a_x$ passing through $x$ so that $\beta(a_x)$ is contained in the hyperbolic plane, $\h^3$, orthogonal to the geodesic  $\beta(\ell_x)$ at $\beta(x)$.
We can in addition assume that the endpoints of $a_x$ are in the complement of $\til{\lam}$.

Next, using orthogonal curves, we pick a curve approximating the geodesic segment $[p,q]$ that intersects $\nu$ almost ``orthogonally''. 
Namely, take finitely many points $x_1, \dots, x_n$ on $[p, q] \cap \til{\nu}$, and pick orthogonal curves $a_1, \dots, a_n$ so that an endpoint of $a_k$ and an endpoint of $a_{k + 1}$ are in the interior of a single stratum of $(\til{\tau}, \til{\nu})$ for each $k$. 
Then let $c_k$ be the geodesic segments connecting the endpoints. 
Taking an union of such $a_j$ and $c_k$, we can construct a curve $\ap$ connecting $p$ to $q$, such that $\ap \cap \nu \sub \cup_{j = 1}^n a_j$ and that $\ap \minus  \cup_{j = 1}^n a_j$ is a union of disjoint geodesic segments.  
For every $\ep > 0$, taking large $n$ so that the orthogonal curves  $a_j$ are sufficiently short, we can  in addition  assume that 
\begin{itemize}
\item $\ap$ is $\ep$-close to $[p,q]$ in the Hausdorff metric,
\item $\Sigma_j \length(a_j) <\ep$, and
\item $-\ep <  \length [p,q] - \Sigma_i \length(c_k) < \ep$.
\end{itemize}
(For the last assertion, consider the nearest point projection of $\ap$ to $[p,q]$).
Thus, for every $\ep > 0$, there is such an approximating curve $\ap$ with $-\ep < \length(\ap) - \length[p,q] < \ep$.
Therefore, in order to show $\length[p, q] + \ep >    \length[\td{\psi}_i(p), \td{\psi}_i(q)]$  for sufficiently large $i$,  it suffices to find a curve $\ap_i$ on $\td{\sigma}_i$ connecting $\td{\psi}_i(p)$ to $\td{\psi}_i(q)$ such that $-\ep < \length(\ap_i) - \length(\ap) < \ep$.

For each orthogonal segment $a_j$,  as defined for Lemma \ref{061313n2}, there is a corresponding curve $b_{i, j}$ on $(\h^2, \td{\nu}_i)$, such that $\beta|a_j$ and $\beta_i| b_{i,j}$ are contained in a single hyperbolic plane and $\beta_i| a_{i,j}$ converges to $\beta| a_j$ as $i \to \In$.
Then, by  Lemma \ref{061313n2},  $\length(b_{i,j}) \to \length(a_j)$ as $i \to \In$. 
Thus $\Sigma_j \length(b_{i,j}) < \ep$ for sufficiently large $i$. 
Then, for each $j$, we can connect the endpoints of $a_{i,j}$ and $b_{i, j+1}$ by geodesic segments $c_{i,j}$ in the complement of $\td{\nu}_i$, to obtain a curve $\ap_i$ connecting $\td{\psi}_i(p)$ and $\td{\psi}_i(q)$.
Then, for every $\ep > 0$, if $i$ is sufficiently large, then the endpoints of  $c_{i, j}$ are $\ep$-close to the $\td{\psi}_i$-image of the endpoints of $c_j$ for all $j$. 
Since $\td{\nu}_{i,\tau} \to \td{\nu}$,  the segment $c_ j$ is disjoint from  $\td{\nu}_{i, \tau}$ for sufficiently large $i$.
Therefore  $\length(c_{i, j})$ is $\ep$-close to  $\length(c_j)$ for all $j$. 
Since $\ep > 0$ is arbitrary,  $-\ep < \length(\ap) - \length(\ap_i) < \ep$ for sufficiently large $i$.  
\end{proof}
Proposition \ref{061313n1} immediately follows from Proposition \ref{061413n1} and Lemma \ref{061413n2}.
\Qed{061313n1}

\section{Rectangular projective structures}\Label{LC}   

\subsection{Projective structures on rectangle supported on cylinders}\Label{10-5-1}

Let $c$ be a round circle on $\rs$. 
A geodesic $g$ in $\h^3$  is  an  \textit{axis} of $c$ on $\rs$ if $g$ is orthogonal to the (totally geodesic) hyperbolic plane in $\h^3$ bounded by $c$.

Let $\mA$ be a \textit{round cylinder} in $\rs$, that is, $\mA$ is bounded by disjoint round circles $c_{-1}$ and $c_1$.
Then the \textit{axis} of $\AA$ is the unique geodesic in $\h^3$ that is orthogonal to both hyperbolic planes bounded by $c_{-1}$ and $c_1$.
Then, there is a unique foliation ${\mF}_{\mA}$ of  ${\mA}$  given by the continuous family of  round circles, $\{c_t\}_{t \in [-1, 1]}$, sharing the axis $g$.
We  call it the \textit{circular foliation} on $\mF_{\mA}$.
Then each round circle $c_t$ has a smooth metric invariant under by elliptic isometries of $\h^3$ fixing $g$.
 It is unique up to scaling, and thus we normalize it so that the length of $c_t$ is $2 \pi$ for all $t \in [-1,1]$ (\textit{canonical metric}).

\begin{definition}\Label{10-2-1}
Let $C = (f, \rho)$ be a projective structure on a simply connected surface $F$.
(In particular $\rho$ is trivial.)
Let $e$ be a simple curve on $C$. 
Then we say that $e$ is \underline{supported} on the round cylinder ${\mA}$ if $f$  embeds $e$ properly into ${\mA}$ so that $e$ transversally intersects all leaves $c_t$ of $\mA$. 
\end{definition}

Let $R$ be a rectangle, and let $e_1, e_2, e_3, e_4$ denote the edges of $R$, cyclically indexed along  $\pt R \,(\cong \s^1)$; Figure \ref{fRectangle}. 
A projective structure $C = (f, \rho)$ on $R$ is \textit{supported} on the round cylinder ${\mA}$
if  
\begin{itemize}
\item[ (i)] $f$ immerses $e_1$ and $e_3$ into $c_{-1}$ and $c_1$, respectively, and
\item[(ii)] $e_2$ and $e_4$ are supported on ${\mA}$.
 \end{itemize}
 Then we say that $C$ is \textit{supported} on the round cylinder ${\mA}$ and  \textit{bounded} by the arcs $f| e_2$ and $f| e_4$  supported on  $\mA$.

\begin{figure}
\begin{overpic}[scale=.8,
] {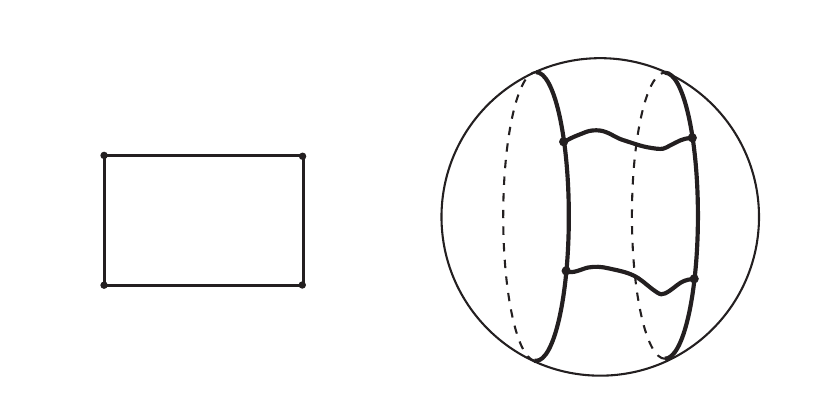}

        \put(7,24 ){$e_1$}  
        \put(23, 13){$e_2$}  
        \put( 38, 24){$e_3$}  
        \put( 23, 34){$e_4$}  
          \put( 23,23 ){$R$}  
        \put(63 ,24 ){$c_{-1}$}  
        \put(85 , 24){$c_{1}$}  
        \put( 71,13 ){$m_2$}  
        \put(71, 36){$m_4$}  
      \end{overpic}
\caption{}\label{fRectangle}
\end{figure}

Then, if $C$ is supported on $\AA$, we can pull back, via $f$, the circular foliation ${\mF}_{\mA}$ on $\mA$ to a circular foliation $\FF_{C}$ on the rectangle $C$. 
Each leaf of $\FF_C$ immerses into a closed leaf of $\FF_\AA$,  and thus it has a metric obtained by  pulling back the canonical metric of the closed leaf. 
Then we say that the \textit{height} of $C$ is $\ep$-close to $W$ for some $W > 0$, if every leaf of $\FF_C$ has length $\ep$-close to $W$.

 \subsection{Grafting a rectangle  supported on a cylinder} (Compare \cite[\S 3.5]{Baba12}.)
Let $C$ be a projective structure on the rectangle  $R$ supported on the round cylinder $\mA$ as above.
Let $m$ be a simple arc on $C$ supported on $\mA$.
Then $m$ is an arc properly embedded in $\mA$.
Then, similarly to grafting a projective surface along a loop (\S \ref{grafting})), we can combine two projective structures $C$ and $\mA$, by cutting and pasting along $m$, and obtain a new projective structure on  $R$ supported on $\mA$. 
Namely,  we can pair up the boundary arcs of  $C \sm m$ and the boundary arcs of $\mA \sm m$ and isomorphically identify them to create a new projective structure on $R$ supported on $\AA$.
We call this operation the {\it grafting} of $C$ along $m$ and denote this resulting projective structure by $\Gr_m(C)$.
We call $m$ an {\it admissible} arc on $C$.
If there is a multiarc $M$ on $C$ consisting of arcs supported on $\mA$ ({\it admissible multiarc}), then we can graft $C$ along all arcs of $M$ simultaneously and obtain a new projective structure on $R$ supported on $\AA$.
We accordingly denote it by  $\Gr_M(C)$.

\begin{lemma}\label{8-11_1}
Let $C_1$ and $C_2$ be projective structures on a rectangle $R$.
Suppose that they are supported on the same round cylinder and bounded by the same pair of  arcs supported on the cylinder.  
Then, we have either $C_1 = \Gr_M(C_2)$ or $C_2 = \Gr_M(C_1)$ for some admissible multiarc $M$.
Furthermore, the multiarc $M$ is unique up to an isotopy of $M$ on $R$ through admissible multiarcs.

Moreover the number of arcs of $M$ times $2\pi$ is equal to the length difference of the corresponding vertical edges of $C_1$ and $C_2$. 
\end{lemma}
\begin{proof}
Let $\mA$ be the round cylinder supporting $C_1$ and $C_2$.
Let $f_1\cn R \to {\mA}$ and  $f_2\cn R \to {\mA}$ be the developing maps of $C_1$ and  $C_2$, respectively. 
Let $\td{{\mA}}$ be the universal cover of ${\mA}$  and  $\Psi\cn \td{{\mA}} \to {\mA}$ be the universal covering map.
Let $m_2$ and  $m_4$ be the simple arcs properly embedded in $\mA$ that bound both $C_1$ and $C_2$ so that $m_2 = f_1(e_2) = f_2(e_2)$ and  $m_4 = f_1(e_4) = f_2(e_4)$.
Pick a lift $\td{m}_4$ of $m_4$ to $\td{{\mA}}$.
Then, for each $k = 1, 2$,  $f_k\cn R \to {\mA}$ uniquely lifts to  $\td{f}_k\cn R \to \td{{\mA}}$ so that  $f_k = \Psi \cc \td{f}_k$ and $\td{f}_k$ embeds $e_4$ onto $\td{m}_4$.
Clearly $\td{f}_k$ is an embedding (although $f_k$ may {\it not} be). 
We see that $\td{f}_k(e_2)$ is a lift of $m_2$ to $\td{{\mA}}$.
Since  projective structures have fixed orientation, $\td{f}_1(e_2)$ and $\td{f}_2(e_2)$ are in the same component of $\td{{\mA}} \sm \td{m}_4$.
If  $\td{f}_1(e_2) = \td{f}_2(e_2)$, then clearly $C_1 = C_2$.
If  $\td{f}_1(e_2) \neq \td{f}_2(e_2)$, then, without loss of generality,  we can assume that $Im(\td{f}_2)$ is a proper subset of $\Im(\td{f}_1)$,  if necessary, by exchanging $C_1$ and $C_2$.
Thus we can naturally regard $\Im(\td{f}_1) \sm \Im(\td{f}_2)$ as  a  projective structure on a rectangle  supported on ${\mA}$, where its developing map is the restriction of  $\Psi$  to $\Im(\td{f}_1) \sm \Im(\td{f}_2)$.
Then its supporting arc are both $m_4$.
Let $d$ be the generic degree of  the developing map of  $\Im(\td{f}_1) \sm \Im(\td{f}_2)$ to $\AA$ (i.e. the degree over a point in ${\mA} \sm m_4$).
Note that $$2\pi d = \length f_1 (e_1) - \length f_2(e_1) =  \length f_1 (e_3) - \length f_2(e_3).$$

Note that the grafting along an admissible arc on $C_2$ increases the length of its vertical edges by $2\pi$. 
Thus, if $M$ is the union of $n$ disjoint admissible arcs on $C_2$, then the length of vertical edges increases by $2\pi n$
Therefore  $Gr_M(C_2) =C_1$ if and only if $n = d$. 
\end{proof}


\subsection{Fat traintracks}\Label{S:traintracks}

Given a rectangle, pick a pair of opposite edges and call them \textit{horizontal edges} and the other edges  \textit{vertical}, to distinguish them. 
\begin{definition}[\cite{Kapovich-01, Penner-Harer-92}]\Label{11-18-1}
Let $F$ be a topological surface. 
A  \linebreak[4] \textit{(fat) traintrack} $T$  on $F$ is a collection of rectangles $R_j\, (j \in J)$ embedded in $F$, called \ul{branches}, such that 
\begin{itemize}
\item $\{R_j\}_{j \in J}$ is locally finite,
\item  branches can intersect only along their vertical edges, and  
\item if $e$ is a vertical edge then either
\begin{itemize}
\item $e$ is (homeomorphically) identified with another vertical edge of a rectangle, 
\item $e$ is  a union of two other vertical edges, which share an endpoint, of some rectangles, or
\item  $e$ is identified with  a segment of another vertical edge containing an endpoint.  
\end{itemize} 
\end{itemize}
Then let $|T| \sub S$ denote the union of the rectangles $R_i$. 
\end{definition}

Let $T = \{R_j\} _{j \in J}$ denote a traintrack on $F$, where $R_j$ are its branches. 
The vertical edges of the branches $R_j$ decompose $T$ into the branches.
The boundary of $|T|$ is the union of the horizontal edges.
If a point of $\pt |T|$ is the common end point of the second possibility for $e$, then it is called a \textit{switch point}.  

Let $\ld$ be a lamination on  $F$.
Then the traintrack $T = \{R_j\}_{j \in J}$ \textit{carries} $\ld$, if
\begin{itemize}
\item the interior of $|T|$ contains $\ld$ and
\item each leaf $\ell$ of $\lam$  is transversal to the vertical edges of $T$ and each component of  $\ell \cap R_j$ is an arc connecting the vertical edges of $R_j$ for each $j \in J$,
\end{itemize}
If, in addition, $R_j \cap \ld \neq \emptyset$ for all $j \in J$, then we say $T$ \textit{fully carries} $\ld$. 
Suppose that $L = (\ld, \mu)$ is a measured lamination carried by $T$.
The \textit{weight} of $L$ on a branch $R_j$ is the transversal measure $\mu$ of a vertical edge of $R_j$; 
we denote it by $\mu(R_j)$.
The weights of branches satisfy some simple equations, called {\it switch conditions}. 

Suppose that $T$ carries two measured laminations $L_1$ and $L_2$.
Then the weights $L_1$ and $L_2$ are nonnegative real numbers on each branch of $T$. 
Thus there is a unique measured lamination $L_1 + L_2$ carried by $T$ such that, the weight of $L_1 + L_2$ on $R_j$ is the sum of the weights of $L_1$ and $L_2$ on $R_j$ for each $j$. 
Suppose that the weight of $L_1$ is at least the weight of $L_2$ on each branch of $T$. 
Then similarly there is a unique measured lamination $L_1 - L_2$ carried by $T$ such that  the weight of $L_1 - L_2$ on $R_j$ is the weight of $L_1$ minus the weight of $L_2$ on $R_j$ for each $j$.

\note{p1-1}
Given $\ep > 0$,  we say $L_1$ is {\it $\ep$-close} to $L_2$ on $T$, if the weight of $L_1$ is $\ep$-close to that of $L_2$ on each branch of $T$. 
We say that  $L_1$ is a {\it good approximation} of $L_2$ on $T$, if $L_1$ is $\ep$-close to $L_2$ for a sufficiently small $\ep > 0$.

We remark that, if a traintrack $T = \{R_j\}_j$ has  weights on its branches satisfying the switch conditions, it corresponds to a unique measured lamination. 
Indeed there is a measured foliation $\FF$ of $|T|$ such that $\FF$ foliates each branch $R_i$ by arcs connecting its vertical edges and the transversal measure of $R_i$ given by $\FF$ realizes  the  weight of $R_i$.
Then by ``straightening" leaves of $F$ fixing a hyperbolic metric on $F$, we obtain a measured (geodesic) lamination. 

Moreover the above addition and subtraction respect the piecewise linear structure on the space of measured laminations. 
In particular,  the set of all possible weight-systems on a traintrack is a  piecewise linear cone in a vector space.

Let $\tau$ be a hyperbolic structure on $F$. 
Then the traintrack $T$ is  \textit{smooth} if all branches are smooth (i.e. the edges of its branches are smooth) and $\pt |T|$ is smooth except at the switch points. 

\begin{definition}\Label{11-12-2}
For $\ep > 0$, a smooth traintrack $T = \{R_j\}_{j \in J}$ on $\tau$ is called \underline{$\ep$-nearly straight} if each branch $R_j$ is $(1 + \ep)$-bilipschitz to some Euclidean rectangle and, at each switch point, the angle of $\pt |T|$ is less than $\ep$.

For $\ep, K > 0$,  $T$ is \underline{$(\ep, K)$-nearly straight}  if $T$ is $\ep$-nearly straight and, for each branch $R_i$ of $T$,   $K$ is less than the length of the horizontal edge of such a Euclidean rectangle corresponding to  $R_i$ (which we call the \underline{length} of $R_i$).
\end{definition}
Such a nearly straight (non-fat) traintrack is introduced in \cite[ch. 8]{Thurston-78}; see also \cite{Brock-00, Minsky92}].
If $\ep > 0$ is sufficiently small,  each branch of an $(\ep, K)$-nearly straight traintrack is hausdorff close to an almost straight curve. 

\section{Decomposition of projective structures by  traintracks}\Label{sec:decomposition}

\begin{definition}
Let $R$ be a branch of  a traintrack $T$ in a projective structure $C = (f, \rho)$  on $S$.
(Recall that a branch is a rectangle.)
Then the branch $R$ is \underline{supported} on a round cylinder $A$ on $\rs$
if  $A$ supports the restriction of $C$ to $R$ so that horizontal edges of $R$  are supported on $A$.
\end{definition}

\begin{lemma}\Label{admissible}
Let $T = \{ R_k \}$ be a traintrack on a projective surface $(S, C)$ such that  each branch $R_k$ is supported on a round cylinder \linebreak[4](\ul{admissible traintrack}).
Note that the circular foliations on $R_k$ yield a circular foliation on $|T|$. 
Then, if a loop $\ell$ is carried by  $T$  and transversal to the circular foliation on $|T|$, then $\ell$ is admissible.
 
\end{lemma}

\begin{proof}
Let $\td{\ell}$ be a lift of $\ell$ to $\td{S}$.
Let $\td{T}$ be the lift of $T$ of $\td{S}$.
Let $\td{R}_{k \in \Z}$ denote the branches of $T$ intersecting $\td{\ell}$, so that $\td{R}_{k}$ and $\td{R}_{k + 1}$ are adjacent.
Let $C = (f, \rho)$, where $f$ is the developing map  and $\rho$ is the holonomy. 
Then $f$ injects $\ell \cap \td{R}_k$ for each $k$.
The supports of $\td{R}_{k}$ have disjoint interiors. 
Thus $f$ embeds $\td{\ell}$ into $\rs$. 
Since this embedding extends the endpoint of  $\td{\ell}$, taking them to distinct points, $\rho(\ell)$ is loxodromic. 
\end{proof}

The following proposition will yield the traintracks on projective surfaces in Theorem A and Theorem B.

\begin{proposition}\Label{6-28-12no1}
Let $C_i \cong (\tau_i, L_i),~   i \in \Z_{> 0},$ be a sequence of projective structures on $S$ with fixed holonomy $\rho$, and let $f_i$ be the developing map of $C_i$.
Let $\LL_i$ be the canonical lamination on $C_i$, which  descends to $L_i$ by the collapsing map $\kap_i \cn C_i \to \tau_i$.

Suppose that $\tau_i$ converges to $\tau_\In$ in $\TT$ as $i \to \infi$, and there is a geodesic lamination $\lam_\In$ (on $\tau_\In$ ), such that, for every $\ep > 0$,  the $\ep$-neighborhood of $|\lam_\In|$ contains $|L_i|$ for sufficiently large $i$. 

Then there are a traintrack $\mathcal{T} = \{\RR_k\} _{k = 1}^n$ on $S$  (which depends only on $\rho, \tau_\In, |\lam_\In|, \ep$) and a homeomorphism  $\phi_i\cn S \to C_i$ for every $i$, such that for every $\ep > 0$ if $i, j$ are sufficiently large, then
\begin{itemize}
\item[(I)]  $\phi_i(\mathcal{T})$ carries $\LL_i$, and $\phi_i(\mathcal{T})$ descends, by $\kap_i$, to an $(\ep, K)$-nearly straight traintrack $T_i$ on $\tau_i$ carrying both $L_i$ and $\lam_\In$, where $K > 0$ is an arbitrarily fixed constant that is less than one third of a shortest closed leaf on $\lam_\In$ (if $\lam_\In$ contains no closed leaves, then $K > 0$ is arbitrarily).

\item[(II)]
$f_i$ and $f_j$ induce an isomorphism from  $C_i \sm \phi_i(\mathcal{T})$  to $C_j \sm \phi_j(\mathcal{T})$ as projective surfaces; thus we can assume that  $\phi_i \cc \phi_j\iv\cn C_j \to C_i$ induces this isomorphism.

\item[(III)]
\begin{itemize}
\item[(i)] For each branch $\td{\RR}_k$  of $\td{\mathcal{T}}$, there exists a round cylinder $\til{\AA}_k$ on $\rs$ that supports its corresponding rectangle $\td{\phi}_i(\td{\RR}_k)$ in $\td{C}_i$ for every sufficiently large $i$, where $\td{\phi}_i\cn \td{S} \to \td{C}_i$ is the lift of $\phi_i$.
\item[(ii)] Moreover, if $a$ is a vertical edge of $\til{\phi}_i(\td{\RR}_k)$ then the length of $a$ (\S \ref{10-5-1}) is $\ep$-close to the transversal measure of  $\phi_i(\RR_k)$ with respect to $\LL$ (\S \ref{S:traintracks}), where $\RR_k$ is the branch of $\mathcal{T}$ that lifts to $\td{\RR}_k$.  
\end{itemize} 
\end{itemize}
\end{proposition}

\begin{remark}
To be precise, by ``descends'' in (I), we mean that $\kap_i$ takes $\phi_i(\mathcal{T})$ to $T_i$ up to an $\ep$-small perturbation of vertical edges as given in Theorem \ref{6-17-12}. Yet  $\kap_i$ takes $|\phi_i(\mathcal{T})|  \st C$ exactly onto $|T_i| \st \tau_i$.
\end{remark}

In particular, by taking all $C_i$ to be a fixed projective structure, we obtain

\begin{corollary}\Label{AdmissibleTraintrack}
Let $C \cong (\tau, L)$ be a projective structure on $S$, and let $\lam$  be a geodesic lamination $\lam$ on $\tau$ containing $|L|$.
Let $\kap\col C \to \tau$ be its collapsing map.
Let $\Lambda$ be the lamination on $C$ that descends to $\lam$ by $\kap$. 
Then for every $\ep > 0$,  there is an admissible traintrack $\TT$ on $C$  carrying  $\Lambda$, so that it descends to an $\ep$-nearly straight track on $\tau$ by $\kap$ up to an $\ep$-small perturbation of vertical edges.  
\end{corollary}

The rest of \S \ref{sec:decomposition} is the proof of Proposition \ref{6-28-12no1}.

{\it Outline of the proof of Proposition \ref{6-28-12no1}.} 
Construct a nearly straight traintrack $T_\infi$ on $\tau_\infi$ carrying $\lam_\infty$ (Lemma \ref{8-2}).
Indeed $T_\infi$ yields all other traintracks in the proposition. 
There is a $\rho$-equivariant pleated surface realizing $(\tau_\infi, \lam_\infi)$, and  it is the limit of  the $\rho$-equivariant pleated surface for $C_i$ (see Lemma \ref{6-20-12no2}). 
By this convergence,  for sufficiently large $i$,  there is a  corresponding nearly straight traintrack $T_i$ on $\tau_i$.  
The traintrack $\mT_i$ in the proposition is obtained by pulling back $T_i$ by the collapsing map $\kap_i \col C_i \to \tau_i$ and perturbing of vertical edges a little bit ( see Proposition \ref{7-7-12no1}. 
The estimate of the lengths of vertical edges is given in \S \ref{11-12}.

\begin{lemma}\Label{6-20-12no2}
Let $C_i \cong (\tau_i, L_i)$ be a sequence of projective structures on $S$ with fixed holonomy $\rho$, such that $\tau_i$ converges to $\tau_\infi$.
For each $i$, let $\beta_i\col \h^2 \to \h^3$ be the $\rho$-equivariant pleated surface corresponding to $C_i$. 

Suppose that there is a geodesic lamination $\lam_\In$ on $\tau_\infi$ such that, given any $\ep > 0$, the $\ep$-neighborhood of $|\lam_\In|$ contains (the geodesic representative of) $|L_i|$ for all sufficiently large $i$.
Then there is a $\rho$-equivariant pleated surface $\beta_\In$ realizing the pair $(\tau_\In, \lam_\In)$,  and  $\beta_i$ converges to $\beta_\In$. 
This convergence is uniform in the $C^0$-topology on $S$ and uniform on compacts in the $C^1$-topology  in the complement of $| \lam_\infi |$.
\end{lemma}

\begin{remark}
In this lemma, there may be a sublamination of $\lam_\In$ that realizes $\beta$ as well.
\end{remark}

\proof
First we show that $\lam_\In$ is realizable by a $\rho$-equivariant pleated surface. 
The assumption on $\lambda_\In$ and $L_i$ implies that $\angle_{\tau_\In}(\lam_\In, L_i) \to 0$ as $i \to \In$. 
Thus, for every $\ep > 0$, if $i$ is sufficiently large, then  $\beta_i$ restricts to a $(1 + \ep)$-bilipschitz map on every leaf of $\til{\lam}_\infi$ on the universal cover of $\tau_i$ (Proposition \ref{8-28-1}). 

Since $\ep > 0$ is arbitrary, there a unique $\rho$-equivariant map $\beta_\In$ taking leaves of  $\td{\lam}_\In$ to geodesics in $\h^3$, so that  $\beta_i$ converges to $\beta_\If$ on $\td{\lam}_\If$  uniformly as $i \to \In$.
Let $\td{\lam}_i$ be the total lift of $\lam_i$ to $\h^2$.

\begin{claim}\Label{sQIcomplements}
For every $\ep > 0$, if $i$ is sufficiently large, then the restriction of $\beta_i$ to each component $R$ of $\h^2 \minus |\td{\lam}_\In|$ is a $(1 + \ep, \ep)$-quasiisometric embedding. 
\end{claim}

\begin{proof}
The proof is similar to arguments in  \cite[\S 7]{Baba-10}.
For every $\del > 0$,  if $i$ is sufficiently large, then, for each component $R$ of $\h^2 \minus |\til{\lam}_\infi|$,  the restriction $\beta_i | R$ is totally geodesic away from the $\del$-neighborhood of $\bdr R$.  
Moreover, if $\del > 0$ is sufficiently small and $i \in \Z_{> 0}$ is sufficiently large, then we can in addition assume that, for every geodesic segment $s$ in the $\del$-neighborhood of $\bdr R$ in $R$, either $\length_{\h^2}(s) < \ep$ or $\angle(s, \til{L}_i) $ is quite small, so that $\beta_i | s$ is a $(1 + \ep, \ep)$-quasiisometric embedding.  

Then the claim immediately follows. 
\end{proof}

The $\ep > 0$ is arbitrary in Claim \ref{sQIcomplements}, and therefore  $\beta_\In$, which is defined on $\pt R$, must extends to $R$ so that $\beta_\In | R$ is a totally geodesic isometric embedding.
Thus $\beta_i$ uniformly converges to $\beta_\In$ realizing $(\tau_\In, \lam_\In)$ in the $C^0$-topology.  
Since $\beta_i$ and $\beta_\infi$ are totally geodesic in the complements of $\lam_i$ and $\lam_\infi$, respectively, the convergence is, in the $C^\infi$-topology, uniform on compacts in the complement of $\til{\lam}_\infi$.
\Qed{6-20-12no2} 

\begin{remark}\Label{111112}
The bending map $\beta_i$ has a natural normal vector field at all points $x$ away from the lifts of closed leaves of $L_i$: Namely, 
it is the direction of  the ray from $\beta(x)$ to the $f_i$-image of the point on $\td{S}$ corresponding to $x$. 
Then the limit of this normal vector field yields a normal vector field on $\beta_i$ away from $|\td{\lam}_\In|$.
\end{remark}

\subsection{Construction of Traintracks}
\begin{lemma}\Label{8-2}
Let $|\nu|$ be a geodesic lamination on a hyperbolic surface $\sigma$ homeomorphic to $S$.
Then there exists a $K > 0$, such that, for every $\ep > 0$, there exists an $(\ep, K)$-nearly straight traintrack $T = \{R_j\} _j$ on $\sigma$ fully carrying $\nu$, such that if a vertical edge of $T$ intersects a leaf of $\nu$, then the angle is $\ep$-close to $\pi/2$.

If $\nu$ contains a closed leaf, then
we can take  $K > 0$ to be any number less than one third of  the length of the shortest closed  leaf of $\nu$ and, otherwise, we can take $K$ to be any positive number. 
\end{lemma}

Such a traintrack can be obtained by taking a $\del$-neighborhood of $|\nu|$ with sufficiently small $\del > 0$ and splitting it so that each branch has a certain amount of length; the details are left to the reader.

By Lemma \ref{6-20-12no2}, we have a $\rho$-equivariant pleated surface $\beta_\In$ realizing $(\tau_\In, \lam_\In)$.
Similarly to Thurston coordinates on projective structures on $S$,  the pair $(\tau_\In, \lam_\In)$ defines a projective structure $C_\In$ on $S \minus |\lam_\In|$. 
Indeed, since $\beta_\In$  is a locally totally geodesic embedding away from the total lift $\td{\lam}_\In$ of $\lam_\In$, it induces a $\rho$-equivariant developing map $f_\In$ from $\td{S}$ minus $|\td{\lam}_\In|$, so that $f_\In(x)$ projects orthogonally to the image of $\beta_\In(x)$ for all $x \in \til{S} \minus \til{\lam}_\infi$. 
In particular there is a natural embedding $\kap_\In$ of $C_\In$  onto  $\tau_\In \minus |\lam_\In|$.

For every $\ep > 0$, let $T_\In (= T_{\In, \ep})$ be an $(\ep, K)$-nearly straight traintrack on $\tau_\If$ given by Lemma \ref{8-2} carrying $\lam_\In$.
Then let $\mathcal{T}_\In = \kap^{-1}(|T_\In|)$ be the subset of $C_\In$, such that the closure of $C_\In \minus \mathcal{T}_\In$ is a compact subset of $C_\In$,  on which  $C_\In$ deformation retracts. 

Suppose that there is a nearly straight traintrack $T_i$ on $\tau_i$ carrying $L_i$.
 Then, since vertical edges of $T_i$ are transversal to $L_i$, thus  $\kap_i^{-1}(T_i) =: \mathcal{T}_i$ is a traintrack on $C_i$ carrying the canonical lamination $\LL_i$.  
Note that $\mathcal{T}_i$ and $T_i$ are the same traintrack as topological traintracks on $S$ and this identification is given by $\kap_i$. 
The measured laminations  $\LL_i$ and $L_i$ represent the same element on $\ML(S)$.
Then if $\RR$ and $R$ are corresponding branches of $\mathcal{T}_i$ and $T_i$, then the weight of $R$ given by $L$ is equal to the weight of $\RR$ given by $\LL_i$.  
In this sense, $(\mathcal{T}_i, \LL_i)$ is {\it isomorphic} to $(T_i, L_i)$  (as weighted traintracks).

Then we show (I) and (II):
\begin{proposition}\Label{7-7-12no1}
Let  $\ep > 0$.
Then, for  sufficiently large $i \in \N$, there exists a traintrack $T_i$ on $\tau_i$  isotopic to $T_\In$ on $\tau_\In$ as a topological traintrack on $S$, such that 
\begin{itemize}
\item[(i)] $T_i$ is $(\ep, K)$-nearly straight,
\item[(ii)] $T_i$ carries $L_i$,
\item[(iii)] there is a marking-preserving $\ep$-rough isometry from $\tau_\infi$ to $\tau_i$ that takes $|T_\In|$ to $|T_i|$, and
\item[(iv)] $C_\If \sm |\mathcal{T}_\If|$ is isomorphic to $C_i \minus |\mathcal{T}_i|$ (as projective surfaces) via their developing maps, where $\mathcal{T}_i$ the traintrack on $C_i$ that descends to $T_i$ via the collapsing map $\kap_i\cn C_i \to \tau_i$.

\end{itemize}
\end{proposition}

\begin{proof}
For each component $P$ of $\tau_\In \sm |T_\In|$, let $P'$ be the component of $\tau_\In \minus |\lam_\In|$ containing $P$.
For  each $i$,  let $\psi_i \cn \tau_\In \to \tau_i$  be a marking-preserving $\del_i$-rough isometry, obtained by Theorem \ref{12-19},  with its distortion $\del_i$ limiting to $0$ as $i \to \In$. 
Then since a small neighborhood of $|\lam_\In|$ contains $|\lam_i|$ for sufficiently large $i$, there is a corresponding component $P'_i$ of $\tau_i \minus |\lam_i|$ such that $\phi_i(P'_i)$ contains $P$. 
 Let $\td{P}'$ be a lift of $P'$ to $\h^2$.
Then $\beta_\In$ takes $\td{P}'$ isometrically into a (totally geodesic) hyperbolic plane in $\h^3$.
The ideal boundary of this hyperplane cuts $\rs$ into two round open balls. 
Then, one of those round balls is in the normal direction of $\beta_\In|\td{P}'$ (see Remark \ref{111112}).
Let $\td{Q}'$ be the region in this round ball that conformally projects onto  $\beta_\In(\td{P}')$ via the orthogonal projection to the hyperplane.
Since $P$ is a subset of $P'$,  let $\td{Q}$ be the regain in $\td{Q}'$ conformal to $\beta_\In(P)$ via the projection.

Similarly, $\beta_i$ isometrically embeds $\til{P}'_i$ into $\h^2$ in $\h^3$, and we let $\td{Q}'_i$ be the region in $\rs$ conformal to $\td{P}'_i$ via the orthogonal projection to this hyperplane.
Then, for sufficiently large $i$,  $\td{Q}'_i$ contains $\td{Q}$ since $\psi_i(\td{P}'_i)$ contains $\td{P}$ and  $\beta_i$ is sufficiently close to $\beta_\In$.
Then via the conformal isomorphisms  $\td{Q}' \cong \td{P}'$ and $\td{Q}_i \cong \td{P}_i$, we have an embedding $\td{\eta}_i\col \td{P} \to \td{P}'_i$.
Then  $\beta_i \cc \td{\eta}_i$ smoothly converges to  $\beta_\In$  as $i \to \In$.  
Thus, for every $\ep > 0$, if $i$ is sufficiently large, $\td{\eta}_i$ and  $\td{\psi}_i$ are $\ep$-close on $\td{P}$ (in the $C^\infi$ topology).

Moreover,  since $\beta_i \to \beta_\In$,  if $i$ is sufficiently large,  
then different complementary components of $|\td{T}_\In|$ have disjoint images in  $\td{\tau}_i \minus |\td{\lam}_i|$. 
Thus we have a conformal embedding $\td{\eta}_i\cn \td{\tau}_\In \minus |\td{T}_\In|  \to  \td{\tau}_i \minus |\td{\lam}_i|$ that commutes with the action of $\pi_1(S)$. 
 Then it descends to an embedding  $\eta_i\cn\tau_\In \minus |T_\In| \to \tau_i \minus |\lam_i|$. 
 Moreover, for every $\ep > 0$, if $i$ is sufficiently large, $\eta_i$ and $\psi_i$ are $\ep$-close on $\tau_\In \minus |T_\In|$.
 Since $\psi_i$ converges to an isometry as $i \to \In$,  therefore $\tau_i \minus \Im \eta_i$ enjoys a traintrack structure $T_i$ carrying $L_i$ (proving (ii)) that satisfies (i) and (iii) for sufficiently large $i$.
There is a unique isomorphic embedding of $C \sm |\mathcal{T}_\In|$ into $C_i \minus |\LL_i|$ compatible with their developing maps, such that it descends to $\eta_i$ via the collapsing maps of $C$ and $C_i$. 
Hence $T_i$ also satisfies (iv). 
\end{proof}

\subsection{Proof of Proposition \ref{6-28-12no1} (III) - (i).}
By the arguments above, we can assume that $T_\infi$  is $(\ep, K)$-nearly straight with the fixed constant $K > 0$ and sufficiently small $\ep > 0$.

Let $b$ be a switch point of $T_\In$.
Then, since $b$ is in the complement of $|\lam_\In|$, the pleated surface $\beta_\In\col \h^2 \to \h^3$ is smooth at each lift $\td{b}$ of $b$ to $\h^2$.
Thus there is a unique point  $\td{\tt b}$ on $\rs$ that orthogonally projects to $\beta_\In(\td{b})$ on  $\beta_\In$ from its normal direction (see Remark \ref{111112}).

Pick a round circle ${\tt c}(\td{b})$ on $\rs$ containing $\til{\tt b}$ so that the hyperplane bounded by ${\tt c}(\td{b})$ contains  $\beta_\In(\td{b})$ and so that it is ``{\it nearly orthogonal} to the traintrack of $\til{T}_\In$ near $\td{b}$'':
Namely, given any $\zeta > 0$, if $T_\infi$ is sufficiently straight (i.e. $\ep > 0$ is sufficiently small), then,  for each leaf $\ell$ of $\lam_\In$ passing through the vertical edge of $\td{T}_\In$ containing $b$,  the geodesic $\beta_\In(\ell)$ is 
 $\zeta$-nearly orthogonal to the hyperplane.
 For different lifts $\td{b}$ of $b$, we  $\rho$-equivariantly take such round circles ${\tt c}(\td{b})$.  
(Note that the upper bound of the number of branches points on $T_\infi$ depends only on the topological surface $S$.)

Then, since each branch of $\td{T}_\In$ has length at least $K$ and we can pick sufficiently small $\zeta > 0$,  for each branch $\td{R}_\infi$ of $\td{T}_\In$, letting $\td{b}_1$ and $\td{b}_2$ be the switch points on the different vertical edges of $\td{R}_\infi$, 
their corresponding round circles ${\tt c}(\td{b}_1)$ and ${\tt c}(\td{b}_2)$ are disjoint (in $\rs$).
Let $\AA(\td{R}_\infi)$ denote the round cylinder bounded by ${\tt c}(\td{b}_1)$ and ${\tt c}(\td{b}_2)$.
Then the convex hull, in $\h^3$, of $\AA(\td{R}_\infi)$ contains most of $\beta_\In (\td{R}_\infi)$.

Consider the two copies of $\h^2$ in $\h^3$ bounded by ${\tt c}(\td{b}_1)$ and ${\tt c}(\td{b}_2)$.
Then, for every $\zeta > 0$, if $\ep > 0$ is sufficiently small, the distance between these hyperplanes is at least  $K - \zeta$ for all branches $\til{R}_\infi$ of $\til{\mathcal{T}}_
\infi$.
Therefore the modulus of  $\AA(\til{R}_\infi)$ is at least $(K - \zeta)/ 2\pi$.

For all sufficiently large $i$, set $T_i = \{R_{i,j} \}$ to be the $(\ep, K)$-nearly straight traintrack on $\tau_i$ obtained by Proposition \ref{7-7-12no1}.
Note that, for different $i$, the traintracks  $T_i$ are isomorphic as smooth traintracks and those isomorphisms, restrict, for each $j$, to a diffeomorphism between corresponding branches $R_{i, j}$. 
Accordingly, set $\mathcal{T}_i = \{\RR_{i, j}\}_j$ to be the corresponding traintrack on $C_i$, so that $\kap_i$ maps $\RR_{i, j}$ to $R_{i, j}$ for each $j$. 
In addition there is a branch $R_{\infi, j}$ of $T_\infi$ corresponding to $\RR_{i, j}$ and $R_{i, j}$.

\begin{proposition}\Label{6-17-12}
For every $\zeta > 0$, if $\ep > 0$ is sufficiently small,  then, for  $i$ sufficiently large so that $T_i$ is $(\ep, K)$-nearly straight, we can isotope $\mathcal{T}_i$ on $C_i$ by a $\zeta$-small isotopy of the vertical edges of $\mathcal{T}_i$, so that,
if $R$, $\RR$ and $R_\infi$ are corresponding branches of $\td{T}_i$, $\td{\mathcal{T}}_i$ and $\til{T}_\infi$, respectively, then 
 \begin{itemize}
\item[(i)]  $(\mathcal{T}_i, \LL_i)$ remains isomorphic to $(T_i, L_i)$ as weighted traintracks,
 \item[(ii)]  each branch  $\RR$ of $\td{\mathcal{T}}_i$ is supported on the round cylinder $\AA(R_\infi)$, and the modulus of $\AA(R_\infi)$  is at least $(K - \zeta)/ 2\pi$.
 \item[(iii)]  both horizontal edges of $\RR$  intersect each leaf of the circular foliation of $\AA(R_\infi)$ at angles in $(\pi/2 -\zeta, \pi/2 + \zeta)$, and  
 \item[(iv)]  the isotopy moves each vertical edge of $\mathcal{T}_i$ at most $\zeta$ in the Hausdorff distance with respect to the Thurston metric on $C_i$. 
\end{itemize}
 \end{proposition}

 \proof
In this proof, we can assume that $i$ is sufficiently large. 
Let $R$ and $\RR$ be corresponding branches of $\td{T}_i$ and  $\td{\mathcal{T}}_i$, respectively. 
Then let $P$ be a stratum of $(\h^2, \td{\lam}_i)$ that intersects $R$.
Let $\mP = \td{\kp}_i^{-1}(P)$, where $\td{\kap}_i\cn \td{C}_i \to \h^2$ is the lift of $\kap_i \col C_i \to \tau_i$.

The horizontal edges of $R$ are contained in different 2-dimensional strata of $(\h^2, \td{\lam}_i)$. 
Then those strata bound a region in $\h^2$ containing all other strata intersecting $R$.

First suppose that $P$ is one of those other strata, so that $P$ intersects no horizontal edge of $R$. 
Let  $\mP_{\AA(R_\infi)}$ be $\mP \cap f_i^{-1}(\AA(R_\infi))$. 
Then we show that
\begin{itemize}
\item $\mP_{\AA(R_\infi)}$ is a rectangle if $\dim \mP = 2$ or an arc supported on $\AA(R_\infi)$ if $\dim \mP = 1$, and 
\item   $\mP_{\AA(R_\infi)}$  is $\zeta$-close to $\RR \cap \mP$ with the Thurston metric  on $C$.
\end{itemize}

Suppose, in addition, that $P$ is a leaf of $\til{\lam}_i$.
If  $\ep > 0$ is sufficiently small, then the geodesic $\beta_i(P)$ is $\zeta$-close to the axis of $\AA(R_\infi)$ since $\beta_i \to \beta_\infi$ as $i \to \infi$. 
Thus, we can assume that, for each leaf $c$ of  the circular foliation of $\AA(R_\infi)$, if we let $H_c \sub \h^3$ be the hyperbolic plane bounded by $c$, 
then $\beta_i(P)$ intersects $H_c$ in a single point at an angle $\zeta$-close to $\pi/2$.
If $P$ has no atomic measure, $f_i(\mP)$ is a circular arc on $\rs$.
Then, if $\ep > 0$ is sufficiently small,  $f_i(\mP)$ intersects each leaf of the foliation of $\AA(\RR_\infi)$ at an angle $\zeta$-close to $\pi/2$.  
In particular $f_i(\mP) \cap \AA(R_\infi)$ is a connected circular curve  supported on $\AA(R_\infi)$.
Thus, if $\ep > 0$ is sufficiently small, $\beta_i(P)$ is sufficiently close to the axis of  $\AA(R_\infi)$, and therefore $f_i^{-1}(\AA(R_\infi)) \cap \mP$ is $\zeta$-close to $\RR \cap \mP$.

If $P$ has positive atomic measure, then $\mP$ is foliated by leaves of $\td{\LL}_i$.
Then if $\ell$ is a leaf of this foliation,  then $f_i(\ell) \cap \AA(R_\infi)$ is supported on $\AA(R_\infi)$  as above. 
In addition, its $f_i^{-1}$-image is $\zeta$-close to $\ell \cap \RR$, and there is a small isotopy between them in $\ell$.
Since $\mP \cap f_i^{-1}(\AA(R_\infi))$ is the union of such arcs,  similarly it is a rectangle supported on $\AA(R_\infi)$ and $\zeta$-close to $\RR \cap \mP$.

Suppose that $P$ is a complementary region of $\td{L}_i$.
Then accordingly $\mP$ is a complementary region of $\td{\LL}_i$.
If $\ep > 0$ is small enough, $R \cap P$ is a very thin rectangle bounded by  the vertical edges of $R$ and two boundary geodesics of $P$ intersecting them.
Then, regardless of the choice of $P$, those boundary geodesics are $\zeta$-close in $R$ and their $\beta_i$-images are geodesics in $\h^3$ that are $\zeta$-close to the axis of $\AA(R_\infi)$ in the convex hull $\Conv(\AA(R_\infi))$.
On the other hand, the other boundary geodesics of $\beta_i(P)$ are far away from  $\Conv\AA(R_\infi)$.
 Then $f_i(\mP) \cap \AA(R_\infi)$ is a rectangle supported on $\AA(R_\infi)$.
Hence, if $\ep > 0$ is sufficiently small,  since $f_i$ embeds $\mP$ into $\rs$, then $\mP \cap f_i^{-1}(\AA(R_\infi))$ is a rectangle supported on $\AA(R_\infi)$ and $\zeta$-close to $\mP \cap \RR$.

Next suppose that  $P$ contains a horizontal edge of $R$.
Then $P \cap R$ is a thin rectangle bounded by the vertical edges of $R$, a leaf $\ell$ of $\td{L}_i$, and a smooth boundary segment $m$ of $|\td{T}_i|$.
Then, if $\ep > 0$ is sufficiently small,   $m$ is an {\it almost geodesic} (i.e. a bilipschitz curve with very small distortion), and thus  $\beta_i(\ell)$ and $\beta_i(m)$ are $\zeta$-close to the axis of $\AA(R_\infi)$ in $\Conv \AA(R_\infi)$, independent of the choice of $P$. 
Let ${\tt l}$ be the boundary leaf of $\mP$ corresponding to $\ell$ and ${\tt m}$ be the smooth boundary segment of $|\td{\mathcal{T}}_i|$ corresponding to $m$.
 Then $f_i$ embeds $\mP$ into $\rs$.
 Thus $f_i({\tt l})$ is a circular arc and $f_i({\tt m})$ is an ``almost'' circular arc on $\rs$, and they intersect $\AA(R_\infi)$ almost orthogonally. 
Therefore, if $\ep > 0$ is sufficiently small,  each leaf of the foliation $\AA(R_\infi)$ intersects $f_i({\tt l})$ and $f_i({\tt m})$ in single points $\zeta$-orthogonally. 
 Thus there is a unique thin rectangular component of $f_i(\mP \cap \td{\mathcal{T}}_i) \cap \AA(R_\infi)$  bounded by $f_i({\tt l})$, $f_i({\tt m})$ and the boundary circular loops of $\AA(R_\infi)$. 
Let $\mP_{\AA(R_\infi)}$ be the subset of $\mP$ that diffeomorphically maps onto the component by $f_i$. 
Then similarly $\mP_{\AA(R_\infi)}$ is $\zeta$-close to $\RR \cap \mP$.

We have shown that, if $P$ is a stratum of $(\h^2, \td{L}_i)$ intersecting $R$, then   $\mP_{\AA(R_\infi)}$ is either an arc or a rectangle supported on $\AA(R_\infi)$. 
Thus $\mP_{\AA(R_\infi)}$ is diffeomorphic to the product of  a point and an interval or of two intervals, where the second factor is in the direction orthogonal to the circular leaves of $\AA(R_\infi)$. 
Let $\RR'$ be the union of all $\mP_{\AA(R_\infi)}$ over all strata $P$ of $(\h^2, \td{L}_i)$ intersecting $R$.
Then since the developing map is a local homeomorphism, by continuity, the product structures of $\mP_{\AA(R_\infi)}$ match up.
Therefore  $\RR'$ is a (smooth) rectangle supported on $\AA(R_\infi)$.
 In addition, since each $\mP_{\AA(R_\infi)}$ is $\zeta$-close to $\RR \cap \mP$, thus $\RR'$ is $\zeta$-close to $\RR$ if $\ep > 0$ is sufficiently small. 
  
One can construct a desired small isotopy of the vertical edges of $\RR$. 
Construct small isotopies  all $\RR \cap \mP$ and $\mP_{\AA(R_\infi)}$ in $\mP$ for all strata $P$ of $(\h^2, \td{L}_i)$ intersecting $R$ so that they match up. 
\endproof

 \subsection{Estimates of  admissible multiarcs  by transversal measure}\Label{11-12}

Recall that $T_\infi$ is an $(\ep, K)$-nearly straight traintrack carrying $\lam_\infi$ and $T_i$ is,  for $i$ sufficiently large, an  $(\ep, K)$-nearly straight traintrack carrying $L_i$. 
Set $L_i = (\lambda_i, \mu_i)$ for each $i \in \n$ with a geodesic lamination $\lam_i$ and its transversal measure $\mu_i$.
We prove Proposition \ref{6-28-12no1}, III - ii: 
\begin{proposition}\Label{6-20-12no1}
For every $\dl>0$, 
if $\ep > 0$ is sufficiently small and $i \in \n$ is sufficiently large, then for  all corresponding branches $R$ and $\RR$  of $\td{T}_i$ and $\td{\mathcal{T}}_i$, respectively,
 we have $| \length_{\AA(R)} (a)  - \mut_i(R) | < \dl$ for each vertical edge $a$ of $\RR$.
\end{proposition}

{\it Idea of proof.}
Let $R_\infi$ be the branch of $\til{T}_\infi$ corresponding to $R$  and $R_i$. 
Then there are hyperbolic  planes in $\h^3$ almost orthogonal to the $\beta_\infi$-image of $R_\infi$. 
The collapsing map $\til{\kap_i}\col \til{C}_i \to \til{\tau}_i$ on $R_\infi$ corresponds to the nearest point projections in $\h^3$ to  the hyperbolic planes supporting $\beta_\infi$ on $R_\infi$. 
Thus this proposition is proven by carefully relating $\length_{\AA(R)} (a)$ and  $\mut_i(R)$ in a hyperbolic plane  almost orthogonal  to $\beta_\infi | R_\infi$.

\begin{remark}\Label{Re:vertical}
Let $\LL_i$ be the canonical measured lamination on $C_i$, which descends to $L_i$. 
Then the transversal measure of $R$ given by $\td{L}_i$ is equal to that of $\RR$ given by $\td{\LL}_i$.
Recall that $\RR$ is foliated by the circular arcs parallel to its vertical edges  (\S\ref{10-5-1}).  
Then Proposition \ref{6-20-12no1} holds for the length of each leaf of the foliation, since the leaves have almost the same length. 
\end{remark}

{\it Proof of Proposition \ref{6-20-12no1}.}
We have the convergence $\beta_i \to \beta_\In$ (Lemma \ref{6-20-12no2}). 
Recall that $T_\If$ is $\ep$-nearly straight with a sufficiently small $\ep > 0$.
Throughout this proof, $\ep > 0$ is a sufficiently small number, which depends on $\del > 0$ but not on the choices of sufficiently large $i$ and  the corresponding branches $R$ and $\RR$.

For an arbitrary branch $R$ of $\td{T}_i$,  let $I = I(R)$ be the minimal sublamination of $\td{L}$ containing the leaves of $\td{L}_i$ intersecting $R$ (c.f. \cite{Baba-10}). 
Set $I = (\ld_I, \mu_I)$, where $\ld_I \in \gl(\h^2)$ and $\mu_I = \mut_i|\ld_I$.
Accordingly, let $\beta_I\cn \h^2 \to \h^3$ denote the bending map induced by $I$.
Then the total transversal measure of $I$ is $\mut_i(R)$.
In particular, for every geodesic segment $s$ on $\h^2$ transversal to $I$, we have $\mu_I(s) \leq \mut_i(R)$. 
Therefore $\beta_I\cn \h^2 \to \h^3$ continuously extends to $\pt_\In \h^2 \cong \s^1 \to \pt_\In \h^3$.

Each complementary component of $|\lam_I|$ is bounded by at most two leaves of $\lam_i$.
Thus the geodesic lamination $\ld_I$ extends to a geodesic foliation $F_I$ on $\h^2$.
Furthermore, 
since $T_i$ is a $(1 + \ep, K)$-nearly straight traintrack carrying $\lam_i$, for every $\dl > 0$, if $\ep > 0$ is sufficiently small, then 
each vertical edge  of $R$ is $\del$-nearly orthogonal to every leaf of $F_I$ unless they are disjoint. 

Let $C_I$ denote the projective structure on the open disk $\mathbb{D}^2$ associated with the measured lamination $I$ on $\h^2$. 
Since $R$ is connected and $I$ is a sublamination of $\Lt_i$, then  $C_I$, as a projective surface, isomorphically embeds into $\Ct_i$ in a canonical way (see \cite[\S 3.8.1]{Baba-10}).
In particular, since $I = \Lt_i$ on $R$,  then $\RR$ canonically  embeds into $C_I$.
Let $f_I\cn \mathbb{D}^2 \to \rs$ denote the developing map of $C_I$, and
let $\mF_I$ denote the canonical foliation on $C_I$ corresponding to $F_I$, so that the collapsing map $\kp_I\cn C_I \to \h^2$ takes each leaf of $\FF_I$ to a leaf of $F_I$ diffeomorphically.
Then the dual tree of $F_I$ is homeomorphic to an open interval. 
With the Thurston metric on $C$, the collapsing map $\kp_I$  continuously extends to a homeomorphism between ideal boundaries $\pt_\In C_I$ and $\pt_\In \h^2 ~ (\cong \s^1)$.

There are exactly two points on $\pt_\In C_I$ that are not endpoints of  leaves of $\FF_I$\,;
they divide $\pt_\In C_I\ (\cong \s^1)$ into two open intervals.
Pick one of the intervals, and let $\Phi$ be the projection of $C_I$ onto the interval along leaves of $\mF_I$.
With respect to the Thurston metric, $C_I$ is divided into Euclidean and hyperbolic regions.  \note{``hyperbolic regions''}
Then each connected component of the Euclidean region is foliated by leaves of $\mF_I$ sharing endpoints on $\pt_\In C_I$, and  
 $\Phi$ takes the component  to the end point in the chosen interval. 

Since each vertical edge $a$ of $\RR$ is transversal to the foliation $\mF_I$,
 $\Phi (a)$ is an arc $b$ in the open interval in  $\pt C_I$.
Recalling that  $\beta_I\cn \h^2 \to \h^3$  extends to $\pt_\In \h^2 \to \pt_\In \h^3$ and $\kp_I\cn C_I \to \h^2$ to $\pt_\In C_I \to \pt_\In \h^2$, we have $f_I = \beta_I \cc \kp_I$ on $\pt_\In C_I$.
Since $\kp_I\cn \pt_\In C_I \to \pt_\In \h^2$ is a homeomorphism,  we can identify $b$ with its image in $\pt \h^2$.
Thus let $\beta_{b}\cn b \to \rs$ denote the (continuous) path obtained by restricting $\beta_I$ to $b$.

The transversal measure  $\mu_I$ of $I$ is defined for arcs in $\h^2$ transversal to $I$.
Since the total measure of $\mu_I$ is finite, $\mu_I$ continuously extends to arcs in $\pt \h^2$.
Then $\beta_I | \pt_\In \h^2$ is determined by $\mu_I | \pt_\In \h^2$.
In particular,  $\beta_b$ can be regarded as a bending map of $b \st \pt_\In \h^2$ by the measure $\mu_I$ on $b$.

Let $g \st \h^3$ be the axis  of the round cylinder $\AA(R)$.
For each vertical edge $a$ of $\RR$, there is a map to its corresponding boundary circle $h$ of $\AA(R)$. 
Identifying $\rs$ with $\s^2$ conformally, we normalize $\rs$ by an element of $\psl$ so that $h$ is the equator. 
Let $\Conv(A)$ be the convex hull of $\AA(R)$ in $\h^3$. 
Then, for every $\dl > 0$, if $\ep > 0$ is sufficiently small,  then, for all branches $R$ of $\td{T}_i$ and for all leaves $l$ of $\td{L}_i$ intersecting $R$, since the geodesic $\beta(l)$ is nearly orthogonal to the boundary hyperplanes of $\Conv(\AA(R))$,  the geodesic segments $\beta_I(l \cap R)$ and $g \cap \Conv(\AA(R))$  are $\dl$-close in the Hausdorff metric. 
Thus, for any (small) $\del > 0$,  if $\ep > 0$ is sufficiently small,  then $\Im(\beta_{b})$ is contained in a  $\del$-neighborhood of an endpoint $O$ of $g$ on $\rs$ in the spherical metric. 
 In particular, $\Im(\beta_{b})$ is contained in  a round disk $D$ on $\rs$ bounded by $h$. 

By the definition of $\Phi\cn C_I \to \pt C_I$, for each  $x \in a$,  $\Phi(x)$ is an endpoint of the leaf $\ell$ of $\mF_I$ containing $x$. 
Consider the ray from $x$ to $\Phi(x)$ contained in $\ell$.
Then  $f_I$ homeomorphically takes  this ray onto  a circular arc in $\rs$ that connects the point $f_I(x)$ to  the point $f_I(\Phi(x)) = \beta_{b}(\Phi(x))$.
Let $r_x\cn [0,1] \to D$ denote this circular arc  with $r_x(0) = f_I(\Phi(x))$ and $r_x(1) = f_I(x)$.
Then, since the geodesic $\beta_I\cc \kap_I(\ell)$ is nearly orthogonal to the hyperplane $\Conv(h)$, for every $\dl > 0$, if $\ep > 0$ is sufficiently small, then   $r_x$ and $h$ are  $\dl$-nearly orthogonal (at the point $f_I(x)$) and $\gm_x(0)$ is $\dl$-close to the center $O$. 

There is a unique maximal ball in $C_I$ whose core contains $x$.
By the definition of a maximal ball, its $f_I$-image   is a round open ball $R_x$ in $\rs$.
The boundary circle of $R_x$  bounds a hyperbolic plane $H_x$ in $\h^3$.
Then $r_x$ orthogonally intersects $\pt R_x $ at the endpoint $r_x(0)$.
For all $\del > 0$, if $\ep > 0$ is sufficiently small, then, since $\beta_I \cc \kp_I(\ell)$ is nearly orthogonal to $\Conv(h)$, 
 the curvature of $r_x\cn [0,1] \to D \st \s^2$ is less than $\dl$ (in the induced spherical metric on $D$). 
Let $\rt_x\cn [0,1] \to D$ be the spherical geodesic segment in $D$ connecting the endpoints of $r_x$. 
Then, since $r_x(0)$ is sufficiently close to the center $O$ of $D$ and the curvature of $r_x(0)$ is sufficiently small, for every $\dl > 0$, if $\ep > 0$ is sufficiently small, then, (in particular) for all $x \in a$,   the ``almost'' geodesic segment  $r_x$ is $\dl$-close to the geodesic segment  $\rt_x$ in the hausdorff metric in  $D$.
\begin{lemma}\Label{10-26-1}
For every $x$ on the arc $a$, there exists a small neighborhood $U_x$ of $x$ in $a$, such that if $y \in U_x$, then $\rt_x$ and $\rt_y$ are disjoint,  except possibly
at their endpoints close to $O$ (i.e. possibly  $\rt_x(0) = \rt_y(0)$).
\end{lemma}
\begin{proof}
Since $f_I|a$ is an immersion into $h$, if $ y \in a$ is sufficiently close to $x$, then   $\rt_x(1) = f_I(x)$ is different from $\rt_y(1) = f_I(y)$.
If $\rt_x(0) = \rt_y(0)$, since $\rt_x$ and $\rt_y$ are  geodesic segments in the hemisphere $D$, they are disjoint except at $\rt_x(0) = \rt_y(0)$. 

Next  assume that $\rt_x(0) \neq \rt_y(0)$ for $y \in a$ sufficiently close to $x$.
The foliation $\FF_I$  on $C_I$ carries a canonical transversal measure that descends to $I$ on $\h^2$ and it has no atomic measure.
Then, by continuity, if $y$ is sufficiently close to $x$,  the transversal measure of the segment in $a$ connecting $x$ to $y$ is sufficiently small. 
Therefore $r_x$ and $r_y$  are disjoint. 
Let $\ell$ be the spherical geodesic in $D$ through $r_x(0)$ and $r_y(0)$ with its endpoints on $h$.
Then  $r_x$ and $r_y$  are  contained in a component $P$ of $D \minus \ell$, so that the endpoints of $r_x$ and $r_y$ on the boundary $P$.
Since $P$ is convex and  $r_x$ and $r_y$  are disjoint,  by the uniqueness of geodesics, $\rt_x$ and $\rt_y$ must be disjoint. 
\end{proof}

For each $x \in a$, let $\gm_x\cn [0, 1] \to D$ denote the geodesic segment  on $D \st \s^2$ connecting the center $O$ to $r_x(1)$.
Then $\gm_x$ intersects $h = \pt D$ orthogonally at $\gm_x(1)$. 
Let $\ap = f_I | a$.
Parametrize $a$ so that $\ap \cn a \to h \st \s^2$ is an isometric immersion, and identify $a$ with the closed interval $[0, A]$, where $A$ denotes the length of $\ap$.
Then,  as $x \in  [0, A]$ increases,  the circular arc $\gm_x\cn [0, 1] \to D~ (x \in [0,1])$ changes by the continuous rotation of $D$ about  $O = \gm_x(0)$ monotonically.
Thus define $\gm\cn a \times [0,1] \to D$ by $\gm(x, t) = \gm_x(t)$.
Then $\gm(x , 1) = \ap(x)$ for all $x \in a$.
 Then  $\gm| a \times (0,1]$ is an immersion.
(The parametrized surface $\gm$ is a {\it fan} where the vertex of the fan  is  $O$ and the angle of the fan is $A$.)
Let $E$ denote the domain $a \times [0,1]$ of $\gm$  with the metric obtained by pulling back the spherical metric on $D$ via $\gm$.
Then we have
\begin{equation}
\Area(E) =  \Area_{\s^2}(D) \cdot (A/2\pi) = A.\label{8-23_1}
\end{equation}

Similarly define $\rt\cn a \times [0,1] \to D$ by $\rt(x, t) = \rt_x(t)$.
Then, by Lemma \ref{10-26-1}, the restriction of $\rt$ to $a\times (0,1]$ is an immersion.
Clearly $\rt| a \times \{0\}$ is $\beta_b \cc \Phi\cn a \to D$, and $\rt| a \times \{1\}$ is $ \ap\cn a \to \pt D$.
Let $F$ be the rectangular domain $a \times [0,1]$ of $\rt$ with the pull-back metric of the spherical metric on $D$.
Then the boundary edges of the rectangle $F$ correspond to the four curves $\ap$, $\beta_{b}$, $\rt_0$, and $\rt_1$.
Applying the Gauss-Bonnet theorem to $F$ with respect to the spherical metric, we have
$$ \on{\Area}(F) + \int_{\pt F }  k ds + \Sigma \theta_p =  2\pi\cdot 1,$$ 
where $k$ is the curvature at smooth points of $\pt F$, and $\theta_p$ are the exterior angles at non-smooth points $p$ of $\pt F$, which  include (infinitesimal) bending angles of $\beta_b$ corresponding to $\mu_I$.
Then, since $\beta_b \col b \to D$ be obtained by bending $b$ with respect to $\mu_I | b$,  the third term $\Sigma \theta_p$  is  $- \mut_i(R) + 2\pi$.

Next consider the second term  $$\int_{\pt F} k\, ds =  \int_{\beta_b}  k\, ds + \int_{\rt_0} k \, ds+ \int_{\rt_1} k\, ds+ \int_\ap k \, ds.$$
Since $\rt_0$ and $\rt_1$ are geodesic segments, $ \int_{\rt_0} k \, ds = 0$ and $\int_{\rt_1} k\, ds = 0$.
Since $\ap$ is a segment of the geodesic loop  $\pt D $ on $\s^2$, we have $\int_{\ap} k ds =0$. 
For every $\del > 0$, if $\ep > 0$ is sufficiently small, then the curvature at every smooth point $x$ of $\ap$ is less than $\del$, since $\beta_b(x)$ is sufficiently close to $O$ and $H_x$ is almost orthogonal to the hyperplane, in $\h^3$, bounded by $\bdr D$. 
Since we can assume that the length of $\beta_b$ is sufficiently small,  we  have  $\int_{\pt F} k\, ds < \del$.  
Thus 
\begin{lemma}\Label{10-18-1}
For every $\dl > 0$, if $i \in \n$ is sufficiently large and $\ep >  0$ is sufficiently small, then, 
$$-\dl < \Area(F) - \mut_i(R) < \dl,$$
for all branches $R$ of $\td{T}_i$.
\end{lemma}

By (\ref{8-23_1}) and  Lemma \ref{10-18-1}, to prove Proposition \ref{6-20-12no1}, it suffices to show  
\begin{proposition}\label{8-23_2} 
For every $\dl > 0$, if $\ep > 0$ is sufficiently small, then, for every pair of corresponding branches $R$ and $\RR$ of $\Tt_\If$ and $\td{\mathcal{T}}_\In$, respectively, and every vertical edge $a$ of $\RR$, we have
 $$-\dl < \Area (F) - \Area (E)  < \dl,$$
where $F$ and $E$ are defined as above. \end{proposition}
\note{eng: are as above.}

\proof
Fix sufficiently small $\del > 0$. 
For each $y \in b$, let $g_y\cn [0,1] \to D$ be the geodesic segment from $O =  g_y(0)$ to $\beta_b(y) = g_y(1)$. 
If $\ep > 0$ is sufficiently small, then  $\Im (\beta_b)$ is contained in the $\del$-neighborhood of $O$.
Thus $\length_{\s^2}(g_y) < \dl$ for all $y \in b$. 
Define  $g\cn b \times [0,1] \to D$ by $g(y,t) = g_y(t)$ for $y \in b$ and $t \in [0,1]$.
Then $g$ is smooth almost everywhere since so is $\beta_b$. 
Let $G = b \times [0,1]$ equipped with the 2-form obtained by pulling back the spherical Riemannian metric of $D$ via $g$. 
In particular, this form induces the arc length of $\beta_b\cn b \to D$ when restricted to $b \times \{1\} = b$.
Then we see that $0 \leq \Area(G) =  \int_{y \in b} \frac{1}{2}\length(g_y) dy$.   
Therefore, if $\ep > 0$ is sufficiently small, then $\Area(G) < \dl$, since  $\length (\beta_b)$ and $\length(g_y)$ for all $y \in b$ are sufficiently small. 
   
Let $\Delta_0$ be the (geodesic) triangle in $D$ bounded by $\rt_0, \gm_0$,  $g_0$, and $\Delta_A$ the triangle bounded by $\rt_{A}, \gm_{A}$, $g_{A}$. 
Then,  if $\ep > 0$ is sufficiently small, then $g_0$ and $g_{A}$  are sufficiently short so that $\Area_{\s^2}(\Delta_0) , \Area_{\s^2}(\Delta_{A}) < \del$. 
Thus, it suffices to show that, for sufficiently small $\ep > 0$, we have
\begin{eqnarray}
| \Area (F) - \Area (E) | < \Area(G) + \Area(\Delta_0) + \Area(\Delta_A). \label{120312}
\end{eqnarray}

In order to prove (\ref{120312}), we decompose the interval $[0, A]$ so that it  accordingly decomposes $F$, $E$, $G$ into subsets,  and we show similar inequalities for the corresponding subsets.
Let $X$ be the set of points $x$ in $[0, A]$ such that $\gamma_x$ and $\rt_x$ are parallel, so that either $\gamma_x \supset \rt_x$ (Type I) or $\gamma_x \subset \rt_x$ (Type II) holds.
Then accordingly either  
 $\gamma_x =  \rt_x \cup g_x$ or $\gamma_x \cup g_x = \rt_x$.

The supporting lamination $| I |$ has measure zero in $\h^2$, since it is obtained from the measured lamination on a closed hyperbolic surface, which has measure zero. 
On a subinterval $J$ of $[0, A]$ where $\beta_b$ is smooth, $J \cap X$ is a finite set. 
Thus $X$ has measure zero in $[0,1]$. 
Therefore, letting $F_X, E_X, G_X$ be the respective subsets of $F, E, G$ corresponding to $X \times [0,1]$ in  $I \times [0,1]$, we have
\begin{eqnarray*}
\Area(F_X) =  \Area(E_X) =  \Area(G_X) = 0.
\end{eqnarray*}

The definition of $X$ implies that $X$ is a closed subset of $[0, A]$.
Then the complement of $X$ is the union of, at most,  countably many disjoint intervals $Y_k \,(k \in K)$.
Then $Y_k$ have open endpoints except at $0$ and $A$. 
For each $k \in K$, let  $0 \leq y_k < z_k \leq 1$, be the endpoints of $Y_k$. 
In addition let $F_k$ and $E_k$ be the subsurfaces of $F$ and $E$, respectively, corresponding to   $Y_k \times [0,1]$. 
 Let $G_k$ be the subsurface of  $G$ corresponding to $[\Phi(y_k), \Phi(z_k)] \times [0,1]$.

Then it suffices to show that, 
if $Y_k$ does not contain an endpoint of $[0, A]$, 
$$|\Area(F_k) - \Area(E_k)| \leq \Area(G_k),$$

and, if $Y_k$ contains an endpoint $p$ of $[0, A],$
 $$|\Area(F_k) - \Area(E_k)| \leq \Area(G_k) + \Area(\Delta_p).$$

(i) First we suppose that the interval $Y_k$ is an open interval $(y_k, z_k)$. 
 At least one of the endpoints  $y_k$ and $z_k$ must be of Type I , since $\beta_b$ is obtained by bending $b$.  
Therefore the endpoints of $Y_k$ are  either: (i - i)  both of Type I (Figure \ref{2-3no1}) or (i - ii)  of the different types, Type I and II (Figure \ref{2-3no2}). 

Suppose the case of (i - i). 
Then $\rt (F_k)$  is disjoint from $O$. 
Then $E_k$ is  naturally the union of $F_k$ and $G_k$ such that $F_k$ and $G_k$ have disjoint interiors in $E_k$. 
In particular $\Area (E_i) = \Area (F_i) + \Area(G_i)$. 

Suppose the case of (i - ii).
Then there is a point  $t \in (y_k, z_k )$ such that $g_t$ is ``tangent'' to $\beta_b$ at $\Phi(t)$ so that the restrictions of $g$ to $[\Phi(y_k), \Phi(t)] \times [0,1]$ and $[\Phi(t), \Phi(z_k)] \times [0,1]$ are immersions of  the opposite orientations (Figure \ref{2-3no2}). 
Let $G_i'$ be the component of $G_i \minus \{\Phi(t)\} \times [0,1]$ that contains $\{\Phi(y_k)\} \times [0,1]$ if $y_k$ is of Type I and $\{\Phi(z_k)\} \times [0,1]$ if $z_k$ is of Type I.  
Then $E_k$ is  the union of $F_k$ and  $G_k'$. 
In particular  $\Area(E_k) <  \Area(F_k) + \Area(G_k')  <  \Area(F_k) + \Area(G_k)$.

\begin{figure}[h]
\begin{overpic}[trim = 0mm 0mm 0mm 0mm, clip, width = 3in
]{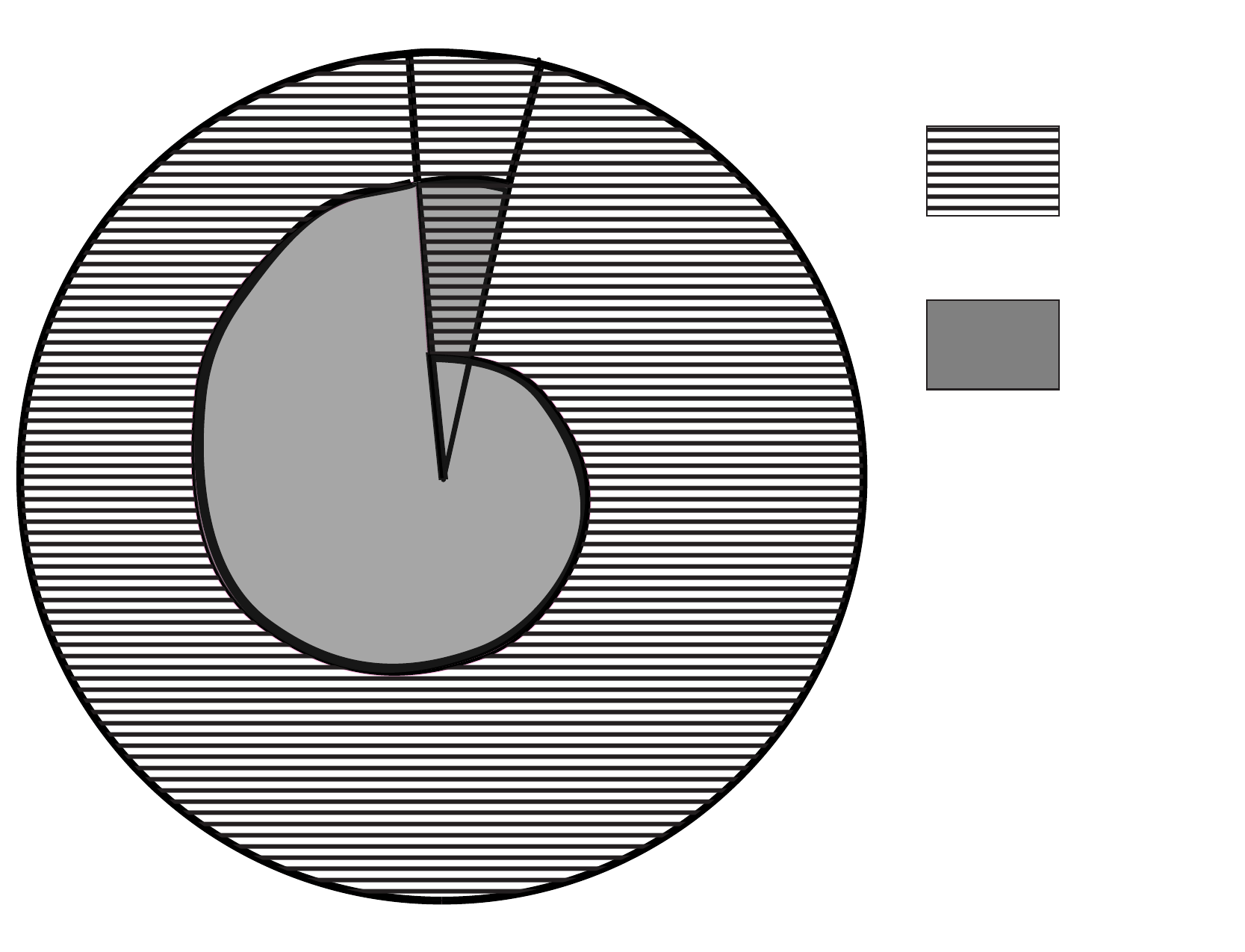} 
\put(87, 60){$F_k$}
\put(87, 46){$G_k$}
      \end{overpic}
\caption{$(y_k, z_k)$ with Type I ends}\label{2-3no1}
\end{figure}

\begin{figure}[h]
\begin{overpic}[trim = 0mm 0mm 0mm 0mm, clip,width = 2.2in
]{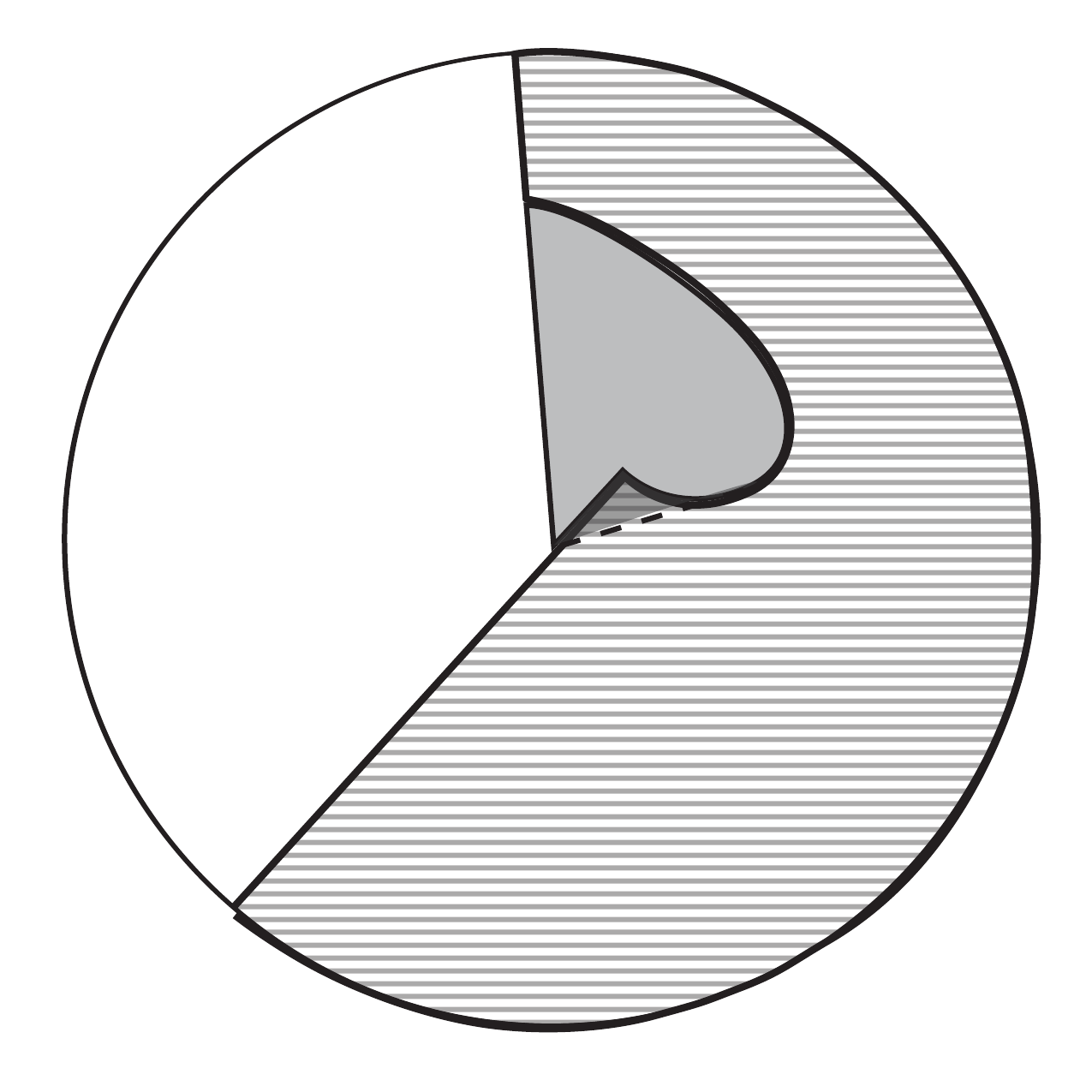}
\put(60, 27){$F_k$}
\put(53, 64){$G_k$}
\put(53, 45){$g_t$}
      \end{overpic}
\caption{$(y_k, z_k)$ with Type I and II ends}\label{2-3no2}
\end{figure}


(ii) Next suppose that exactly one of the endpoints of $Y_k$ is closed. 
Consider the case of  $Y_k  = (y_k, z_k]$ so that $z_k = A$.  (The case of $[y_k, z_k)$ is similar.) 
First suppose, in addition, that $y_k$ is of Type I. 
 Then  if $\ep > 0$ is  sufficiently small, then either $\rt_{y + \ep}$ is disjoint from all $\gam_{t}$ for $y_k \leq t  < y_k +\ep$  (Case I, Figure \ref{10-29-4}) or    $\gam_{y + \ep}$ is disjoint from $\rt_t$  for $y_k \leq t < y_k + \ep$ (Case II, Figure \ref{2-14no1}). 

In  Case I, we naturally have  $$E_k \cup \Delta_A = G_k \cup F_k,$$ so that $E_k$ and $\Delta_A$ have disjoint interiors and $G_k$ and $F_k$ have disjoint interiors.
Then 
$E_k \minus F_k \sub G_k$ and $F_k \minus E_k \sub \Delta_A$. 
Therefore
$$|\Area(E_k) - \Area(F_k)| < \Area(G_k) + \Area(\Delta_A).$$ 

In Case II , we naturally have $$E_k = F_k \cup G_k \cup \Delta_A,$$
so that $F_k, G_k, \Delta_A$ have disjoint interiors.
Then

$$0 < \Area (E_k) - \Area(F_k)  =  \Area( G_k) + \Area(\Delta_A).$$

\begin{figure}
\begin{overpic}[scale=.5
] {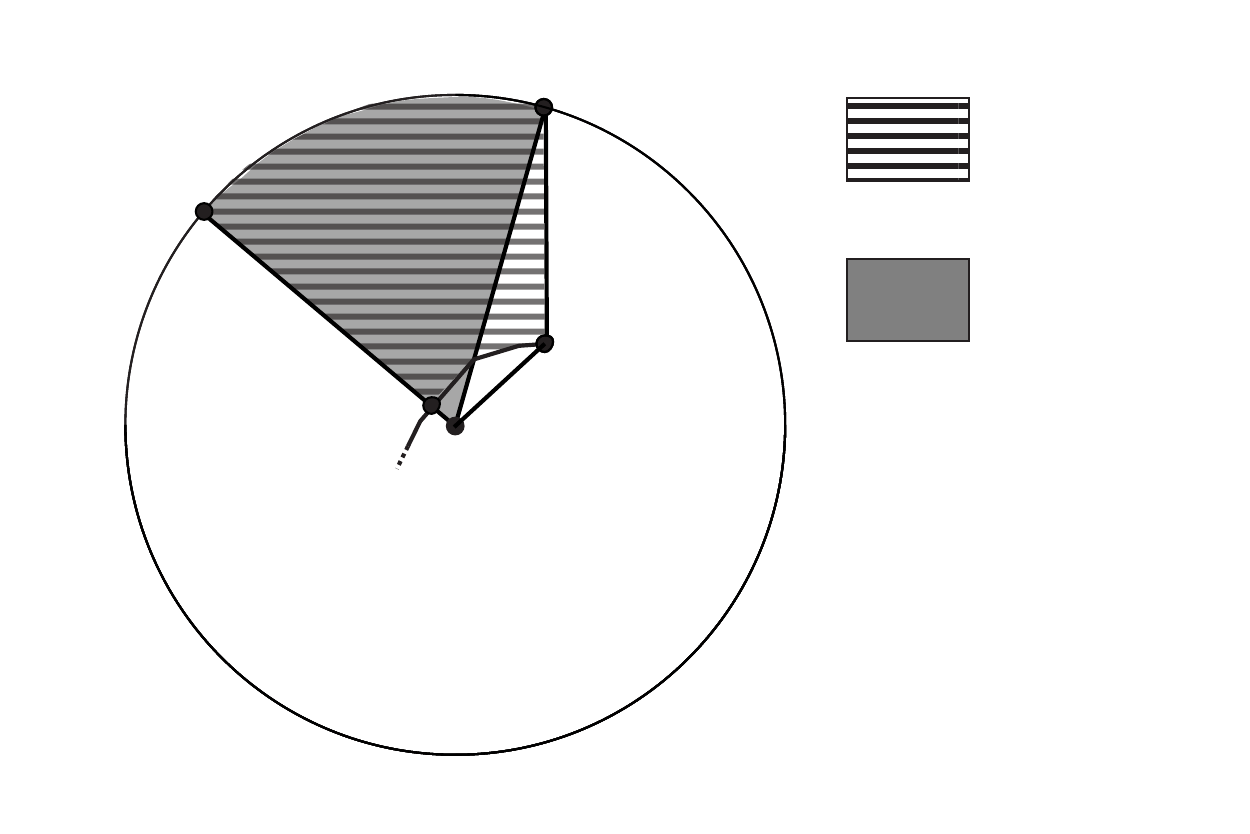}
      \put(37, 28){O}
      \put(15, 32){$\beta \cc \Phi(y_k)$} 
         \put(45, 40){$\beta \cc \Phi(A)$} 
       
          \put(9, 53){$f(y_k)$}  
  \put(43, 60){$f(A)$}  

\put(80, 54){$F_k$}
\put(80, 41){$E_k$}
      \end{overpic}
\caption{$(y_k, z_k]$ with Type I at $y_k$, Case I}\label{10-29-4}
\end{figure}

\begin{figure}
\begin{overpic}[scale=.5
] {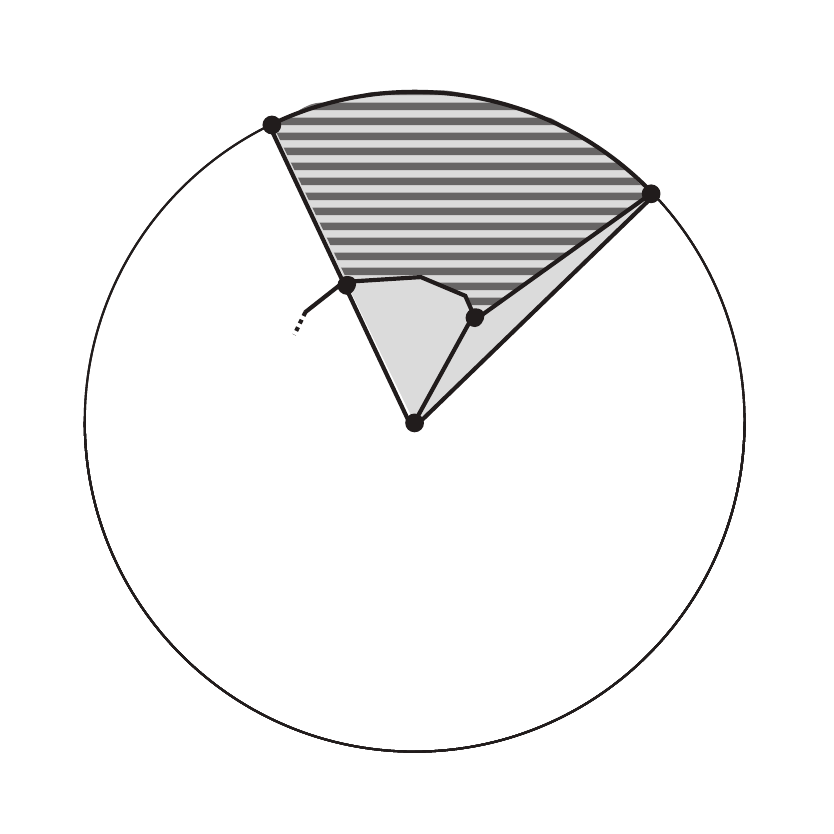}
      \put(48,42 ){O}
          \put(15, 65){$\beta \cc \Phi(y_k)$} 
         \put(63, 58){$\beta \cc \Phi(A)$} 
          \put(30, 90){$f(y_k)$}  
  \put(80, 80){$f(A)$}  \
      \end{overpic}
\caption{$( y_k, z_k]$ with Type I at $y_k$, Case II}\label{2-14no1}
\end{figure}

Next  suppose instead that $y_k$ is of Type II (Figure \ref{10-29-5-2}). 
Then $(E_k \sm F_k) \sqcup (F_k \minus E_k)$ is naturally embedded in $G_k \sqcup \Delta_k$.
Thus $$|\Area (E_k) - \Area(F_{n}) |  < \Area(G_k) + \Area(\Delta_k).$$ 

\begin{figure}
\begin{overpic}[scale=.5
] {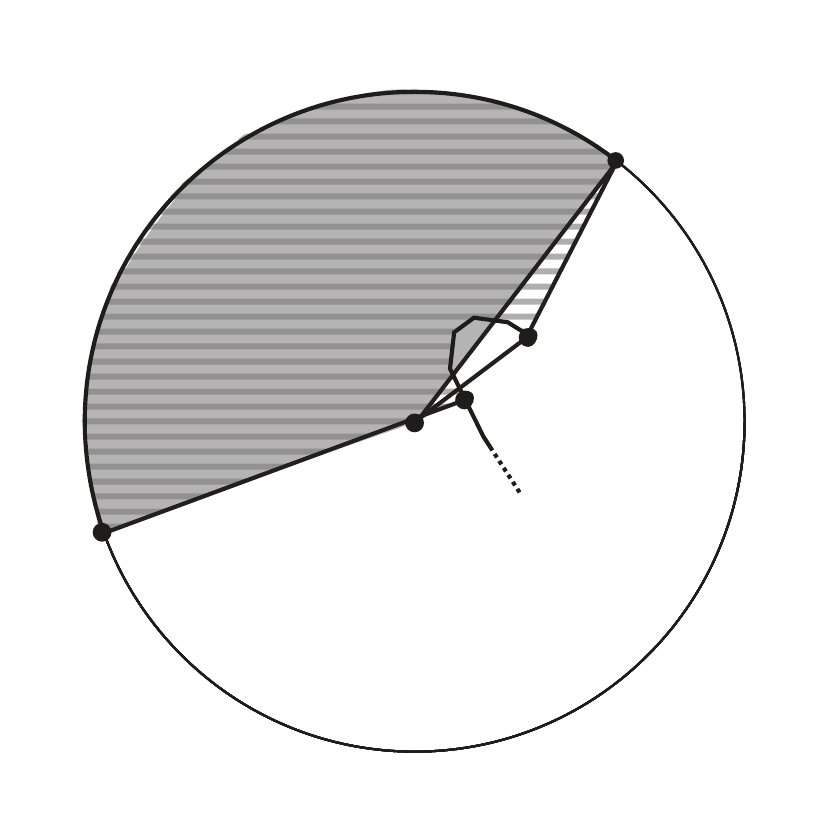}
    \put(48,42 ){O}
    \put(48,42 ){O}
         \put(60, 50){$\beta \cc \Phi(y_k)$} 
         \put(64, 58){$\beta \cc \Phi(z_k)$} 
          \put(0, 30){$f(y_k)$}  
  \put(72, 85){$f(z_k)$}  
 
      \end{overpic}
\caption{Type II at $y_k$}\label{10-29-5-2}

\end{figure}
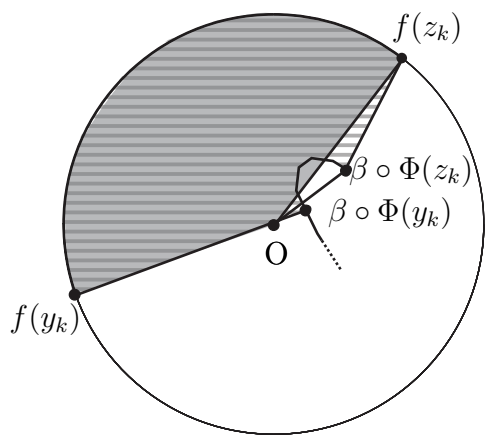

Lastly suppose that $Y_k$ is the closed interval. 
Then $Y_k$ must be  the entire interval $[0, A]$. 
Then a similar argument proves (\ref{120312}).
\Qed{8-23_2}

\section{Characterization of $\PP_\rho$ via Thurston coordinates}\Label{s:charactrization}
\subsection{Local characterization of $\PP_\rho$ in $\GL$.}

Let $C \cong (\tau, L)$ be a projective structure on $S$ with holonomy $\rho\cn \po(S) \to \psl$, and let $\kap\cn C \to \tau$ be the collapsing map.
Let $\LL$ be the canonical lamination on $C$, which descends to $L$ by $\kap$.
Let $C_0 \cong (\tau, L_0)$ be the corresponding projective structure with holonomy $\rho$ as in \S \ref{120612}, so that $C = \Gr_M(C_0)$, where $M$ is  the maximal weighted multiloop ``contained in $L$''  so that $M = L - L_0$. 
Similarly let $\kap_0\cn C_0 \to \tau$ be its collapsing map and  $\LL_0$ be the canonical lamination on $C_0$.
Then $C$ and $C_0$ correspond to the same $\rho$-equivariant pleated surface $\h^2 \to \h^3$.
 Given another projective structure $C' \cong (\tau', L') \in \PP_\rho$,   let $\kap'\cn C' \to \tau'$ be its collapsing map and  $\LL'$ be the canonical lamination on $C'$. 
Consider a shortest closed geodesic loop on $\tau$, and let $K > 0$ be one-fourth of its length (or any  positive number less than one-third of it). 

\begin{theorem}\Label{8-28-12}
For every $\ep > 0$, there is a $\dl > 0$, which depends only on $C$ and $\ep$,  such that, 
if  $C' \cong (\tau', L')$ satisfies $\angle_\tau ( L, L' ) < \del$, then 
there are a traintrack $\mathcal{T} = \{\RR_j\}$ on $S$ and marking homeomorphisms  $$\phi\cn S \to C, ~  \phi_0\cn S \to C_0, ~ \phi'\cn S \to C'$$  
 taking $\mathcal{T}$ to admissible traintracks (on $C, C_0, C'$) such that:
 \begin{itemize}
\item[(I)] \begin{itemize}
\item[\textbullet]  $\phi(\mathcal{T})$  carries $\LL$ on $C$, and it descends, by $\kap$, to an $(\ep, K)$-nearly straight traintrack $T$ on $\tau$ carrying $L$.
\item[\textbullet] $\phi_0(\mathcal{T})$ carries $\LL_0$ on $C_0$, and it descends, by $\kap_0$, also to   $T$  on $\tau$ (carrying $L_0$).
\item[\textbullet] $\phi'(\mathcal{T})$ carries $\LL'$ on $C'$, and it descends, by $\kap'$, to an $(\ep, K)$-nearly straight traintrack $T'$ on $\tau'$ carrying $L'$. 
\end{itemize}
\end{itemize}
By identifying $\mathcal{T}$  and its images by the homeomorphisms $\phi, \phi_0, \phi'$, we have:
\begin{itemize}
\item[(II)]  $C'$ is obtained by grafting $C_0$ along a weighted multiloop $M'$ carried by $\mathcal{T}$, and  moreover $M'$ is $\ep$-close $\LL' - \LL_0$  on $\mathcal{T}$. 
 \item[(III)]    We can graft $C$ and $C'$ along some weighted multiloops carried by $\mathcal{T}$, respectively, to a common projective structure.
Indeed, if there are weighted multiloops $\hat{M}$ and $\hat{M'}$ carried by $\mathcal{T}$ such that $\hat{M} + M = \hat{M}' + M'$ on $\mathcal{T}$, then 
\end{itemize} 
\begin{eqnarray*}
\Gr_{\hat{M}} (C) = \Gr_{\hat{M}'}(C').
\end{eqnarray*}
\end{theorem}

\begin{remark}

In (II) and (III), we make the multiloops transversal to the circular foliations of $\mathcal{T}$'s on $C, C_0, C'$, so that they are admissible (by Lemma \ref{admissible}).  

In (I), the correspondences between the traintracks on the projective surfaces and the hyperbolic surfaces are up to  small perturbations of the vertical edges (as in Proposition \ref{6-17-12}).

The traintracks $T, T'$  are obtained by Proposition \ref{7-7-12no1}. 
Thus we can in addition assume that there is an $\ep$-rough isometry that takes $T$ to $T'$ that preserves the marking.

In Theorem B, $\mL' - \mL_0$ is replaced by $L' - L_0$ as they coincide as topological measured laminations. 
\end{remark}

The following proposition immediately yields  Theorem \ref{8-28-12}  (I), and in addition it will be promoted to  Theorem \ref{8-28-12} (II) and  (III). \note{eng: promote}
\begin{proposition}\Label{120912}
For every $\ep > 0$, there is a $\del > 0$ such that, 
if $C' \cong (\tau', L')$ is a projective structure with holonomy $\rho$ and  $\angle_{\tau}(L, L') < \del$, 
then
there are a traintrack $\mathcal{T} = \{\RR_j\} _{j = 1}^n$ on $S$ and  marking homeomorphisms 
 $$\phi\cn S \to C, ~  \phi_0\cn S \to C_0, ~ \phi'\cn S \to C'$$ taking $\mathcal{T}$ to admissible traintracks (on $C, C_0, C')$, such that: 
\begin{itemize}
\item[(I)]
\begin{itemize}
\item[\textbullet]   $\phi(\mathcal{T})$  descends, by $\kap$, to an $(\ep, K)$-nearly straight traintrack on $\tau$ carrying $\mL$.
\item[\textbullet]  $\phi_0(\mathcal{T})$ descends, by $\kap_0$,  to the same $(\ep, K)$-nearly straight traintrack on $\tau$ (which carries $\mL_0$).
\item[\textbullet]  $\phi'(\mathcal{T})$ descends,  by $\kap'$, to an $(\ep, K)$-nearly straight traintrack  on $\tau'$ carrying $\mL'$.
\end{itemize}
\item[(II)]
The developing maps of $C, C_0, C'$ induce isomorphisms between  $C \sm \phi(\mathcal{T}), C_0 \sm \phi_0(\mathcal{T}), C' \sm \phi'(\mathcal{T})$  as projective surfaces (see Definition \ref{isomorphic}).
\end{itemize}
Thus we can assume that  $\phi, \phi_0, \phi'$ induce those isomorphisms,  and  \note{english: structure}
\begin{itemize}
\item[(III)]  
\begin{itemize}
\item[i.] For every branch $\RR_j$ of $\mathcal{T}$ and its lift $\td{\RR}_j$ to $\td{S}$, the corresponding branches $\td{\phi}'(\td{\RR}_j)$, $\td{\phi}(\td{\RR}_j)$, $ \td{\phi}_0(\td{\RR}_j)$ are supported on a common round cylinder on $\rs$, where $\td{\phi}\cn \td{S} \to \td{C}$ and $\td{\phi}'\cn \td{S} \to \td{C}'$, $\td{\phi}_0\cn \td{S} \to \td{C}_0$  are the lifts of $\phi, \phi', \phi_0$. 
\item[ii.] By identifying $\mathcal{T}$ and its images under $\phi, \phi_0, \phi',$ then $\phi'(\RR_j) \st C'$ is obtained by grafting $\phi_0(\RR_j) \st C_0$ along a multiarc $M'_j$ that is $\ep$-close to $\mL' - \mL_0$ on $\RR_j$, and  $\phi(\RR_j) \st C$ is obtained by grafting  $\phi_0(\RR_j) \st C_0$  along the multiarc $M_j$ corresponding  exactly to $\mL - \mL_0$ on $\RR_j$.
 \end{itemize}
\end{itemize}
\end{proposition}

\begin{proof}
Let $C'_i \cong (\tau'_i, L'_i)$ be a sequence in $\PP_\rho$ such that $\angle_\tau(L, L'_i) \to 0$.
Then it suffices to show the proposition for $C'_i$ with sufficiently large $i$.
By Theorem \ref{12-19},  $\tau'_i \to \tau$ as $i \to \In$.  

Without loss of generality, we can assume that both $C$ and $C_0$ appear in the sequence  $\{C'_i\}$  infinitely many times. 
Since $\GL(S)$ is compact,  we can in addition assume that $|L'_i|$ converges to a geodesic lamination $\lam_\In$ in the Hausdorff topology as $i \to \In$, by taking a subsequence if necessary. 
For every $\ep > 0$,  by applying Proposition \ref{6-28-12no1} to $\rho, \tau, \lam_\In$ and Lemma \ref{8-11_1}, we obtain the proposition. 
Note that, since $L_0$ has no leaves of weight at least $2\pi$, for every $\ep > 0$, if $\del > 0$ is sufficiently small,  the weight of $\mL'$  is more than the weight of $\mL_0$ minus $\ep$ on each branch of $\mT$. 
\end{proof}

\proofof{Theorem \ref{8-28-12}}
For $\ep > 0$,  let $\del > 0$ be the constant obtained by applying  Proposition \ref{120912}.
Then,  for every $C' = (f', \rho) \in \PP_\rho$ with $\angle_\tau(L, L') < \del$,   Proposition \ref{120912} yields  a topological traintrack $\mathcal{T} = \{\RR_j\} _{j = 1}^n$ on $S$ and   marking-preserving homeomorphisms $\phi\cn S \to C$, $\phi_0\cn S \to C_0$ and $\phi'\cn S \to C'$.
Thus we have (I) by Proposition \ref{120912} I. \note{eng: $(I)$ or $I$.}
In particular  $\kap$ and $\kap_0$ take $\phi(\mathcal{T})$ and $\phi_0(\mathcal{T})$, respectively, to the same $(\ep, K)$-nearly straight traintrack $T$ on $\tau$ carrying both $L$ and (the geodesic representative of) $L'$ on $\tau$.
Recall that $L_0$ has no closed leaf of weight at least $2\pi$.

First  we prove (II).
We have a natural decomposition of $S$ by the traintrack  $\mathcal{T}$:
$$S = (S \minus |\mathcal{T}|) \cup \mathcal{T}  = (S \minus |\mathcal{T}| ) \cup (\cup_{j= 1}^{n} \RR_j).$$
Then, this decomposition of $S$ descends to  decompositions of  $C$ and $C'$ via the homeomorphisms $\phi$ and $\phi'$, respectively: 
$$C' = (C' \sm \phi'(\mathcal{T})) \cup \phi'(\mathcal{T})  = (C' \sm \phi'(\mathcal{T})) \cup (\cup_j \phi'(\RR_j))$$
$$C_0 = (C_0 \sm \phi_0(\mathcal{T})) \cup \phi_0(\mathcal{T})  = (C_0 \sm \phi_0(\mathcal{T})) \cup (\cup_j \phi_0(\RR_j))~.$$

Then by Proposition \ref{120912} II, $\phi' \cc \phi\iv$ yields an isomorphism from $(C_0\sm \phi(\mathcal{T}))$  to  $(C' \sm \phi'(\mathcal{T}))$ compatible with the developing maps $f$ and $f'$.

In addition, 
 by  Proposition \ref{120912} III - i,  for each $j = 1,2, \dt, n$, the corresponding branches $\phi'(\RR_j)$ and $\phi_0(\RR_j)$ are supported on a common round cylinder. 
 Since $L_0$ has no closed leaf of weight at least $2\pi$,  by   Proposition \ref{120912} III- ii,   we have $\phi'(\RR_j) = \Gr_{M'_j}(\phi_0(\RR_j))$  for some multiarc $M'_j$ that is $\ep$-close to $\mu'(\phi'(\RR_j)) \minus \mu_0(\phi_0(\RR_j))$, where $\mu'$ and $\mu_0$ are the transversal measures of $\mL'$ and  $\mL_0$, respectively.  
Let $\kp_0\cn C_0\to \tau$ and $\kp'\cn C' \to \tau'$ be the collapsing maps. 
Then, by Proposition \ref{120912} I, $\kp_0$ takes the traintrack $\phi_0(\mathcal{T})$ on $C_0$ to an $(\ep, K)$-nearly straight traintrack on $\tau$ carrying $L_0$, and 
$\kp'$ takes the traintrack $\phi'(\mathcal{T})$ on $C'$ to an $(\ep, K)$-nearly straight traintrack on $\tau'$ carrying $L'$.

Since $\phi' (\mathcal{T})$ carries $\mL'$, the $n$-tuple $\{\mu'(\phi'(\RR_j)) \}_{j = 1}^n$  satisfies the switch conditions of the traintrack  $\phi'(\mathcal{T}) \cong \mathcal{T}$.
Similarly,  since $\phi_0(\mathcal{T})$ carries $\mL$, 
the $n$-tuple $\{\mu(\phi_i (\RR_j)) \}_{j = 1}^n$ satisfies the switch conditions of the traintrack $\phi(\mathcal{T}) \cong \mathcal{T}$. 
Thus the $n$-tuple of their differences, $\{ \mu'(\kp' \cc \phi'(\RR_j)) - \mu(\kp_0\cc \phi(\RR_j)) \}_{j =1}^n$, satisfies the switch conditions as well.
Therefore,  the $n$-tuple of the numbers of the arcs of $M'_j$ $(j =1,2,\dt,n)$  also satisfies the switch conditions.
Thus, after isotoping $M'_j$ on $\phi(\RR_j)$ through admissible multiarcs so that their endpoints match up on the vertical edges,  the union $\cup_j M'_j =: M'$ is a multiloop carried by the traintrack $\phi(\mathcal{T})$. 
Note that   the isomorphisms $\phi'(\RR_j) = \Gr_{M'_j}(\phi_0(\RR_j)),$ $j = 1, \dots, n,$ as projective surfaces remain true under such an isotopy (Lemma \ref{8-11_1}).  
Since $\phi_0(\mathcal{T})$ carries $\mL$ and  $\phi'(\mathcal{T})$ carries $\mL'$,  we can regard  $\mL' - \mL$ as  a measured lamination on $S$ carried by $\mathcal{T}$.
Since $\mu'(\phi'(\RR_j)) - \mu_0(\phi_0(\RR_j))$, therefore $\mL' - \mL$ is $\ep$-close to $M'$ on $\mT$.
 
Next we compare the traintracks $\phi_0(\mathcal{T}) \st C_0$ and $\phi'(\mathcal{T}) \st C'$ as projective structures on $\mathcal{T}$  (compare with \cite{Baba12}).
Let $\RR_i$ and $\RR_j$ be  branches of $\mathcal{T}$ that are adjacent along a vertical edge $e$.
 Then let $m_i$ and $m_j$ be arcs of $M'_i$ and $M'_j$, respectively, that share an endpoint on $e$, so that $m_i \cup m_j$ is a simple arc on  $\RR_i \cup \RR_j$, which is obtained by naturally gluing $\RR_i$ and $\RR_j$ along $e$.
Since $\RR_j$ and $\RR_i$ are supported on a round cylinder, the projective structure inserted by the grafting of  $\RR_i \cup \RR_j$ along  $m_i \cup m_j$ is exactly the union of projective structures inserted by the graftings  of  $\RR_i$ along $m_j$ and  of  $\RR_j$ along $m_j$. 
Since this holds for all adjacent arcs, we have
 $$ \phi'(\mathcal{T}) =  \cup_j \phi'(\RR_j)  =  \cup_j \Gr_{M'_j}(\phi_0(\RR_j)) =  \Gr_{M'}(\phi_0(\mathcal{T})).$$

 Hence
 \begin{eqnarray*}
C' &=&  (C' \sm \phi'(\mathcal{T})) \,\cup \, \phi'(\mathcal{T}) \\
 &=&(C_0\sm \phi_0(\mathcal{T})) \cup ( \Gr_{M'}(\phi_0(\mathcal{T}))) = \Gr_{M'}(C_0). \label{121112}
\end{eqnarray*}

Next we prove (III),  that is,  $\Gr_{\hat{M}} (C) = \Gr_{\hat{M}'}(C')$.
Since  $C = \Gr_M(C_0)$, we have
$$C =  (C \sm \phi(\mathcal{T})) \,\cup \, \phi(\mathcal{T}) = (C_0\sm \phi_0(\mathcal{T})) \cup \Gr_{M}(\phi_0(\mathcal{T})) = \Gr_{M}(C_0).$$
and 
$$ C \sm \phi(\mathcal{T})  = C_0\sm \phi_0(\mathcal{T}). $$

The traintrack  $\phi_0(\mathcal{T}) = \{\phi_0(\RR_j) \} (\cong \mathcal{T})$ carries $\LL_0$.
Thus the grafting of $C_0$ along $M$ naturally decomposes into grafting of all branches: 
$$\cup_j  \phi(\RR_j) = \phi(\mT) =   \Gr_{M}(\phi_0(\mathcal{T})) = \cup_j \Gr_{M|\RR_j}(\phi_0(\RR_j)).$$
In particular  $ \phi(\RR_j) =  \Gr_{M|\RR_j}(\phi_0(\RR_j)).$

Since $\hat{M}$ is also carried by $\phi(\mathcal{T}) \cong \mathcal{T}$ and each branch $\phi(\RR_j)$ of $\phi(\mathcal{T})$on $C$ is supported on a round cylinder, $\hat{M}$ is admissible on $C$ (Lemma \ref{admissible})
Then $\Gr_{\hat{M}}(C)$ is well defined, and $\Gr_{\hat{M}}(C) = \Gr_{\hat{M}} \cc \Gr_M(C_0)$.

Recall that the homeomorphism $\phi\cn S \to C$ represents the marking of $C$.
Then  there is a marking homeomorphism $\hat{\phi}\cn S \to \Gr_{\hat{M}}(C)$ so that 
$\hat{\phi} \cc \phi^{-1}$ induces an isomorphism from $C \minus \phi(|\mathcal{T}|)$ to $\Gr_{\hat{M}}(C) \minus \hat{\phi}(|\mathcal{T}|)$ and 
 $\Gr_{\hat{M}_j}\phi(\RR_j) = \hat{\phi}(\RR_j)$ for all $j = 1, \dots, n$, where $\hat{M_j} = \hat{M} | \phi(\RR_j)$. 
  Note that $(M + \hat{M}) | \RR_j$ is the multiarc on $\RR_j$ such that the number of its arcs is the sum of the number of the arcs of $M |\phi_0(\RR_j)$ and $\hat{M} | \phi(\RR_j)$.
Then $\hat{\phi}(\RR_j)$ is obtained by grafting $\phi_0(\RR_j)$ along $(M + \hat{M}) | \RR_j$, and
$$\Gr_{\hat{M}}(C) =\cup_j \Gr_{(M + \hat{M})|\RR_j}  \phi_0(\RR_j)  \cup   (C_0 \minus \phi_0(\mT)).$$
Similarly,  since $\Gr_{\hat{M'}}(C') = \Gr_{\hat{M}'} \cc \Gr_{M'}(C_0)$,  the traintrack $\mathcal{T}$ yields a decomposition of $\Gr_{\hat{M'}}(C')$.
$$\Gr_{\hat{M}'}(C') = \cup_j \Gr_{(M' + \hat{M'})|\RR_j} \phi_0(\RR_j)  \cup (C_0 \minus \phi_0(\mT)).$$
Since $M + \hat{M} = M' + \hat{M}'$ on $\mathcal{T}$, therefore   $\Gr_{\hat{M}} (C) = \Gr_{\hat{M}'}(C')$.
\Qed{8-28-12}

\begin{theorem}\Label{062213}
For every $\ep > 0$ and every compact set $X$ in the moduli space of $S$, there is $\del > 0$, such that 
the assertions of  Theorem \ref{8-28-12} hold true for every projective structure $C \cong (\tau, L)$  on $S$ with unmarked $\tau$ in $X$.   
\end{theorem}

\begin{proof}
We can observe that, for the proof of Theorem \ref{8-28-12},  the assumption $\angle_\tau(L, L')  < \del$ is only used to guarantee that, there is an $(\ep, K)$-nearly straight traintrack  $T$ on $\tau$ carrying both $L$ and $L'$, where $K = K_\tau$ is one-fourth of the length of the shortest closed geodesic loop on $\tau$. 
(All traintracks in Theorem \ref{8-28-12} are constructed from $T$.)

Since $X$ is compact, now let $K$ be the infimum of $K_\tau$ over all $\tau \in \TT$ that project to $X$.  
Then it suffices to show that, for every $\ep > 0$, there is a $\del > 0$,  such that, if $C  \cong (\tau, L)$ is a projective structure on $S$ with holonomy $\rho$ and with unmarked $\tau$ in $X$ and   $C'  \cong (\tau', L')$ is  another projective structure  with the same holonomy $\rho$ with $\angle_\tau(L, L') < \del$, then there is an $(\ep, K)$-nearly straight traintrack $T$ on $\tau$ carrying both $L$ and $L'$.  
This claim  is equivalent to 
\begin{claim}
 If there are two sequences, $C_i \cong (\tau_i, L_i)$ and $C'_i \cong (\tau_i', L_i')$, of projective structures on $S$  with all unmarked $\tau_i$ in $X$, such that $\Hol(C_i) = \Hol(C_i')$ and  $\angle_{\tau_i}(L, L_i) <  \del$ for all $i$, then there is an $(\ep, K)$-nearly straight traintrack $T_i$  on $\tau_i$ carrying both $L_i$ and $L_i'$ for sufficiently large $i$. 
\end{claim}
For each $i$, without loss of generality, we can change the markings of $C_i$ and $C_i'$ simultaneously by a single homeomorphism of $S$.  
Thus, by the compactness of $\GL$ and $X$, we can in addition assume that $\tau_i$ converges to $\tau \in \TT$ and that $|L_i|$ and $|L_i'|$ converge to $\lam$ and $\lam'$, respectively, in $\GL$.
Then since $\angle_{\tau_i}(L_i, L_i') < \del$, we have $\angle_\tau(\lam, \lam') \leq \del$. 
Thus if $\del > 0$ is sufficiently small,  for every $\ep > 0$, there is an $(\ep, K)$-nearly straight traintrack $T$ on $\tau$ carrying both geodesic laminations $\lam$ and $\lam'$.
By the convergence of $\tau_i$, $|L_i|$, $|L_i|$,  for sufficiently large $i$, there is an $(\ep, K)$-nearly straight  traintrack  $T_i$ on $\tau_i$ that carries both $L_i$ and  $L_i'$. 
\end{proof}

\subsection{Alternative proof of Ito's Theorem.}

Suppose that $\rho \cn \pi_1(S) \to \pslr$ is a fuchsian representation. 
Let $\tau$ be the  marked hyperbolic structure $\h^2 / \Im(\rho)$ corresponding to $\rho$.
Then, As discussed in \S \ref{S:intro}, given arbitrary $C, C' \in \PP_\rho$, we can express them in Thurston coordinates as  $C \cong (\tau, M)$ and $C' \cong (\tau, M')$ with unique multiloops $M$ and $M'$.
Then, Theorem \ref{ito} also follows from Theorem \ref{062213}. 
\begin{theorem}\Label{cor:ito}
 $C$ and $C'$ can be transformed to a common projective structure by grafting $C$ along  $M'$ and $C'$ along  $M$,
 $$\Gr_{M'}(C) = \Gr_{M}(C').$$
\end{theorem}

\proof
recall that $\ML_\N$ denotes the set of weighted multiloops. 
Since $\rho$ is fuchsian,  $\PP_\rho$ is canonically identified with $\ML_\N$ by Thurston coordinates (see Theorem \ref{10-30-1}). 
Thus, for a different fuchsian representation $\eta$, there $\PP_{\rho}$ and $\PP_{\eta}$ are naturally identified via $\ML_\N$. 
In fact, there is a quasiconformal map $\Theta = \Theta_\eta \col \rs \to \rs$ that conjugates $\rho$ to $\eta$ and realizes the identification $\PP_\rho \cong \PP_\eta$ by postcomposing the developing maps of structures in $\PP_\rho$ with $\Theta$.
Then, if  $M$ is an admissible multiloop on $C \in \PP_\rho$, then $\Theta$ takes $M$ to an admissible loop $\Theta(M)$ on $\Theta(C)$. 
Then we have 
$$\Theta (\Gr_M (C)) = \Gr_{\Theta(M)} (\Theta(C)).$$
Thus it suffices to show $$\Gr_{\Theta_\eta(M')}(\Theta_\eta(C)) = \Gr_{\Theta_\eta(M)}(\Theta_\eta(C'))$$ for some fuchsian representation $\eta$.

Let $D_M$ be the (simultaneous) Dehn twist of $S$ along all loops of $M$.  
Then $D^k_M (\tau)$ denote the $k$-iterates of $D_M^k$ on $\tau$ for $k \in \Z_{> 0}$.
Then  $\angle_{D_M^k(\tau)}(M, M') \to 0$ as $k \to \infi$. 
Note that $D_M$ acts trivially on the moduli space, and in particular it preserves unmarked $\tau$.
Let $\eta_k\col \pi_1(S) \to \pslr$ be the fuchsian representation realizing $D_M^k(\tau)$. 
 Then, by Theorem \ref{062213}, if $k$ is sufficiently large, then $$\Gr_{\Theta_{\eta_k}(M')}(\Theta_{\eta_k}(C)) = \Gr_{\Theta_{\eta_k}(M)}(\Theta_{\eta_k}(C')).$$
\endproof

\subsection{Local characterization of $\PP_\rho$ in $\PML$}
\begin{theorem}\Label{ThmA} 
Let $C \cong (\tau, L)$ be a projective structure on $S$ with (arbitrary) holonomy $\rho\col \pi_1(S) \to \PSL$. 
For every $\ep > 0$,  there is a neighborhood $U$ of the projective class $[L]$ in $\PML$, such that, if another projective structure $C' \cong (\tau', L')$ with holonomy $\rho$ satisfies $[L'] \in U$, 
then there are a traintrack $\mathcal{T}$ on $S$ and marking homeomorphisms $\phi\cn S \to C$ and $\phi' \cn S \to C'$  such that, by identifying $\mathcal{T}$ and its images under those homeomorphisms:  
 \begin{itemize}
 \item  $\mathcal{T}$ is an admissible traintrack on both $C$ and $C'$, 
 \item the traintrack $\mathcal{T}$ carries $\LL$ on $C$, and the collapsing map $\kap\cn C \to \tau$ descends $\mathcal{T}$ to an $(\ep, K)$-nearly straight traintrack on $\tau$, 
 \item  the traintrack $\mathcal{T}$ carries $\LL'$ on $C'$, and the collapsing map $\kap'\cn C \to \tau'$ descends $\mathcal{T}$ to  an $(\ep, K)$-nearly straight traintrack on $\tau'$,
\end{itemize}
 and indeed we have either
 \begin{itemize}
 \item[(i)]  $Gr_M(C) = C'$ for  a weighted multiloop  $M$ on $C$ carried by $\mathcal{T}$ that  is $\ep$-close to  $\LL' - \LL$ calculated on $\mathcal{T}$ or
 \item[(ii)]   $C = Gr_{M'}(C')$ for  a weighted multiloop  $M' = \LL - \LL'$ on $\mathcal{T}$.
\end{itemize}
\end{theorem}

\begin{remark}
Similarly to Theorem \ref{8-28-12}, to be precise,  an $\ep$-small perturbation of the vertical edges  is needed to make  the image of traintracks under the collapsing maps $(\ep, K)$-nearly straight (as in Proposition \ref{6-17-12}). 
\end{remark}

\begin{proof} 
If $L = \emptyset$ or  $L' = \emptyset$, then $C$ or $C'$ is accordingly a hyperbolic structure.  
In particular $\rho$ is fuchsian.

Then, by Theorem \ref{10-30-1}, if $L = \emptyset$, then (i) holds and if $L' = \emptyset$ then (ii)  holds. 
Thus we can suppose that $L, L' \neq \emptyset$.
Since the holonomy map $\Hol\cn\PP \to \chi$ is a local homeomorphism, 
$\PP_\rho$ is a discrete subset of $\PP$.
For a neighborhood $U$ of $[L]$ in $\PML$, 
let $\PP(\rho, U)$ be the set of projective structures with holonomy $\rho$ such that,  in Thurston coordinates, their projective measured laminations are in $U$. 
By Theorem \ref{12-19}, for every neighborhood $V$ of $\tau$ in $\TT$, if $U$ is sufficiently small, then, for   $C' \cong (\tau',  L')$  in $\PP(\rho, U)$, we have   $\tau' \in V$.
Thus if two projective structures in $\PP(\rho, U)$ share a measured lamination in Thurston coordinates, then they must coincide.

For every $\ep > 0$, if $U$ is sufficiently small, then, for every $C' \cong (\tau', L')$ in $\PP(\rho, U)$ with  $[L'] \in U$, then  
$\angle_\tau(L , L') < \ep$.
Thus, as in Theorem \ref{8-28-12},  we can decompose $C$ and $C'$ by a traintrack  $\mathcal{T} = \{\RR_j\} _j$ on $S$ given by Proposition \ref{120912}.

As in the proof of Theorem \ref{8-28-12}. 
Let $\phi\col S \to C$ and $\phi'\col S \to C'$ be the marking homeomorphisms obtained by Proposition \ref{120912}, so that $\phi(\mathcal{T})$ and $\phi'(\mathcal{T})$ are corresponding admissible traintracks on $C$ and $C'$, respectively. 
Let $T$ be the $(\ep, K)$-nearly straight traintrack on $\tau$ carrying $L$ and $L'$ such that $\phi(\mathcal{T})$ descends by $\kap$. 
Let $T'$ be the $(\ep, K)$-nearly straight traintrack on $\tau'$ such that $\phi'(\mathcal{T})$ maps to by $\kap'$.

Then, for every $\ep > 0$, if $U$ is sufficiently small then, there is a constant $c > 0$ such that,  
the weight ratios $\frac{\mu'(\phi' (\RR_j))}{\mu(\phi (\RR_j))}$ are $\ep$-close to $c$. 
Since $\PP_\rho$ is a discrete subset of $\PP$, unless $C = C'$, we can in addition assume that  either  (Case One) $c > 1$, or (Case Two) $0 < c < 1$ and the ratios are exactly $c$ for all $j$. 

In Case One,  $\mu'(\phi' (\RR_j)) - \mu(\phi (\RR_j))$ is $\ep$-close to a positive multiple of $2\pi$ for each $j$. 
Thus, similarly to the proof of  Theorem \ref{8-28-12} (II), we can show that $C' = Gr_M(C)$ and $M$ is $\ep$-close to  $\LL' - \LL$ calculated on $\mathcal{T}$.

In Case Two, we have $c [L] = [L']$ with $0 < c < 1$.
 Since the ratio $c$ is independent on $j$,   we see that $L$ and $L'$ must be multiloops: Otherwise, letting $F$ be a subsurface of $S$ such that $F \cap L$ is a minimal irrational lamination, the holonomy of $C$ must be different from that of $C'$ on $\pi_1(F)$.
 Then $L - L'$ must be a weighted multiloop, whose weights are $2\pi$-multiples,  since otherwise the holonomy  $C$ must be different from that of $C'$.
 Therefore, letting $M'$ be a multiloop $\LL - \LL'$ on $\phi'(\mathcal{T})$, we have
  $C = \Gr_{M'}(C')$. 
\end{proof}

\bibliographystyle{plain}
\bibliography{../../Reference}

\end{document}